\documentclass[10pt]{article}
\usepackage{amsmath,amssymb,amsfonts,amsbsy,amsthm}
\usepackage{graphicx,psfrag,epstopdf}
\usepackage{verbatim}
\usepackage{algorithm,algorithmic}
\graphicspath{{./pics/}}
\usepackage{latexsym}
\usepackage{color}
\usepackage{url}

\newcommand{\E}{\mathbb{E}}

\usepackage{epstopdf}

\newtheorem{theorem}{Theorem}[section]
\newtheorem{remark}{Remark}[section]

\newtheorem{lemma}{Lemma}[section]
\newtheorem{proposition}[theorem]{Proposition}

\newtheorem{exam}{Example}[section]
\numberwithin{equation}{section}

\allowdisplaybreaks

\title{Preasymptotic Convergence of Randomized Kaczmarz Method}
\author{Yuling Jiao\thanks{School of Statistics and Mathematics, Zhongnan University of Economics and Law, Wuhan, 430063, P.R. China. (yulingjiaomath@whu.edu.cn)}\quad\and
Bangti Jin\thanks{Department of Computer Science, University College London, Gower Street, London WC1E 6BT, UK. (bangti.jin@gmail.com, b.jin@ucl.ac.uk)}
\quad\and Xiliang Lu\thanks{Corresponding author. School of Mathematics and Statistics, and
Hubei Key Laboratory of Computational Science, Wuhan University, Wuhan 430072, P.R. China. (xllv.math@whu.edu.cn)}
}
\begin{document}
\maketitle
\begin{abstract}
Kaczmarz method is one popular iterative method for solving inverse problems, especially in computed tomography. Recently,
it was established that a randomized version of the method enjoys an exponential convergence for well-posed problems, and
the convergence rate is determined by a variant of the condition number. In this work, we analyze the preasymptotic
convergence behavior of the randomized Kaczmarz method, and show that the low-frequency error (with respect to the right
singular vectors) decays faster during first iterations than the high-frequency error. Under the assumption that the
initial error is smooth (e.g., sourcewise representation), the results allow explaining the fast empirical convergence behavior,
thereby shedding new insights into the excellent performance of the randomized Kaczmarz method in practice. Further, we
propose a simple strategy to stabilize the asymptotic convergence of the iteration by means of variance reduction. We
provide extensive numerical experiments to confirm the analysis and to elucidate the behavior of the algorithms.
\end{abstract}

\section{Introduction}

Kaczmarz methood \cite{Kaczmarz:1937}, named after Polish mathematician Stefan Kaczmarz, is one  popular
iterative method for solving linear systems. It is a special form of the general alternating projection method.
In the computed tomography (CT) community, it was rediscovered in 1970 by Gordon, Bender and Herman \cite{Gordon:1970},
under the name algebraic reconstruction techniques. It was implemented in the very first  medical CT scanner, and since
then it has been widely employed in CT reconstructions \cite{HermanLentLutz:1978,HermanMeyer:1993,Natterer:1986}.

The convergence of Kaczmarz method for {consistent} linear systems is not hard to
show. However, the theoretically very important issue of convergence rates of Kaczmarz method (or
the alternating projection method for linear subspaces) is very challenging. There are several known convergence rates
results, all relying on (spectral) quantities of the matrix $A$ that are difficult to compute or verify in practice (see
\cite{Galantai:2005} and the references therein). This challenge is well reflected by the fact that the convergence rate of
the method depends strongly on the ordering of the equations.

It was numerically discovered several times independently in the literature that using the rows of the matrix $A$ in Kaczmarz
method in a random order, called randomized Kaczmarz method (RKM) below, rather than the given order, can often
substantially improve the convergence \cite{HermanMeyer:1993,Natterer:1986}. Thus RKM is quite appealing for practical
applications. However, the convergence rate analysis was given only very recently. In an
 influential paper \cite{StrohmerVershynin:2009}, in 2009, Strohmer and Vershynin established the exponential
convergence of  RKM for consistent linear systems, and the convergence rate depends on (a
variant of) the condition number. This result was then extended and refined in various directions \cite{Needell:2010,
EldarNeedell:2011,MaNeedellRamdas:2015,AgaskarWangLu:2014,WangAgaskarLu:2015,GowerRichtarik:2015,
SchopferLorenz:2016}, including inconsistent or underdetermined linear systems. Recently, Sch\"{o}pfer and Lorenz
\cite{SchopferLorenz:2016} showed the exponential convergence for RKM for
sparse recovery with elastic net. We recall the result of Strohmer and Vershynin and its counterpart for noisy data in
Theorem \ref{thm:rkm} below. It is worth noting that all these estimates involve the
condition number, and for noisy data, the estimate contains a term inversely
proportional to the smallest singular value of the matrix $A$.

These important and interesting existing results do not fully explain the excellent empirical performance of RKM for
solving linear inverse problems, especially in the case of noisy data, where the term due to noise is amplified by
a factor of the condition number. In practice, one usually observes that the iterates first converge quickly to a  good approximation to the true
solution, and then start to diverge slowly. That is, it exhibits the typical ``semiconvergence'' phenomenon for
iterative regularization methods, e.g., Landweber method and conjugate gradient methods \cite{Hansen:1998,
KaltenbacherNeubauerScherzer:2008}. This behavior is not well reflected in the known estimates given in Theorem
\ref{thm:rkm}; see Section \ref{sec:Kaczmarz} for further comments.

The purpose of this work is to study the preasymptotic convergence behavior of RKM. This is achieved by analyzing
carefully the evolution of the low- and high-frequency errors during the randomized Kaczmarz iteration, where the
frequency is divided according to the right singular vectors of the matrix $A$. The results indicate that during
initial iterations, the low-frequency error decays must faster than the high-frequency one, cf. Theorems
\ref{thm:err-exact} and \ref{thm:err-noise}. Since the inverse solution (relative to the initial guess $x_0$) is
often smooth in the sense that it consists mostly of low-frequency components \cite{Hansen:1998}, it explains
the good convergence behavior of RKM, thereby shedding new insights into its excellent practical performance. This
condition on the inverse solution is akin to the sourcewise representation condition in classical regularization
theory \cite{EnglHankeNeubauer:1996,ItoJin:2015}. Further, based on the fact that RKM is a special case of the
stochastic gradient method \cite{RobbinsMonro:1951}, we propose a simple modified version using the idea of variance
reduction by hybridizing it with the Landweber method, inspired by \cite{JohnsonZhang:2013}. This variant enjoys
both good preasymptotic and asymptotic convergence behavior, as indicated by the numerical experiments.

Last, we note that in the context of inverse problems, Kaczmarz method has received much recent attention,
and has demonstrated very encouraging results in a number of applications. The regularizing property and
convergence rates in various settings have been analyzed for both linear and nonlinear inverse problems
(see \cite{KowarScherzer:2002,BurgerKaltenbacher:2006,HaltmeierLeitaoScherzer:2007,JinWang:2013,ElfvingHansenNikazad:2014,
KindermannLeitao:2014,LeitaoScaiter:2016,Jin:2016} for an incomplete list). However, these interesting works
all focus on a fixed ordering of the linear system, instead of the randomized variant under consideration here,
and thus they do not cover RKM.

The rest of the paper is organized as follows. In Section \ref{sec:Kaczmarz} we describe RKM and recall the basic
tool for our analysis, i.e., singular value decomposition, and a few useful notations. Then in Section
\ref{sec:conv} we derive the preasymptotic convergence rates for exact and noisy data. Some practical issues are
discussed in Section \ref{sec:implement}. Last, in Section \ref{sec:numer}, we present extensive numerical
experiments to confirm the analysis and shed further insights.

\section{Randomized Kaczmarz method}\label{sec:Kaczmarz}
Now we describe the problem setting and RKM, and also recall known convergence rates results for
both consistent and inconsistent data. The linear inverse problem with exact data can be cast into
\begin{equation}\label{eqn:lin}
  Ax = b,
\end{equation}
where the matrix $A\in \mathbb{R}^{n\times m}$, and $b\in\mathbb{R}^n$ and $b\in \mathrm{range}(A)$.
We denote the $i$th row of the matrix $A$ by $a_i^t$, with $a_i\in\mathbb{R}^m$ being a column vector,
where the superscript $t$ denotes the vector/matrix transpose. The linear system \eqref{eqn:lin} can
be formally determined or under-determined.

The classical Kaczmarz method \cite{Kaczmarz:1937} proceeds as follows. Given the initial guess $x_0$, we iterate
\begin{equation}\label{eqn:Kaczmarz}
  x_{k+1}= x_k + \frac{b_i-\langle a_i,x_k\rangle}{\|a_i\|^2} a_i,\quad i=(k\ \mathrm{ mod }\ n)+1,
\end{equation}
where $\langle\cdot,\cdot\rangle$ and $\|\cdot\|$ denote the Euclidean inner product and norm, respectively.
Thus, Kaczmarz method sweeps through the equations in a cyclic manner, and $n$ iterations constitute one complete cycle.

In contrast to the cyclic choice of the index $i$ in Kaczmarz method, RKM randomly selects $i$. There
are several different variants, depending on the specific random choice of the index $i$. The variant
analyzed by Strohmer and Vershynin \cite{StrohmerVershynin:2009} is as follows. Given an initial guess $x_0$, we iterate
\begin{equation}\label{eqn:rkm}
  x_{k+1}= x_{k} + \frac{b_i-\langle a_i,x_k\rangle}{\|a_i\|^2}a_i,
\end{equation}
where $i$ is drawn independent and identically distributed (i.i.d.) from the index set $\{1,2,\ldots,n\}$ with the
probability $p_i$ for the $i$th row given by
\begin{equation}\label{eqn:probability}
  p_i=\frac{\|a_i\|^2}{\|A\|_F^2},\quad i=1,\ldots,n,
\end{equation}
where $\|\cdot\|_F$ denotes the matrix Frobenius norm. This choice of the probability distribution $p_i$ lends itself to
a convenient convergence analysis \cite{StrohmerVershynin:2009}. In this work, we shall focus on the variant
\eqref{eqn:rkm}-\eqref{eqn:probability}.

Similarly, the noisy data $b^\delta$ is given by
\begin{equation}\label{eqn:noisy-data}
  b_i^\delta=\langle a_i,x^*\rangle+\eta_i,\quad i=1,\ldots,n,\qquad \mbox{with }\ \ \|\eta\| \leq \delta,
\end{equation}
where $\delta$ is the noise level. RKM reads: given the initial guess $x_0$, we iterate
\begin{equation*}
  x_{k+1} = x_k + \frac{b_i^\delta-\langle a_i,x_k\rangle}{\|a_i\|^2}a_i,
\end{equation*}
where the index $i$ is drawn i.i.d. according to \eqref{eqn:probability}.

The following theorem summarizes typical convergence results of RKM for consistent and inconsistent
linear systems \cite{StrohmerVershynin:2009,Needell:2010,ZouziasFreris:2013} (see \cite{MaNeedellRamdas:2015}
for in-depth discussions), under the condition that the matrix $A$ is of {full column-rank}. For a
rectangular matrix $A\in\mathbb{R}^{n\times m}$, we denote by $A^{\dag}\in\mathbb{R}^{m\times n}$ the pseudoinverse of $A$, $\|A\|_2$ denotes the matrix
spectral norm, and $\sigma_\mathrm{min}(A)$ the smallest singular value of $A$. The error $\|x_k-x^*\|$
of the RKM iterate $x_k$ (with respect to the exact solution $x^*$) is stochastic due to the random choice of the index $i$. Below $\E[\cdot]$
denotes expectation with respect to the random row index selection. Note that $\kappa_A$ differs from the usual
condition number \cite{GolubVanLoan:2013}.
\begin{theorem}\label{thm:rkm}
Let $x_k$ be the solution generated by RKM \eqref{eqn:rkm}--\eqref{eqn:probability} at
iteration $k$, and $\kappa_A=\|A\|_F\|A^{\dag}\|_2$ be a (generalized) condition number. Then the following statements hold.
\begin{itemize}
  \item[(i)] For exact data, there holds
  \begin{equation*}
   \E[\|x_k-x^*\|^2] \leq \left(1-\kappa_A^{-2}\right)^k\|x_0-x^*\|^2.
  \end{equation*}
  \item[(ii)]For noisy data, there holds
  \begin{equation*}
    \E[\|x_k-x^*\|^2] \leq \left(1-\kappa_A^{-2}\right)^k\|x_0-x^*\|^2 + \frac{\delta^2}{\sigma_{\min}^2(A)}.
  \end{equation*}
\end{itemize}
\end{theorem}

Theorem \ref{thm:rkm} gives error estimates (in expectation) for any iterate $x_k$, $k\geq 1$: the convergence
rate is determined by $\kappa_A$. For ill-posed linear inverse problems (e.g., CT), bad conditioning
is characteristic and the condition number $\kappa_A$ can be huge, and thus the theorem predicts a very slow
convergence. However, in practice, RKM converges rapidly during the initial iteration. The estimate
is also deceptive for noisy data: due to the presence of the term $\delta^2/\sigma^2_{\min}(A)$,
it implies blowup at the very first iteration, which is however not the case in practice. Hence,
these results do not fully explain the excellent empirical convergence  of RKM for inverse problems.

The next example compares the convergence rates of Kaczmarz method and RKM.
\begin{exam}
Given $n\geq 2$, let $\theta=\frac{2\pi}{n}$. Consider the linear system with $A\in\mathbb{R}^{n\times 2}$,
$a_i=(\begin{array}{c}\cos(i-1)\theta\\ \sin(i-1)\theta\end{array})$ and the exact solution $x^*=0$, i.e.,
$b=0$. Then we solve it by Kaczmarz method and RKM. For any $e_0=(x_0,y_0)$, after one Kaczmarz iteration,
$e_1=(x_0, 0)$, and generally, after $k$ iterations,
\begin{equation*}
  \|e_{k+1}\| = |\cos\theta|^{k}\|e_1\|.
\end{equation*}
For large $n$, the decreasing factor $|\cos\theta|$ can be very close to one, and thus each Kaczmarz iteration
can only decrease the error slowly. Thus, the convergence rate of Kaczmarz method depends strongly on $n$: the
larger is $n$, the slower is the convergence. Similarly, for RKM, there holds
\begin{equation*}
    \E[\|e_{k+1}\|^2|e_k] = \frac{1}{n}\sum_{i=1}^n|\cos i\theta|^2\|e_k\|^2
     =\frac{1}{2n}\sum_{i=1}^n(1-\cos 2i\theta)\|e_k\|^2=\frac{1}{2}\|e_k\|^2,
\end{equation*}
and
\begin{equation*}
    \E[\|e_{k+1}\|^2]  = 2^{-(k+1)}\|e_0\|^2.
\end{equation*}
For RKM, the convergence rate is independent of $n$. Further, for any $n> 8$, we have $0<\theta<\frac{\pi}{4}$,
and $\cos \theta\geq |\cos \frac{\pi}{4}|>2^{-1/2}$. This shows the superiority of RKM over the cyclic one.
\end{exam}

Last we recall singular value decomposition (SVD) of the matrix $A$ \cite{GolubVanLoan:2013}, which is the basic
tool for the convergence analysis in Section \ref{sec:conv}. We denote SVD of $A\in\mathbb{R}^{n\times m}$ by
\begin{equation*}
  A = U\Sigma V^t,
\end{equation*}
where $U\in \mathbb{R}^{n\times n}$ and $V\in \mathbb{R}^{m\times m}$ are column orthonormal matrices and their column
vectors known as the left and right singular vectors, respectively, and $\Sigma \in\mathbb{R}^{n\times m}$ is
diagonal with the diagonal elements ordered nonincreasingly, i.e., $\sigma_1\geq \ldots\geq \sigma_r> 0$, with $r=\min(m,n)$.
The right singular vectors $v_i$ span the solution space, i.e., $x\in \mathrm{span}(v_i)$. We shall write
\begin{equation*}
  U = \left(\begin{array}{c}
    u_1^t\\ \vdots\\ u_n^t
  \end{array}\right)\quad\mbox{and}\quad V^t=\left(\begin{array}{c}
    v_1^t\\ \vdots\\ v_m^t
  \end{array}\right),
\end{equation*}
i.e., $V=(v_1\ \ldots \ v_m)$. Note that for inverse problems, empirically, as the index $i$ increases, the
right singular vectors $v_i$ are increasingly more oscillatory, capturing more high-frequency components \cite{Hansen:1998}.
The behavior is analogous to the inverse of Sturm-Liouville operators. For a general class of convolution integral
equations, such oscillating behavior was  established in \cite{FaberManteuffel:1986}. For many practical
applications, the linear system \eqref{eqn:lin} can be regarded as a discrete approximation to the underlying
continuous problem, and thus inherits the corresponding spectral properties.

Given a frequency cutoff number $1\leq L\leq m$, we define two (orthogonal) subspaces of $\mathbb{R}^m$ by
\begin{equation*}
  \mathcal{L}=\mathrm{span}\{v_1,\ldots,v_L\}\quad\mbox{and}\quad \mathcal{H}=\mathrm{span}\{v_{L+1},\ldots,v_m\},
\end{equation*}
which denotes the low- and high-frequency solution spaces, respectively. This is motivated by the
observation that in practice one only looks for smooth solutions that are spanned/well captured
by the first few right singular vectors \cite{Hansen:1998}. This condition is akin to the concept
of sourcewise representation in regularization theory, e.g., $x\in A^*w$ for some $w\in\mathbb{R}^n$
or its variants \cite{EnglHankeNeubauer:1996,ItoJin:2015}, which is needed for deriving convergence rates
for the regularized solution. Throughout, we always assume that the
truncation level $L$ is fixed. Then for any vector $z\in\mathbb{R}^m$, there exists a unique
decomposition $z=P_Lz+P_Hz$, where $P_L$ and $P_H$ are orthogonal projection operators into
$\mathcal{L}$ and $\mathcal{H}$, respectively, which are defined by
\begin{equation*}
  P_Lz =\sum_{i=1}^L\langle v_i,z\rangle v_i\quad \mbox{and}\quad P_Hz=\sum_{i=L+1}^m\langle v_i,z\rangle v_i.
\end{equation*}
These projection operators will be used below to analyze the preasymptotic behavior of RKM.

\section{Preasymptotic convergence analysis}\label{sec:conv}

In this section, we present a preasymptotic convergence analysis of RKM. Let $x^*$ be one solution of
linear system \eqref{eqn:lin}. Our analysis relies on decomposing the error $e_k=x_k - x^*$ of the
$k$th iterate $x_k$ into low- and high-frequency components (according to the right singular vectors).
We aim at bounding the conditional error $\E[\|e_{k+1}\|^2|e_k]$ (on $e_k$, where the expectation
$\E[\cdot]$ is with respect to the random choice of the index $i$, cf. \eqref{eqn:probability}) by analyzing separately $\E[\|P_Le_{k+1}
\|^2|e_k]$ and $\E[\|P_He_{k+1}\|^2|e_k]$. This is inspired by the fact that the inverse solution
consists mainly of the low-frequency components, which is akin to the concept of the source condition
in regularization theory \cite{EnglHankeNeubauer:1996,ItoJin:2015}. Our error estimates allow explaining
the excellent empirical performance of RKM in the context of inverse problems.

We shall discuss the preasymptotic convergence for exact and noisy data separately.

\subsection{Exact data}\label{sec:noisefree}
First, we analyze the case of noise free data. Let $x^*$ be one solution to the linear system \eqref{eqn:lin}, and
$e_{k}=x_k-x^*$ be the error at iteration $k$. Upon substituting the identity $b=Ax^*$ into RKM iterate, we
deduce that for some $i\in\{1,\ldots,n\}$, there holds
\begin{equation}\label{eqn:err}
  e_{k+1}= \left(I-\frac{a_ia_i^t}{\|a_i\|^2}\right)e_k.
\end{equation}
Note that $I-\frac{a_ia_i^t}{\|a_i\|^2}$ is an orthogonal projection operator. 
We first give two useful lemmas.
\begin{lemma}\label{lem:bdd}
For any $e_L\in\mathcal{L}$ and $e_H\in \mathcal{H}$, there hold
\begin{equation*}
  \sigma_L\|e_L\|\leq \|Ae_L\|\leq \sigma_1\|e_L\|,\quad \|Ae_H\|\leq \sigma_{L+1}\|e_H\|,\quad \mbox{and}\quad \langle Ae_L,Ae_H\rangle =0.
\end{equation*}
\end{lemma}
\begin{proof}
The assertions follow directly from simple algebra, and hence the proof is omitted.
\end{proof}

\begin{lemma}\label{lem:bound-a}
For $i=1,\ldots,n$, there holds
\begin{equation*}
  \|P_Ha_i\|^2 \leq \sigma_{L+1}^2\quad \mbox{and}\quad \sum_{i=1}^n\|P_Ha_i\|^2\leq \sum_{i=L+1}^r\sigma_i^2.
\end{equation*}
\end{lemma}
\begin{proof}
By definition, $P_H a_i = \sum_{j=L+1}^m\langle a_i,v_j\rangle v_j.$
Since $a_i^t = u_i^t\Sigma V^t$, there holds
$
  \langle a_i,v_j\rangle = u_i^t\Sigma V^tv_j = \langle u_i,\sigma_je_j\rangle = \sigma_j(u_i)_j.
$
Hence,
$\|P_Ha_i\|^2 = \sum_{j=L+1}^m \langle a_i,v_j\rangle^2 = \sum_{j=L+1}^m \sigma_j^2 |(u_i)_j|^2\leq \sigma_{L+1}^2$. The second estimate follows similarly.
\end{proof}

The next result gives a preasymptotic recursive estimate on $\E[\|P_Le_{k+1}\|^2|e_k]$ and $ \E[\|P_He_{k+1}\|^2
|e_k]$ for exact data $b\in \mathrm{range}(A)$. This represents our first main theoretical result.
\begin{theorem}\label{thm:err-exact}
Let $c_1=\frac{\sigma^2_L}{\|A\|_F^2}$ and $c_2=\frac{\sum_{i=L+1}^r\sigma_i^2}{\|A\|_F^2}$. Then there hold
\begin{equation*}
  \begin{aligned}
    \E[\|P_Le_{k+1}\|^2|e_k] &\leq (1-c_1)\|P_Le_k\|^2 + c_2\|P_He_k\|^2,\\
    \E[\|P_He_{k+1}\|^2|e_k] &\leq c_2\|P_Le_k\|^2 + (1+c_2)\|P_He_k\|^2.
  \end{aligned}
\end{equation*}
\end{theorem}
\begin{proof}
Let $e_L$ and $e_H$ be the low- and high-frequency errors $e_k$, respectively, i.e.,
$e_L=P_Le_k$ and $e_H=P_He_k$. Then by the identities $P_Le_{k+1} =e_L - \frac{1}{\|a_i\|^2}(a_i,e_k)P_La_i$
and $\langle P_La_i,e_L\rangle = \langle a_i,e_L\rangle$, we have
\begin{align*}
  \|P_Le_{k+1}\|^2 & = \|e_L\|^2 - \frac{2}{\|a_i\|^2}\langle P_La_i,e_L\rangle \langle a_i,e_k\rangle + \langle a_i,e_k\rangle^2\frac{\|P_La_i\|^2}{\|a_i\|^4} \\
     & = \|e_L\|^2 - \frac{2}{\|a_i\|^2}\langle a_i,e_L\rangle \langle a_i,e_k\rangle + \langle a_i,e_k\rangle^2\frac{\|P_La_i\|^2}{\|a_i\|^4}\\
    & \leq \|e_L\|^2 -\frac{2}{\|a_i\|^2}\langle a_i,e_L\rangle \langle a_i,e_k\rangle + \frac{\langle a_i,e_k\rangle^2}{\|a_i\|^2}\\
    & = \|e_L\|^2 -\frac{2}{\|a_i\|^2}\langle a_i,e_L\rangle \langle a_i,e_k\rangle + \frac{\langle a_i,e_L\rangle^2+2\langle a_i,e_L\rangle\langle a_i,e_H\rangle+\langle a_i,e_H\rangle^2}{\|a_i\|^2}.
\end{align*}
Upon noting the identity $\sum_{i=1}^n a_ia_i^t = A^t A$,
taking expectation on both sides yields
\begin{align*}
  \E[\|P_{L}e_{k+1}\|^2|e_k] & \leq \|e_L\|^2 - \frac{2}{\|A\|_F^2}\langle e_k,A^tAe_L\rangle + \frac{\|Ae_L\|^2+2\langle e_H,A^tAe_L\rangle+\|Ae_H\|^2}{\|A\|_F^2}.
\end{align*}
Now substituting the splitting $e_k = e_L + e_H$ and rearranging the terms give
\begin{align*}
  \E[\|P_{L}e_{k+1}\|^2|e_k] & \leq \|e_L\|^2 - \frac{2}{\|A\|_F^2}\langle e_L,A^tAe_L\rangle - \frac{2}{\|A\|_F^2}\langle e_H,A^tAe_L\rangle \\
 &\quad
  + \frac{\|Ae_L\|^2+2\langle e_H,A^tAe_L\rangle+\|Ae_H\|^2}{\|A\|_F^2}\\
 &\leq \|e_L\|^2 - \frac{1}{\|A\|_F^2}\|A e_L\|^2 + \frac{\|Ae_H\|^2}{\|A\|_F^2}.
\end{align*}
Thus the first assertion follows from Lemma \ref{lem:bdd}.
The high-frequency component $P_He_{k+1}$ satisfies
$P_He_{k+1} =e_H - \frac{1}{\|a_i\|^2}\langle a_i,e_k\rangle P_Ha_i.$ We appeal to
the inequality $\langle a_i,e_k\rangle^2\leq \|a_i\|^2\|e_k\|^2
=\|a_i\|^2(\|e_L\|^2+\|e_H\|^2)$ to get
\begin{align*}
  \|P_He_{k+1}\|^2 & = \|e_H\|^2 - \frac{2}{\|a_i\|^2}\langle a_i,e_H\rangle \langle a_i,e_k\rangle + \langle a_i,e_k\rangle^2\frac{\|P_Ha_i\|^2}{\|a_i\|^4}\\
     & \leq \|e_H\|^2 -\frac{2}{\|a_i\|^2}\langle a_i,e_H\rangle \langle a_i,e_k\rangle + \frac{\|P_Ha_i\|^2}{\|a_i\|^2}(\|e_L\|^2 +\|e_H\|^2).
\end{align*}
Taking expectation yields
\begin{align*}
  \E[\|P_He_{k+1}\|^2|e_k] & \leq \|e_H\|^2 - \frac{2}{\|A\|_F^2}\|Ae_H\|^2 + \frac{1}{\|A\|_F^2}(\|e_L\|^2+\|e_H\|^2)\sum_{i=1}^n\|P_Ha_i\|^2\\
   & \leq \left(1+\frac{\sum_{i=L+1}^r\sigma_i^2}{\|A\|_F^2}\right)\|e_H\|^2 + \frac{\sum_{i=L+1}^r\sigma_i^2}{\|A\|_F^2}\|e_L\|^2.
\end{align*}
Thus we obtain the second assertion and complete the proof.
\end{proof}

\begin{remark}
By Theorem \ref{thm:err-exact}, the decay of the error
$\E[\|P_Le_{k+1}\|^2|e_k]$ is largely determined by the factor $1-c_1$ and only mildly affected by
$\|P_He_k\|^2$ by a factor $c_2$. The factor $c_2$ is very small in the presence of a gap in the singular value
spectrum at $\sigma_L$, i.e., $\sigma_{L}\gg \sigma_{L+1}$, showing clearly the role of the gap.
\end{remark}

\begin{remark}
Theorem \ref{thm:err-exact} also covers the rank-deficient case, i.e., $\sigma_{L+1}=0$, and it yields
\begin{equation*}
  \begin{aligned}
    \E[\|P_Le_{k+1}\|^2|e_k] \leq (1-c_1)\|P_Le_k\|^2\quad\mbox{and}\quad
    \E[\|P_He_{k+1}\|^2|e_k] \leq \|P_He_k\|^2.
  \end{aligned}
\end{equation*}
If $L=m$, it recovers Theorem \ref{thm:rkm}(i) for exact data. The rank-deficient case was analyzed in
\cite{GowerRichtarik:2015b}.
\end{remark}

\begin{remark}\label{rmk:iter}
By taking expectation of both sides of the estimates in Theorem \ref{thm:err-exact}, we obtain
\begin{eqnarray*}
  \E[\|P_Le_{k+1}\|^2] &\leq (1-c_1)\E[\|P_Le_k\|^2] + c_2\E[\|P_He_k\|^2],\\
    \E[\|P_He_{k+1}\|^2] &\leq c_2\E[\|P_Le_k\|^2] + (1+c_2)\E[\|P_He_k\|^2].
  \end{eqnarray*}
Then the error propagation is given by
\begin{equation*}
 \left[\begin{array}{c}\E[\|P_Le_{k}\|^2]\\ \E[\|P_He_{k}\|^2]\end{array}\right]
   \leq D^k \left[\begin{array}{c}\|P_Le_0\|^2\\ \|P_He_0\|^2\end{array}\right]\qquad \mbox{with }  D =\left[\begin{array}{cc}
   1-c_1 & c_2\\ c_2 & 1+c_2\end{array}\right].
\end{equation*}
The pairs of eigenvalues $\lambda_\pm $ and (orthonormal) eigenfunctions $v_\pm$ of $D$ are given by
\begin{equation*}
    \lambda_\pm =\frac{2-c_1+c_2\pm ((c_1+c_2)^2+4c_2^2)^{1/2}}{2},
\end{equation*}
and
\begin{equation*}
v_\pm = \frac{[((c_1+c_2)^2+4c_2^2)^\frac{1}{2}\mp (c_1+c_2)]^\frac{1}{2}}{\sqrt{2}((c_1+c_2)^2+4c_2^2)^{1/4}}
\left[\begin{array}{c}1\\
    \frac{2c_2}{((c_1+c_2)^2+4c_2^2)^{1/2}\mp(c_1+c_2)}\end{array}\right].
\end{equation*}
For the case $c_2\ll c_1 <1$, i.e., $\alpha = \frac{c_2}{c_1}\ll 1$, we have
\begin{equation*}
  \lambda_+ = 1+ c_1 (\alpha  + O(\alpha^2))  \quad \mbox{and}\quad \lambda_- = 1- c_1 (1 + O(\alpha^2))
\end{equation*}
and
\begin{equation*}
  v_+  \approx \frac{1}{(1+\alpha^2)^\frac{1}{2}}\left[\begin{array}{c}-\alpha\\
     1\end{array}\right] \quad\mbox{and}\quad v_- \approx \frac{1}{(1+\alpha^2)^\frac{1}{2}}\left[\begin{array}{c}1\\
    \alpha\end{array}\right].
\end{equation*}
With $V=[v_+ \ v_-]$, we have the approximate eigendecomposition if $k = O(1)$:
\begin{equation*}
   D^k \approx V \left[\begin{array}{cc} 1 + k \alpha c_1 &\\ & (1-c_1)^k \end{array}\right]V^t.
\end{equation*}
Thus, for  $c_1\gg c_2$, we have the following approximate error propagation for $k = O(1)$:
\begin{equation*}
  \begin{aligned}
    \E[\|P_Le_k\|^2] & \approx (1-c_1)^k\|P_Le_{0}\|^2 + \alpha (1 - (1 - c_1)^k)\|P_He_{0}\|^2,\\
    \E[\|P_He_k\|^2] & \approx \alpha (1 - (1 - c_1)^k) \|P_Le_0\|^2 + (1 + k\alpha c_1)\|P_He_0\|^2.
  \end{aligned}
\end{equation*}
\end{remark}

\subsection{Noisy data}
Next we turn to the case of noisy data $b^\delta$, cf. \eqref{eqn:noisy-data}, we use the superscript $\delta$
to indicate the noisy case. Since $b_i^\delta=b_i+\eta_i$, the RKM iteration reads
\begin{equation*}
  x_{k+1}-x^* = x_k-x^* + \frac{\langle a_i,x^*-x_k\rangle}{\|a_i\|^2}a_i+\frac{\eta_i a_i}{\|a_i\|^2},
\end{equation*}
and thus the random error $e_{k+1} = x_{k+1} - x^*$ satisfies
\begin{equation}\label{eqn:recurs-noise}
  e_{k+1}=\left(I-\frac{a_ia_i^t}{\|a_i\|^2}\right)e_k + \frac{\eta_ia_i}{\|a_i\|^2}.
\end{equation}

Now we give our second main result, i.e., bounds on the errors $\E[\|P_Le_{k+1}\|^2|e_k]$ and $\E[\|P_He_{k+1}\|^2|e_k]$.
\begin{theorem}\label{thm:err-noise}
Let $c_1=\frac{\sigma_L^2}{\|A\|_F^2}$ and $c_2=\frac{\sum_{i=L+1}^r\sigma_i^2}{\|A\|_F^2}$. Then there hold
\begin{align*}
    \E[\|P_Le_{k+1}\|^2|e_k] & \leq (1-c_1)\|P_Le_k\|^2+ c_2\|P_He_k\|^2 + \tfrac{\delta^2}{\|A\|_F^2} + \tfrac{2}{\|A\|_F}\delta\sqrt{c_2}\|e_k\|,\\
    \E[\|P_He_{k+1}\|^2|e_k] &\leq c_2 \|P_Le_k\|^2 + (1+c_2)\|P_He_k\|^2 + \tfrac{\delta^2}{\|A\|_F^2} + \tfrac{2}{\|A\|_F}\delta\sqrt{c_2}\|e_k\|.
\end{align*}
\end{theorem}
\begin{proof}
By the recursive relation \eqref{eqn:recurs-noise}, we have the splitting
\begin{equation*}
  \E[\|P_Le_{k+1}\|^2|e_k] = {\rm I}_1 + {\rm I}_2 + {\rm I}_3,
\end{equation*}
where the terms are given by (with $e_L=P_Le_k$ and $e_H=P_He_k$)
\begin{align*}
  {\rm I}_1 &= \sum_{i=1}^n \frac{\|a_i\|^2}{\|A\|_F^2}\|e_L-\frac{\langle a_i,e_k\rangle}{\|a_i\|^2}P_La_i\|^2,\quad
  {\rm I}_2 = \sum_{i=1}^n\frac{\|a_i\|^2}{\|A\|_F^2}\frac{\eta_i^2\|P_La_i\|^2}{\|a_i\|^4},\\
  {\rm I}_3 & = \sum_{i=1}^n\frac{\|a_i\|^2}{\|A\|_F^2}\left[\frac{2\eta_i}{\|a_i\|^2}\langle P_La_i,e_L\rangle - \frac{2\eta_i}{\|a_i\|^4}\langle P_La_i,P_La_i\rangle\langle a_i,e_k\rangle\right].
\end{align*}
The first term ${\rm I}_1$ can be bounded directly by Theorem \ref{thm:err-exact}. Clearly,
${\rm I}_2\leq \frac{\delta^2}{\|A\|_F^2}$. For the third term ${\rm I}_3$, we note the splitting
\begin{align*}
  &\langle P_La_i,e_L\rangle - \frac{\|P_La_i\|^2}{\|a_i\|^2}\langle a_i,e_k\rangle \\
= & \frac{\|P_La_i\|^2+\|P_Ha_i\|^2}{\|a_i\|^2}\langle P_La_i,e_L\rangle - \frac{\|P_La_i\|^2}{\|a_i\|^2}(\langle P_La_i,e_L\rangle+\langle P_Ha_i,e_H\rangle)\\
=& \frac{\|P_Ha_i\|^2\langle P_La_i,e_L\rangle-\|P_La_i\|^2\langle P_Ha_i,e_H\rangle}{\|a_i\|^2}:={\rm I}_{3,i}.
\end{align*}
By the Cauchy-Schwarz inequality, we have
\begin{equation*}
  |{\rm I}_3|\leq \frac{2}{\|A\|_F^2}\|\eta\|\Big(\sum_{i=1}^n{\rm I}_{3,i}^2\Big)^{1/2}.
\end{equation*}
Direct computation yields
\begin{align*}
  {\rm I}_{3,i}^2 \leq &\frac{\|P_Ha_i\|^2\|P_La_i\|^2}{\|a_i\|^2}\cdot\frac{\|P_Ha_i\|^2\|e_L\|^2 + 2\|P_Ha_i\|\|P_La_i\|\|e_L\|\|e_H\|+ \|P_La_i\|^2\|e_H\|^2}{\|P_La_i\|^2+\|P_Ha_i\|^2}\\
  & \leq \|P_Ha_i\|^2 \frac{(\|P_La_i\|^2 + \|P_Ha_i\|^2)(\|e_L\|^2+\|e_H\|^2)}{\|P_La_i\|^2+\|P_Ha_i\|^2}=\|P_Ha_i\|^2\|e_k\|^2.
\end{align*}
Consequently, by Lemma \ref{lem:bound-a}, we obtain
\begin{equation*}
  |{\rm I}_3|\leq \frac{2}{\|A\|_F^2}\delta\Big(\sum_{i=L+1}^r\sigma_i^2\Big)^{1/2}\|e_k\|.
\end{equation*}
These estimates together show the first assertion. For the high-frequency component $P_He_{k+1}$, we have
\begin{equation*}
  \E[\|P_He_{k+1}\|^2|e_k] = {\rm I}_{4} + {\rm I}_5 + {\rm I}_6,
\end{equation*}
where the terms are given by
\begin{align*}
  {\rm I}_4 &= \sum_{i=1}^n \frac{\|a_i\|^2}{\|A\|_F^2}\|e_H-\frac{\langle a_i,e_k\rangle}{\|a_i\|^2}P_Ha_i\|^2,\quad
  {\rm I}_5 = \sum_{i=1}^n\frac{\|a_i\|^2}{\|A\|_F^2}\frac{\eta_i^2\|P_Ha_i\|^2}{\|a_i\|^4},\\
  {\rm I}_6 & = \sum_{i=1}^n\frac{\|a_i\|^2}{\|A\|_F^2}\left[\frac{2\eta_i}{\|a_i\|^2}\langle P_Ha_i,e_H\rangle - \frac{2\eta_i}{\|a_i\|^4}\langle P_Ha_i,P_Ha_i\rangle\langle a_i,e_k\rangle\right].
\end{align*}
The term ${\rm I}_4$ can be bounded by Theorem \ref{thm:err-exact}. Clearly,
${\rm I}_5 \leq \frac{\delta^2}{\|A\|_F^2}$. For the term ${\rm I}_6$, note the splitting
\begin{align*}
    \langle P_Ha_i,e_H\rangle - \frac{\|P_Ha_i\|^2}{\|a_i\|^2}\langle a_i,e_k\rangle
  =  \frac{1}{\|a_i\|^2}(\|P_La_i\|^2\langle P_Ha_i,e_H\rangle - \|P_Ha_i\|^2\langle P_La_i,e_L\rangle),
\end{align*}
and thus ${\rm I}_6 = -{\rm I}_3$. This shows the second assertion, and completes the proof of the theorem.
\end{proof}

\begin{remark}
Recall the following estimate for RKM \cite[Theorem 3.7]{ZouziasFreris:2013}
\begin{equation*}
  \E[\|e_{k+1}\|^2|e_k] \leq \left(1-\kappa_A^{-2}\right)\|e_k\| + \|A\|_F^{-2}\delta^2.
\end{equation*}
In comparison, the estimate in Theorem \ref{thm:err-noise} is free from $\kappa_A$, but introduces an
additional term $\tfrac{2}{\|A\|_F}\delta\sqrt{c_2}\|e_k\|$. Since $c_2$
is generally very small, this extra term is comparable with $\|A\|_F^{-2}\delta^2$.
Theorem \ref{thm:err-noise} extends Theorem \ref{thm:err-exact} to the noisy case: if $\delta=0$, it recovers
Theorem \ref{thm:err-exact}. It indicates that if the initial error $e_0=x_0-x^*$ concentrates mostly
on low frequency, the iterate will first decrease the error. The smoothness assumption on the initial
error $e_0$ is realistic for inverse problems, notably under the standard source type conditions (for deriving
convergence rates) \cite{EnglHankeNeubauer:1996,ItoJin:2015}. Nonetheless, the deleterious noise influence
will eventually kick in as the iteration proceeds.
\end{remark}

\begin{remark}
One can discuss the evolution of the iterates for noisy data, similar to Remark \ref{rmk:iter}.
By Young's inequality $2ab\leq \epsilon a^2+\epsilon^{-1}b^2$, the error satisfies {\rm(}with $\bar c_1 = c_1 - \epsilon c_2$ and $\bar c_2 = (1+\epsilon)c_2${\rm)}
\begin{align*}
    \E[\|P_Le_{k+1}\|^2|e_k] & \leq (1-\bar c_1)\|P_Le_k\|^2+ \bar c_2\|P_He_k\|^2 + \tfrac{(1+\epsilon^{-1})\delta^2}{\|A\|_F^2},\\
    \E[\|P_He_{k+1}\|^2|e_k] &\leq \bar c_2 \|P_Le_k\|^2 + (1+\bar c_2)\|P_He_k\|^2 + \tfrac{(1+\epsilon^{-1})\delta^2}{\|A\|_F^2}.
\end{align*}
Then it follows that
\begin{align*}
 \left[\begin{array}{c}\E[\|P_Le_{k}\|^2]\\ \E[\|P_He_{k}\|^2]\end{array}\right]
   & \leq D^k \left[\begin{array}{c}\|P_Le_0\|^2\\ \|P_He_0\|^2\end{array}\right] + \tfrac{(1+\epsilon^{-1})\delta^2}{\|A\|_F^2}(I-D)^{-1}(I-D^{k})\left[\begin{array}{c}
   1\\ 1\end{array}\right],\quad D =\left[\begin{array}{cc}
   1-\bar c_1 & \bar c_2\\ \bar c_2 & 1+\bar c_2\end{array}\right].
\end{align*}
In the case $\bar c_2\ll\bar c_1<1$ and $\alpha = \frac{\bar c_2}{\bar c_1}\ll 1$ {\rm(}by choosing sufficiently
small $\epsilon${\rm)}, for $k=O(1)$, repeating the analysis in Remark \ref{rmk:iter} yields
\begin{equation*}
  \begin{aligned}
    \E[\|P_Le_k\|^2] & \approx (1-\bar c_1)^k\|P_Le_{0}\|^2 + \alpha (1 - (1 - \bar c_1)^k)\|P_He_{0}\|^2 + k\tfrac{(1+\epsilon^{-1})\delta^2}{\|A\|_F^2},\\
    \E[\|P_He_k\|^2] & \approx \alpha (1 - (1 - \bar c_1)^k) \|P_Le_0\|^2 + (1 + k\alpha \bar c_1)\|P_He_0\|^2 + k\tfrac{(1+\epsilon^{-1})\delta^2}{\|A\|_F^2}.
  \end{aligned}
\end{equation*}
Thus, the presence of data noise only influences the error of the RKM iterates mildly by an additive factor $(k\delta^2)$, during the initial iterations.
\end{remark}

\section{RKM with variance reduction}\label{sec:implement}

When equipped with a proper stopping criterion, Kaczmarz method is a regularization method
\cite{KowarScherzer:2002,KaltenbacherNeubauerScherzer:2008}. Naturally, one would expect that this
assertion holds also for RKM \eqref{eqn:rkm}--\eqref{eqn:probability}. This however remains to be
proven due to the lack of a proper stopping criterion.  To see the delicacy, consider one natural
choice, i.e., Morozov's discrepancy principle \cite{Morozov:1966}:  choose the smallest integer $k$ such that
\begin{equation}\label{eqn:dp}
   \|Ax_k-b^\delta\| \leq \tau\delta,
\end{equation}
where $\tau>1$ is fixed \cite{EnglHankeNeubauer:1996,ItoJin:2015}. Theoretically, it is still unclear
that \eqref{eqn:dp} can be satisfied within a finite number of iterations for every noise level
$\delta>0$. In practice, computing the residual $\|Ax_k-b^\delta\|$ at each iteration is undesirable
since its cost is of the order of evaluating the full gradient, whereas avoiding the latter is the
very motivation for RKM! Below we propose one simple remedy by drawing on its connection with stochastic
gradient methods \cite{RobbinsMonro:1951} and the vast related developments.

First we note that the solution to \eqref{eqn:lin} is equivalent to minimizing the least-squares problem
\begin{equation}\label{eqn:ls}
  \min_{x\in\mathbb{R}^n}\Big\{f(x):=\frac{1}{2n}\sum_{i=1}^n|\langle a_i,x\rangle-b_i|^2\Big\}.
\end{equation}
Next we recast RKM as a stochastic gradient method for problem \eqref{eqn:ls}, as noted
earlier in \cite{NeedellSrebroWard:2016}. We include a short proof for completeness.

\begin{proposition}
The RKM iteration \eqref{eqn:rkm}-\eqref{eqn:probability} is a (weighted) stochastic gradient update
with a constant stepsize $n/\|A\|_F^2$.
\end{proposition}
\begin{proof}
With the weight $w_i=\|a_i\|^2$, we rewrite problem \eqref{eqn:ls} into
\begin{equation*}
  \begin{aligned}
    \frac{1}{2n}\sum_{i=1}^n(\langle a_i,x\rangle-b_i)^2&=\frac{1}{2n}\sum_{i=1}^n\frac{w_i}{\|A\|_F^2} \frac{\|A\|_F^2}{w_i}(\langle a_i,x\rangle-b_i)^2,\\
     & = \sum_{i=1}^n\frac{w_i}{\|A\|_F^2} f_i, \quad  \mbox{with}\quad f_i(x) = \frac{\|A\|_F^2}{2nw_i}(\langle a_i,x\rangle-b_i)^2.
  \end{aligned}
\end{equation*}
Since $\sum_{i=1}^n w_i=\|A\|_F^2$, we may interpret $p_i=w_i/\|A\|_F^2$ as a probability distribution
on the set $\{1,\ldots,n\}$, i.e. \eqref{eqn:probability}. Next we apply the stochastic gradient
method. Since $g_i(x):=\nabla f_i(x)
=\frac{\|A\|_F^2}{nw_i}(\langle a_i,x\rangle-b_i)a_i$, with a fixed step length $\eta=n\|A\|_F^{-2}$, we get
\begin{equation*}
  x_{k+1} = x_k - w_i^{-1}(\langle a_i,x\rangle-b)a_i,
\end{equation*}
where $i\in \{1,\ldots,n\}$ is drawn i.i.d. according to  \eqref{eqn:probability}. Clearly, it is
equivalent to RKM \eqref{eqn:rkm}-\eqref{eqn:probability}.
\end{proof}

Now we give the mean and variance of the stochastic gradient $g_i(x)$.
\begin{proposition}\label{prop:var}
Let $g(x)=\nabla f(x)$. Then the gradient $g_i(x)$ satisfies
\begin{equation*}
  \begin{aligned}
  \mathrm{\E}[g_i(x)] & = g(x),\\
  \mathrm{Cov}[g_i(x)] & = \frac{\|A\|_F^2}{n^2}\sum_{i=1}^n(\langle a_i,x\rangle-b_i)^2\frac{a_ia_i^t}{\|a_i\|^2}-\frac{1}{n^2}A^t(Ax-b)(Ax-b)^tA.
  \end{aligned}
\end{equation*}
\end{proposition}
\begin{proof}
The full gradient $g(x):=\nabla f(x)$ at $x$ is given by $g(x)= \tfrac{1}{n}A^t(Ax-b)$.
The mean $\E[g_i(x)]$ of the (partial) gradient $g_i(x)$ is given by
\begin{equation*}
  \E[g_i(x)]=\frac{1}{n}\sum_{i=1}^n\frac{\|a_i\|^2}{\|A\|^2_F}\frac{\|A\|_F^2}{\|a_i\|^2}(\langle a_i,x\rangle-b_i)a_i=\tfrac{1}{n}A^t(Ax-b).
\end{equation*}
Next, by bias-variance decomposition, the covariance $\mathrm{Cov}[g_i(x)]$ of the gradient $g_i(x)$
is given by
\begin{equation*}
  \begin{aligned}
    \mathrm{Cov}[g_i(x)]&=\E[g_i(x)g_i(x)^t]-\E[g_i(x)]\E[g_i(x)]^t\\
      & = \frac{\|A\|_F^4}{n^2}\sum_{i=1}^n\frac{\|a_i\|^2}{\|A\|^2_F}\frac{1}{\|a_i\|^4}(\langle a_i,x\rangle-b_i)^2a_ia_i^t-\frac{1}{n^2}A^t(Ax-b)(Ax-b)^tA\\
      & =  \frac{\|A\|_F^2}{n^2}\sum_{i=1}^n(\langle a_i,x\rangle-b_i)^2\frac{a_ia_i^t}{\|a_i\|^2}-\frac{1}{n^2}A^t(Ax-b)(Ax-b)^tA.
  \end{aligned}
\end{equation*}
This completes the proof of the proposition.
\end{proof}

Thus, the single gradient $g_i(x)$ is an unbiased estimate of the full gradient $g(x)$. For consistent
linear systems, the covariance $\mathrm{Cov}[g_i(x)]$ is asymptotically vanishing: as $x_k\to x^*$,
both terms in the variance expression tend to zero. However, for inconsistent linear systems, the
covariance $\mathrm{Cov}[g_i(x)]$ generally does not vanish at the optimal solution $x^*$:
\begin{equation*}
    \mathrm{Cov}[g_i(x^*)] \approx \frac{\|A\|_F^2}{n^2}\sum_{i=1}^n(\langle a_i,x^*\rangle-b_i^\delta)^2\frac{a_ia_i^t}{\|a_i\|^2},
\end{equation*}
since one might expect $A^t(Ax^*-b^\delta)\approx 0$. Further, $\mathrm{Cov}[g_i(x^*)]$ is of the order
$\delta^2$ in the neighborhood of $x^*$. One may predict the (asymptotic) dynamics of RKM via a stochastic
modified equation from the covariance \cite{LiTai:2015}. The RKM iteration eventually deteriorates due to the
nonvanishing covariance so that its asymptotic convergence slows down.

These discussions motivate the use of variance reduction techniques developed for stochastic gradient methods
to reduce the variance of the gradient estimate. There are several possible strategies, e.g., stepsize reduction,
stochastic variance reduction gradient (SVRG), averaging and mini-batch (see e.g., \cite{SchmidtBach:2017,
JohnsonZhang:2013}). We only adapt  SVRG \cite{JohnsonZhang:2013} to RKM, termed as RKM with variance reduction
(RKMVR), cf. Algorithm \ref{alg:rkm-vr} for details. It hybridizes the stochastic gradient with the (occasional)
full gradient to achieve variance reduction. Here, $s$ is the length of epoch, which determines the frequency of
full gradient evaluation and was suggested to be $n$ \cite{JohnsonZhang:2013}, and $K$ is the maximum number of
iterations. In view of Step 2, within the first epoch, it performs only the standard RKM, and at the end of the epoch,
it evaluates the full gradient. In RKMVR, the residual $\|Ax_k-b^\delta\|$ is a direct by-product of full gradient
evaluation and occurs only at the end of each epoch, and thus it does not invoke additional computational effort.

The update at Step 8 of Algorithm \ref{alg:rkm-vr} can be rewritten as (for $k\geq s$)
\begin{equation*}
   x_{k+1}  = x_k + \frac{\langle a_i,\tilde x-x_k\rangle a_i}{\|a_i\|^2}-\frac{n}{\|A\|_F^2}\tilde g,
\end{equation*}
and thus $\tilde x-x_k\to 0$ as the iteration proceeds, and it recovers the Landweber method. With this choice,
the variance of the gradient estimate is asymptotically vanishing \cite{JohnsonZhang:2013}. Numerically,
Algorithm \ref{alg:rkm-vr} converges rather steadily. That is, it combines the strengthes of RKM and the
Landweber method: it merits the fast initial convergence of the former and the excellent stability of the latter.

\begin{algorithm}
  \centering
  \caption{Randomized Kaczmarz method with variance reduction (RKMVR).\label{alg:rkm-vr}}
  \begin{algorithmic}[1]
  \STATE Specify $A$, $b$, $x_0$, $K$, and $s$.
  \STATE Initialize $g_i(\tilde x) = 0$, and $\tilde g = 0$.
  \FOR {$k=1,\ldots,K$}
     \IF {$k\ \mathrm{mod}\ s=0$}
      \STATE Set $\tilde x = x_k$ and $\tilde g= g(x_k)$.
      \STATE Check the discrepancy principle \eqref{eqn:dp}.
     \ENDIF
     \STATE Pick an index $i$ according to \eqref{eqn:probability}.
     \STATE Update $x_k$ by
     \begin{equation*}
       x_{k+1} = x_k - \frac{n}{\|A\|_F^2}(g_i(x_k)-g_i(\tilde x) + \tilde g).
     \end{equation*}
  \ENDFOR
  \end{algorithmic}
\end{algorithm}

\section{Numerical experiments and discussions}\label{sec:numer}

Now we present numerical results for RKM and RKMVR to illustrate their distinct features. All the numerical
examples, i.e., \texttt{phillips}, \texttt{gravity} and \texttt{shaw}, are taken from the public domain
\texttt{MATLAB} package \textbf{Regutools}\footnote{Available from\url{http://www.imm.dtu.dk/~pcha/Regutools/},
last accessed on June 21, 2017}. They are Fredholm integral equations of the first kind, with the first
example being mildly ill-posed, and the last two severely ill-posed, respectively. Unless otherwise stated,
the examples are discretized with a dimension $n=m=1000$. The noisy data $b^\delta$ is generated from the
exact data $b$ as
\begin{equation*}
  b^\delta_i = b_i + \delta \max_{j}(|b_j|)\xi_i,\quad i =1,\ldots,n,
\end{equation*}
where $\delta$ is the relative noise level, and the random variables $\xi_i$s follow an
i.i.d. standard Gaussian distribution. The initial guess $x_0$ for the
iterative methods is $x_0=0$. We present the squared error $e_k$ and/or the squared residual $r_k$, i.e.,
\begin{equation}\label{eqn:err-res}
  e_k =\mathbb{E}[\|x^*-x_k\|^2]\quad \mbox{and}\quad r_k = \mathbb{E}[\|Ax_k-b^\delta\|^2].
\end{equation}
The expectation $\mathbb{E}[\cdot]$ with respect to the random choice of the rows is approximated
by the average of 100 independent runs. All the computations were carried out on a personal laptop
with 2.50 GHz CPU and 8.00G RAM by \texttt{MATLAB} 2015b.

\subsection{Benefit of randomization}

First we compare the performance of  RKM with the cyclic Kaczmarz method (KM) to illustrate the benefit
of randomization. Overall, the random reshuffling can substantially improve the convergence of KM, cf.
the results in Figs. \ref{fig:phil-kmrkm}-\ref{fig:shaw-kmrkm} for the examples with different noise levels.

\begin{figure}[hbt!]
  \centering
  \setlength{\tabcolsep}{0pt}
  \begin{tabular}{cccc}
   \includegraphics[trim={2.0cm 0 2.3cm 0.2cm},clip,width=.25\textwidth]{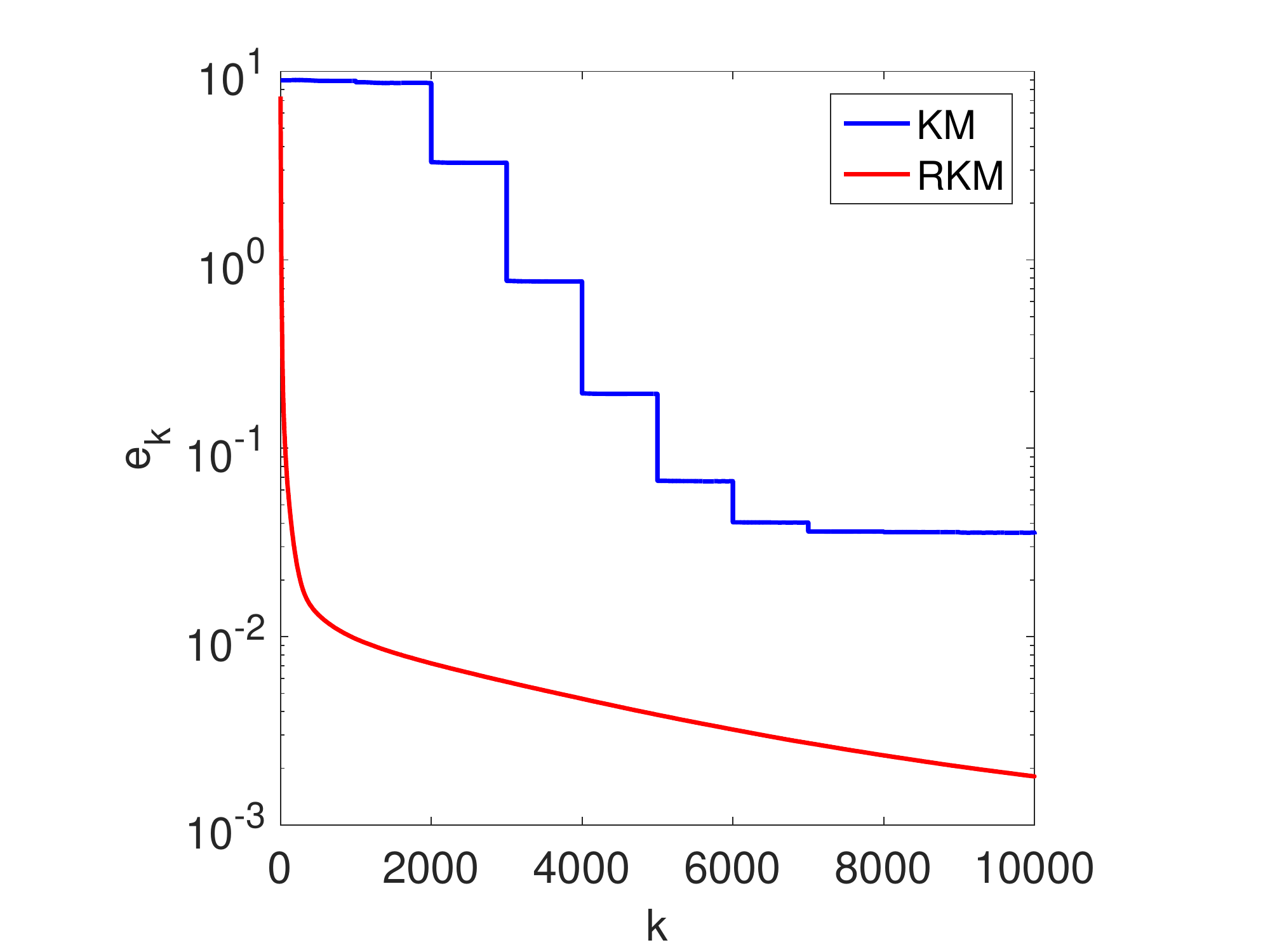} & \includegraphics[trim={2.cm 0 2.3cm 0.2cm},clip,width=.25\textwidth]{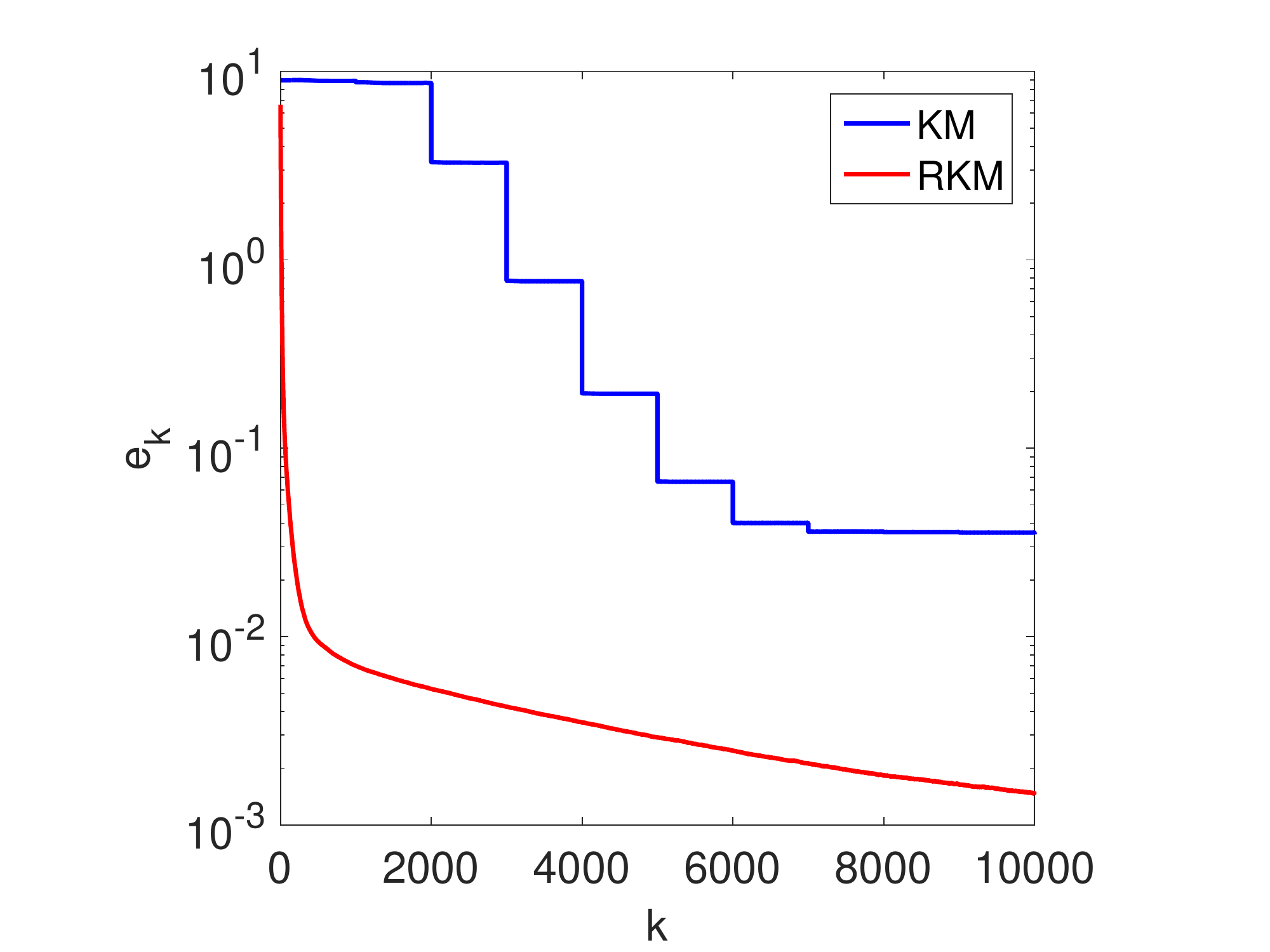}& \includegraphics[trim={2.0cm 0 2.3cm 0.2cm},clip,width=.25\textwidth]{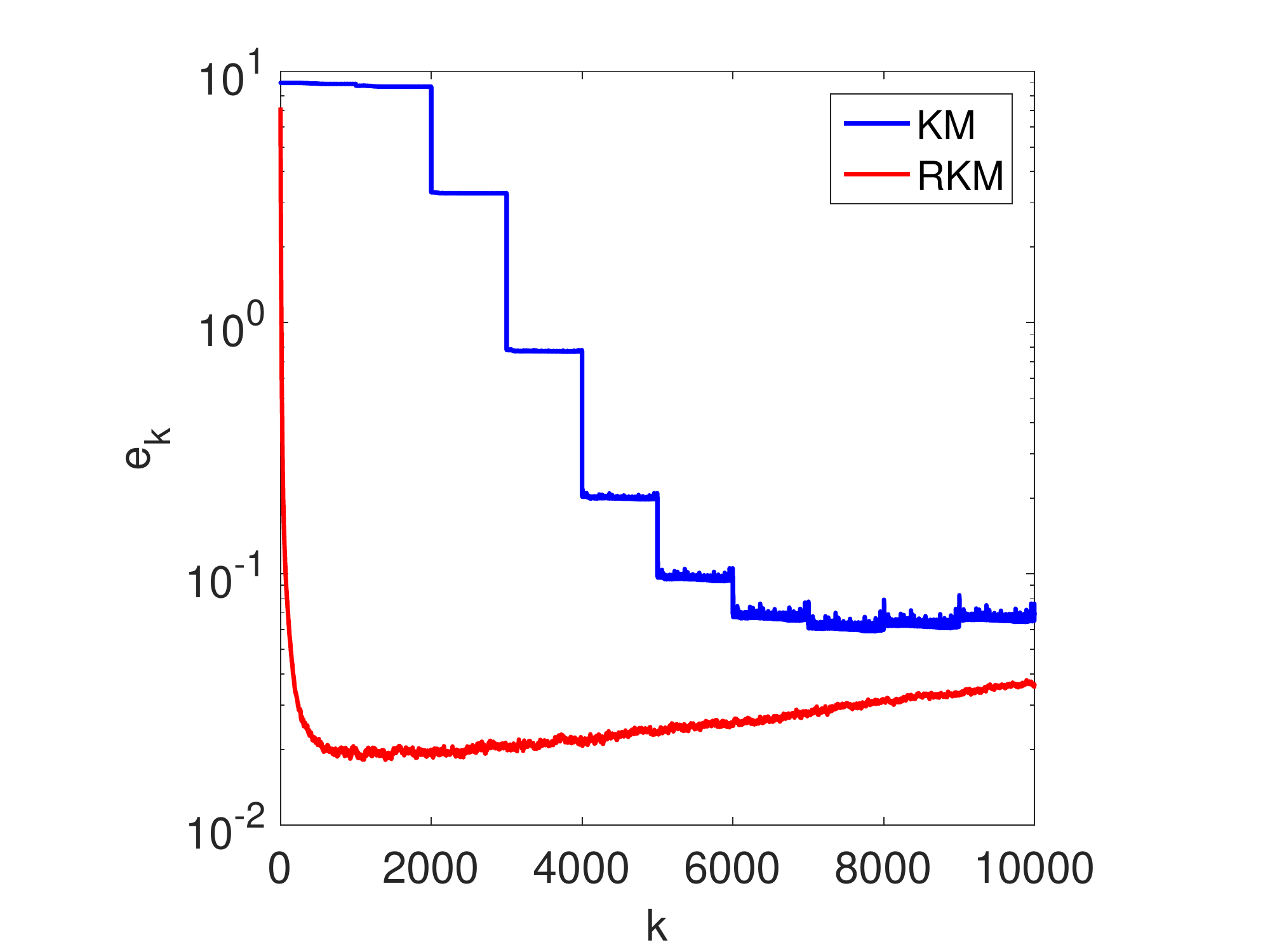} & \includegraphics[trim={2.cm 0 2.3cm 0.2cm},clip,width=.25\textwidth]{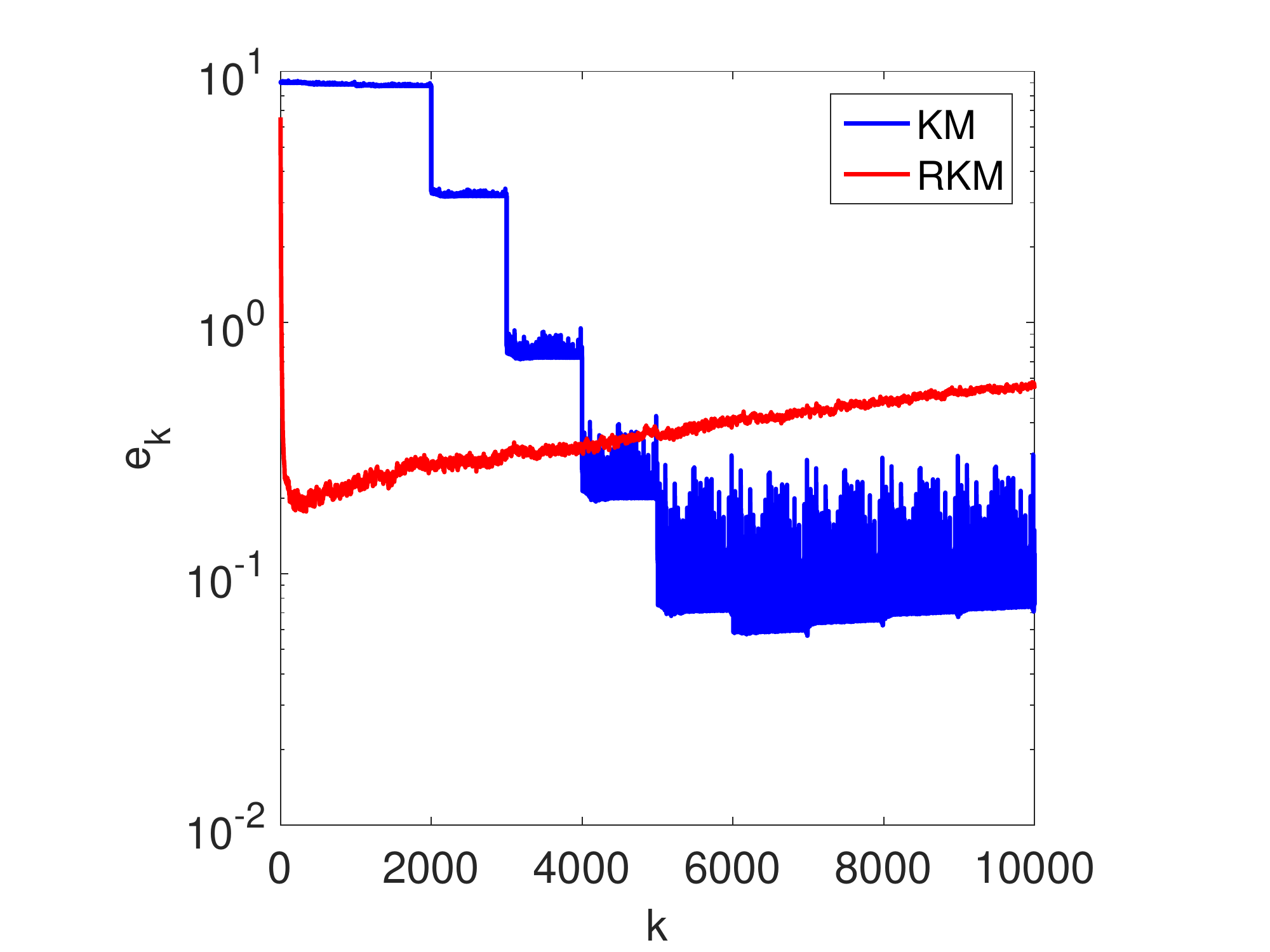}\\
   \includegraphics[trim={2cm 0 2.3cm 0.2cm},clip,width=.25\textwidth]{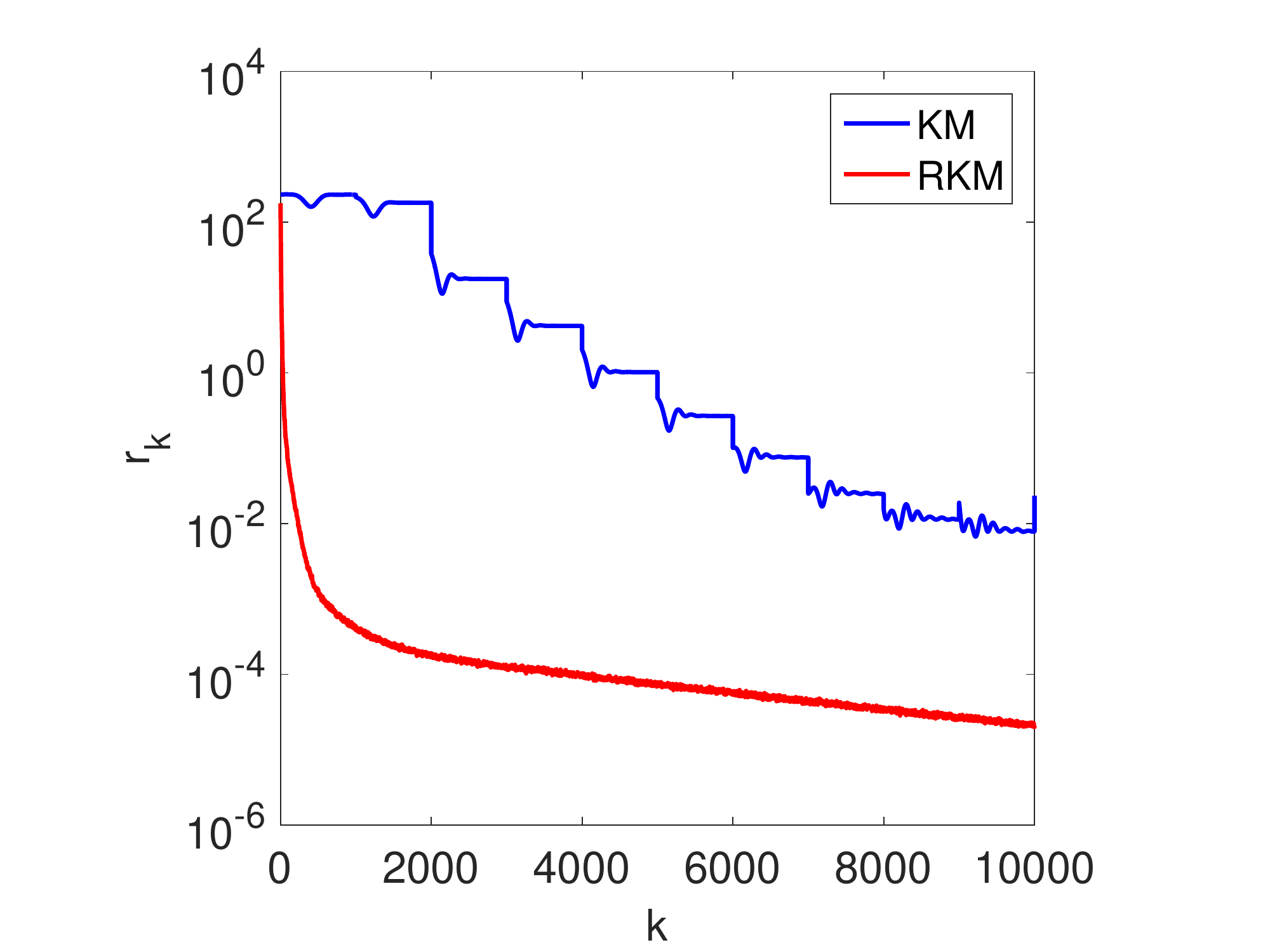} & \includegraphics[trim={2cm 0 2.3cm 0.2cm},clip,width=.25\textwidth]{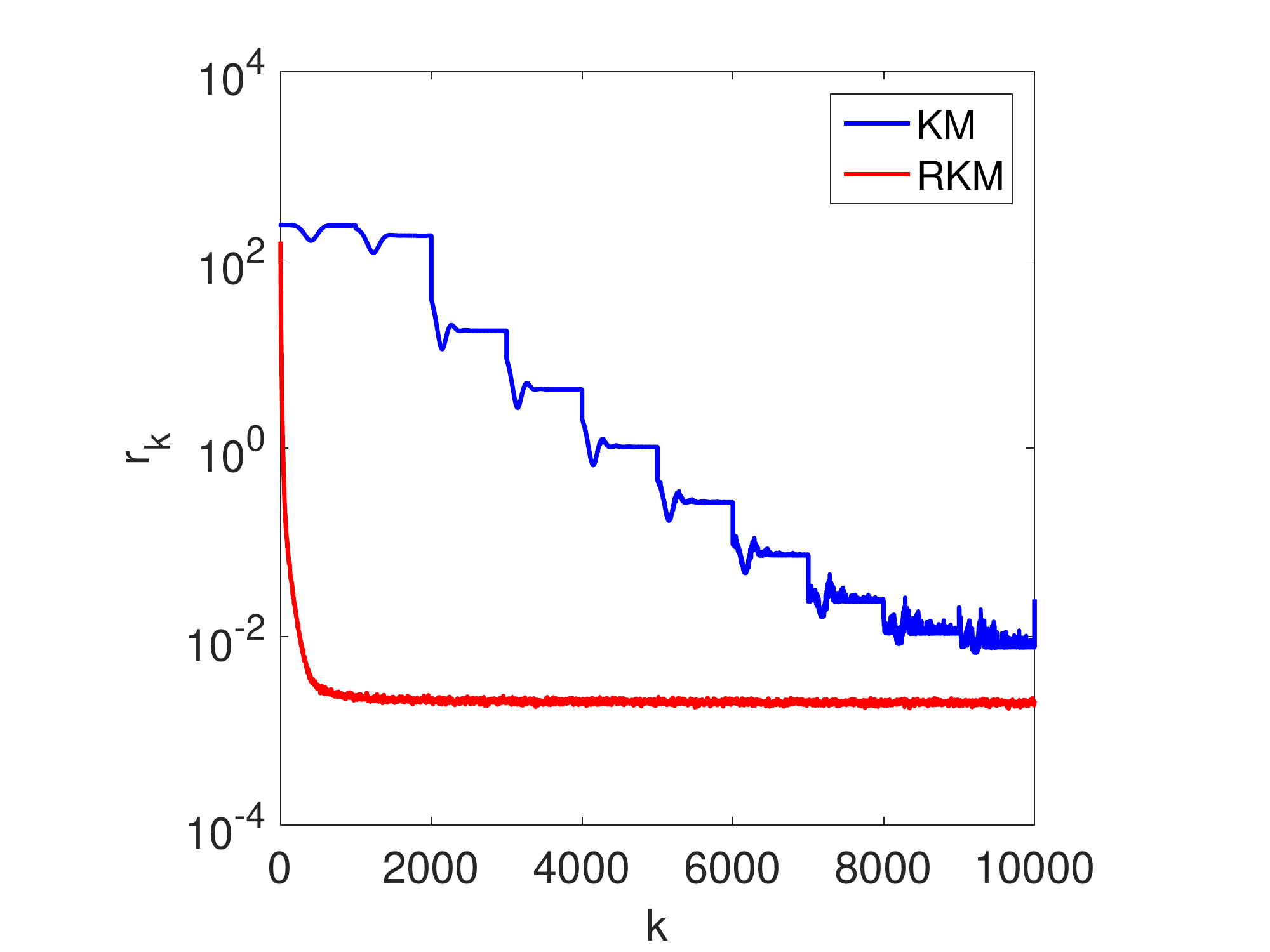}& \includegraphics[trim={2cm 0 2.3cm 0.2cm},clip,width=.25\textwidth]{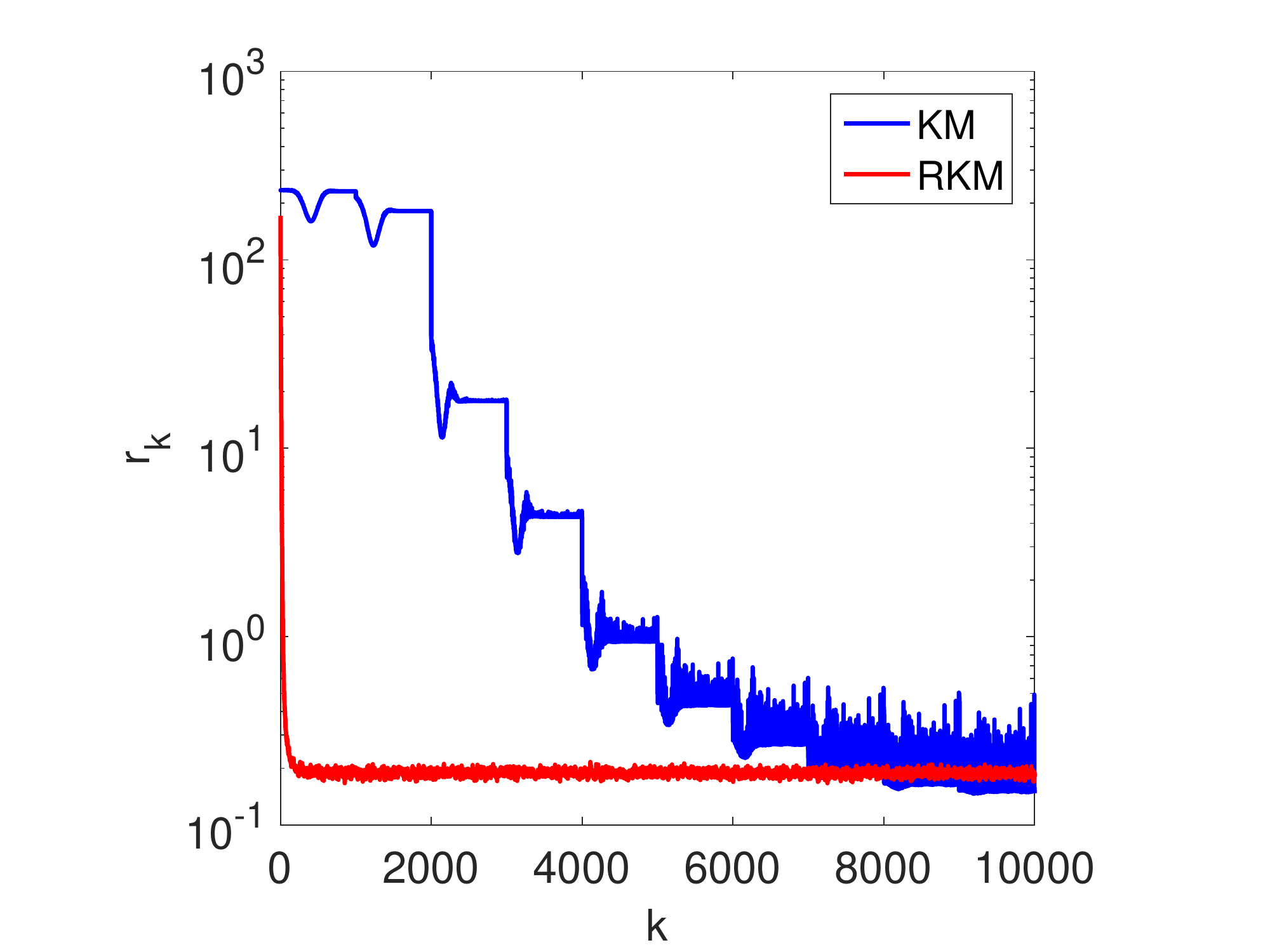} & \includegraphics[trim={2cm 0 2.3cm 0.2cm},clip,width=.25\textwidth]{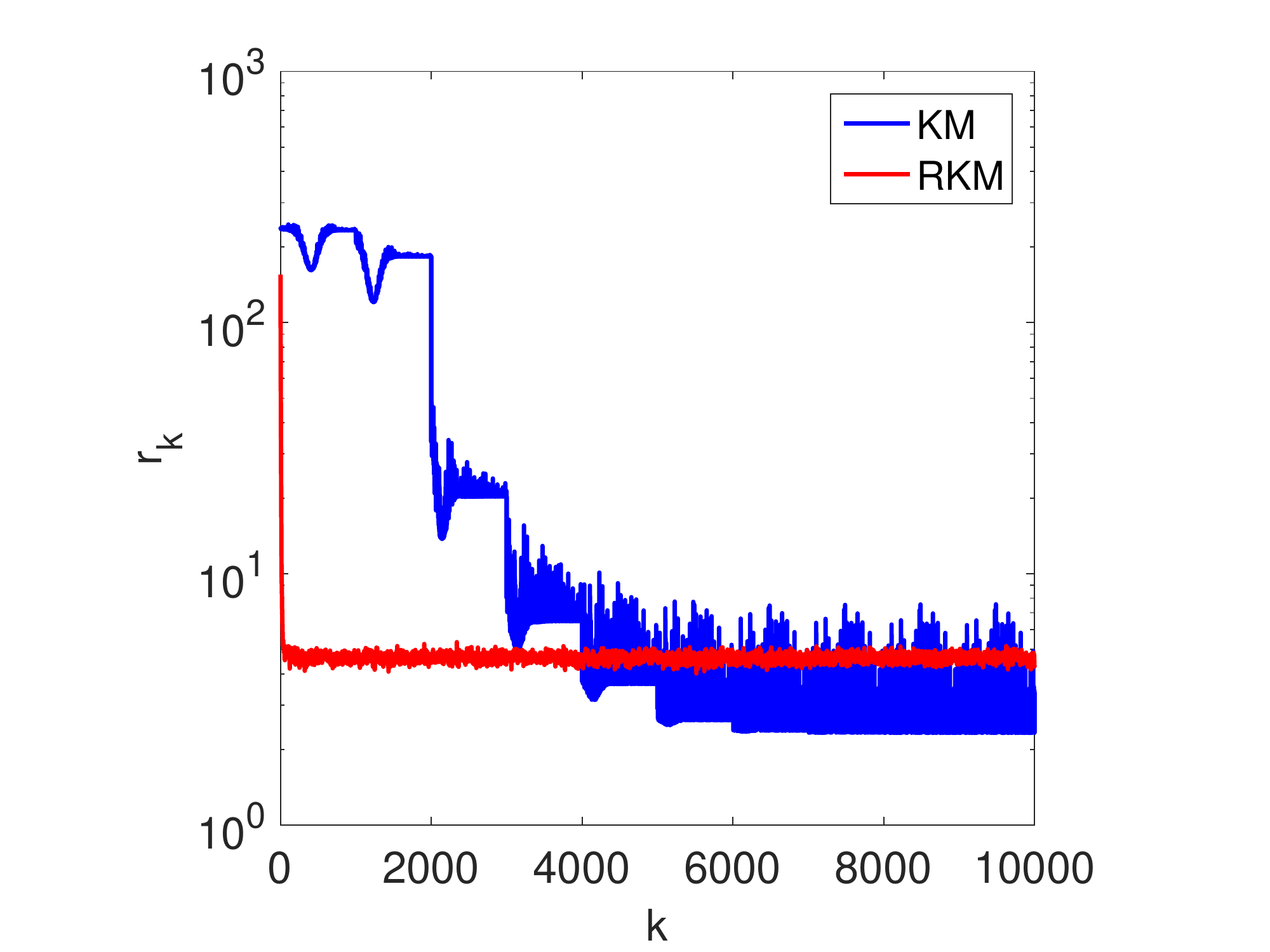}\\
   (a) $\delta=0$ & (b) $\delta=10^{-3}$ & (c) $\delta = 10^{-2}$ & (d) $\delta=5\times10^{-2}$
  \end{tabular}
  \caption{Numerical results ($e_k$ and $r_k$) for \texttt{phillips} by KM and RKM. \label{fig:phil-kmrkm}}
\end{figure}

Next we examine the convergence more closely. The (squared) error $e_k$ of the Kaczmarz iterate $x_k$ undergoes
a sudden drop at the end of each cycle, whereas within the cycle, the drop after each Kaczmarz iteration is small.
Intuitively, this can be attributed to the fact that the neighboring rows of the matrix $A$ are highly correlated
to each other, and thus each single Kaczmarz iteration reduces only very little the (squared) error $e_k$, since
roughly it repeats the previous projection. The strong correlation between the neighboring rows is the culprit of
the slow convergence of the cyclic KM. The randomization ensures that any two rows chosen by two consecutive RKM
iterations are less correlated, and thus the iterations are far more effective for reducing the error $e_k$, leading
to a much faster empirical convergence. These observations hold for both exact and noisy data. For noisy data, the
error $e_k$ first decreases and then increases for both KM and RKM, and the larger is the noise level $\delta$,
and the earlier does the divergence occur. That is, both exhibit a ``semiconvergence'' phenomenon typical for
iterative regularization methods. Thus a suitable stopping criterion is needed. Meanwhile, the residual $r_k$
tends to decrease, but for both methods, it oscillates wildly for noisy data and the oscillation magnitude
increases with $\delta$. This is due to the nonvanishing variance, cf. the discussions in Section \ref{sec:implement}.
One surprising observation is that a fairly reasonable inverse solution can be obtained by RKM within one cycle
of iterations. That is, by ignoring all other cost, RKM can solve the inverse problems reasonably well at a cost
less than one full gradient evaluation!

\begin{figure}[hbt!]
  \centering
  \setlength{\tabcolsep}{0pt}
  \begin{tabular}{cccc}
   \includegraphics[trim={2.0cm 0 2.3cm 0.2cm},clip,width=.25\textwidth]{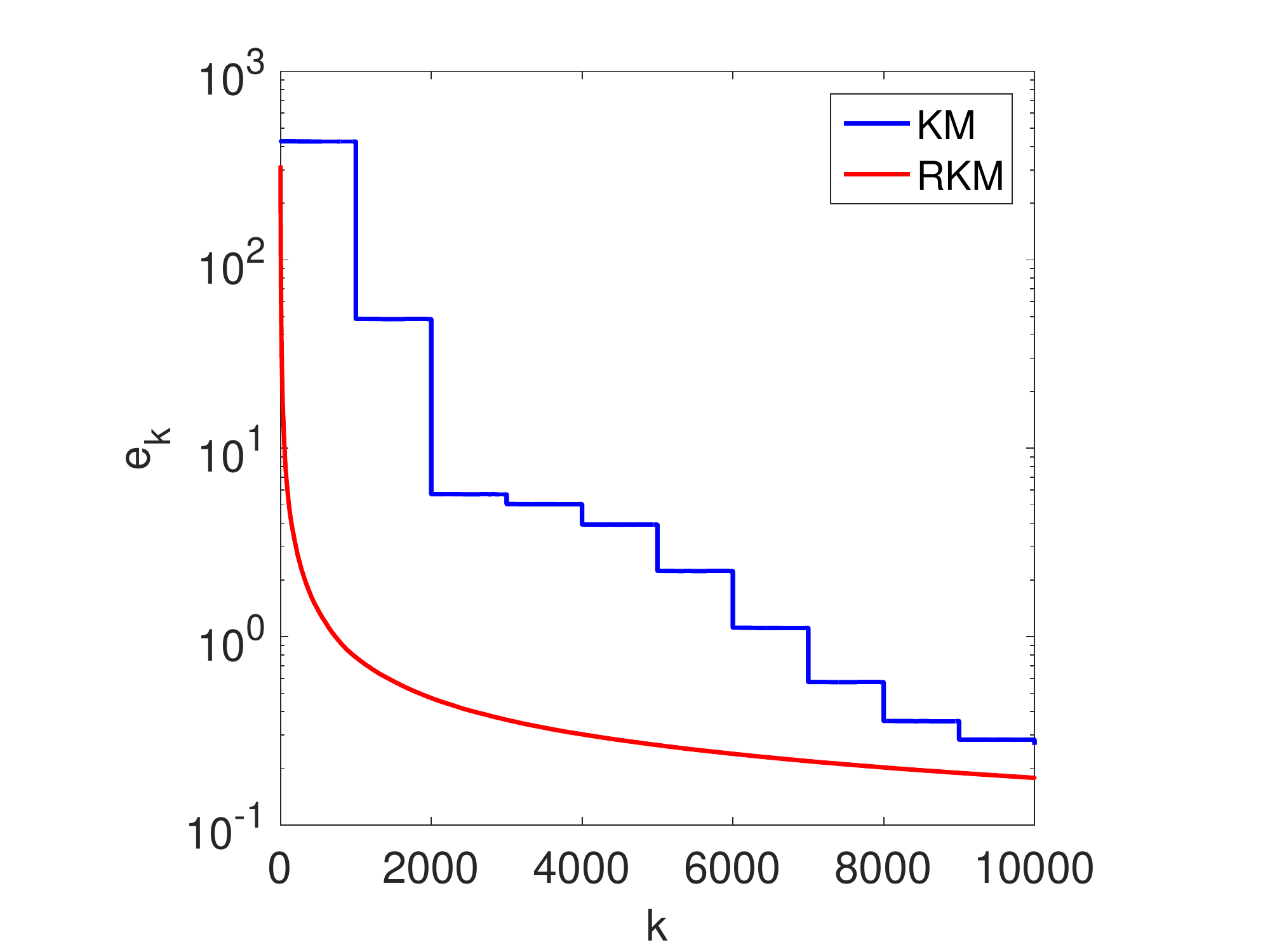}& \includegraphics[trim={2.0cm 0 2.3cm 0.2cm},clip,width=.25\textwidth]{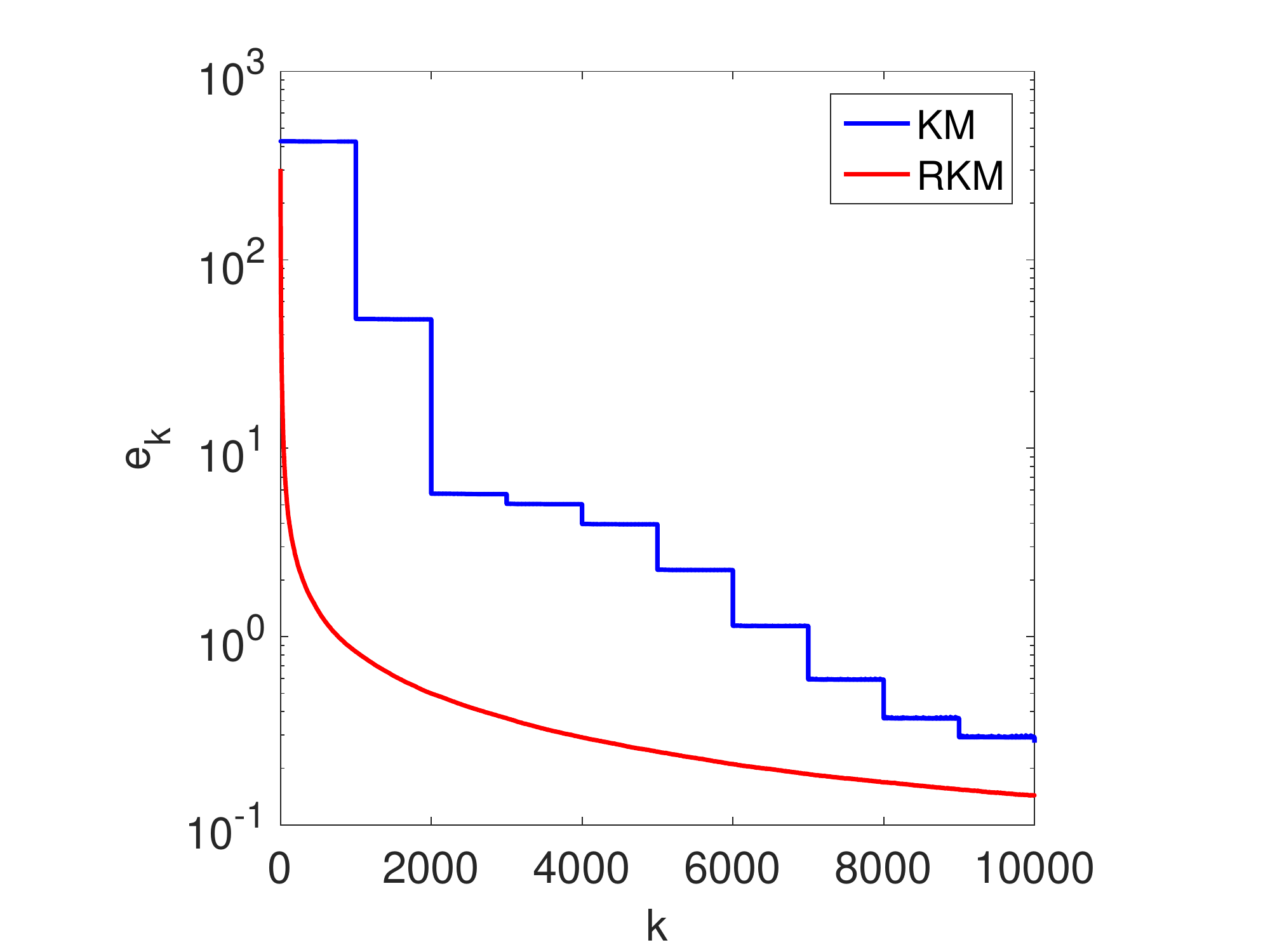}& \includegraphics[trim={2.0cm 0 2.3cm 0.2cm},clip,width=.25\textwidth]{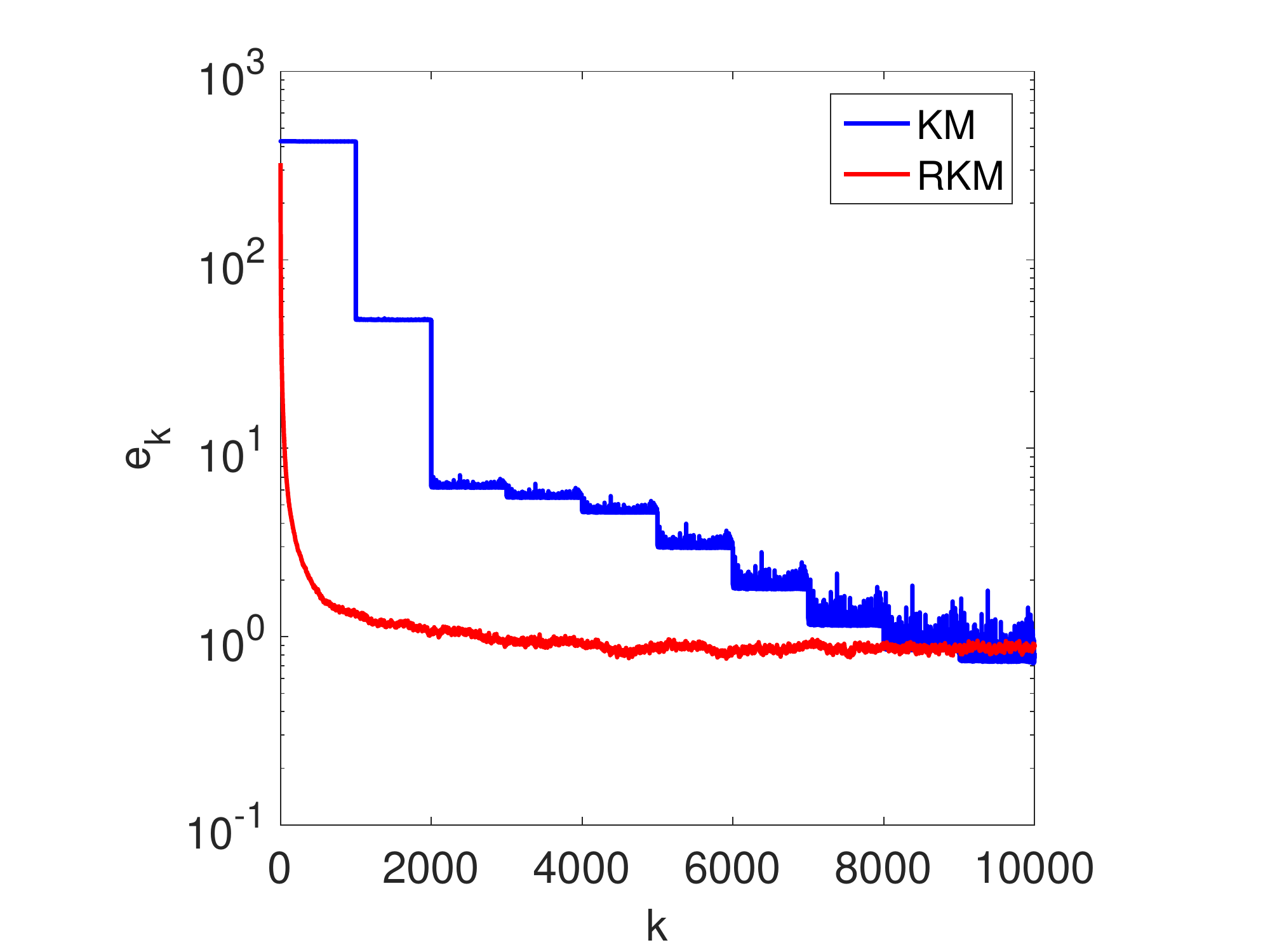}& \includegraphics[trim={2.0cm 0 2.3cm 0.2cm},clip,width=.25\textwidth]{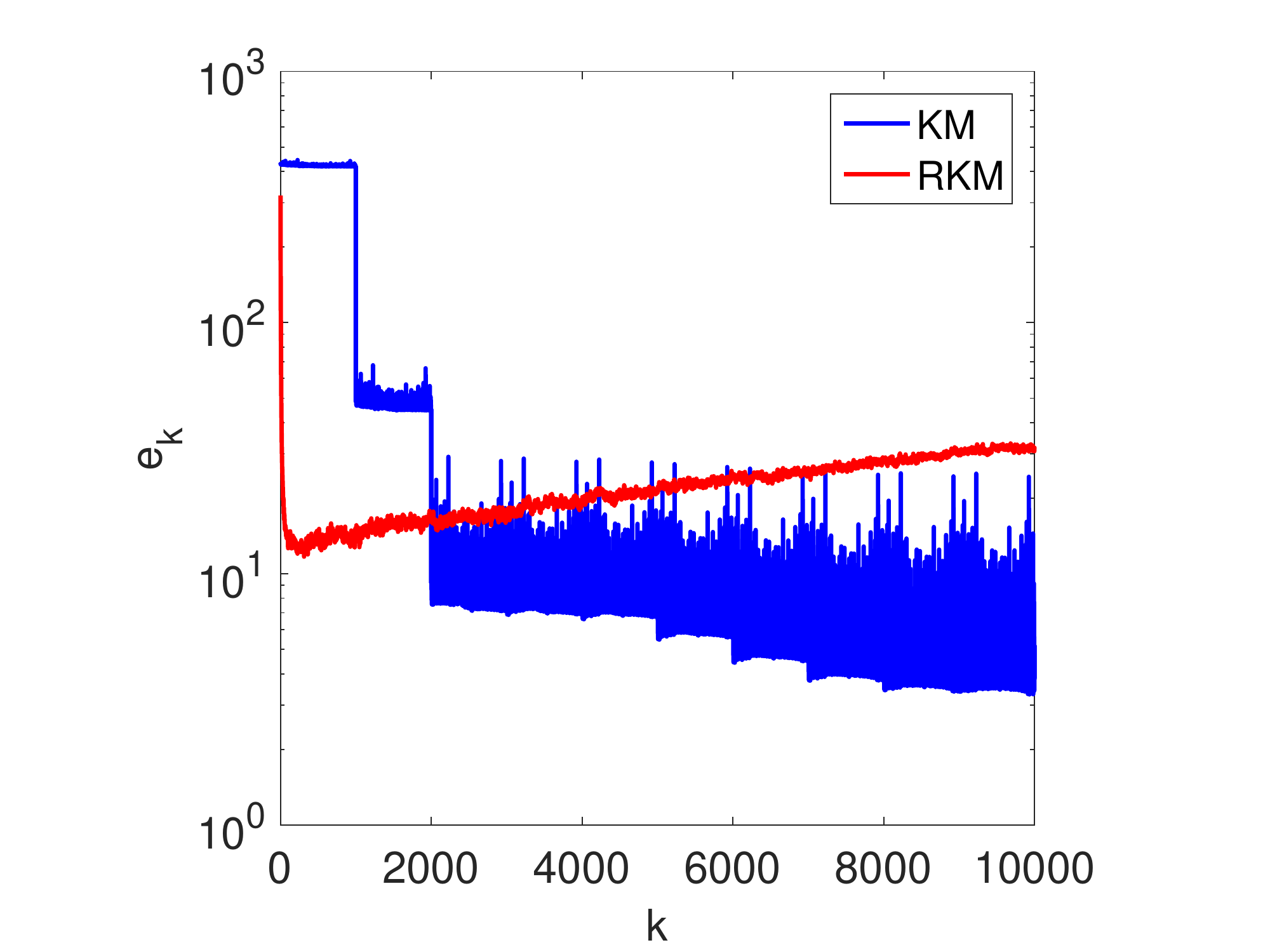}\\
   \includegraphics[trim={2.0cm 0 2.3cm 0.2cm},clip,width=.25\textwidth]{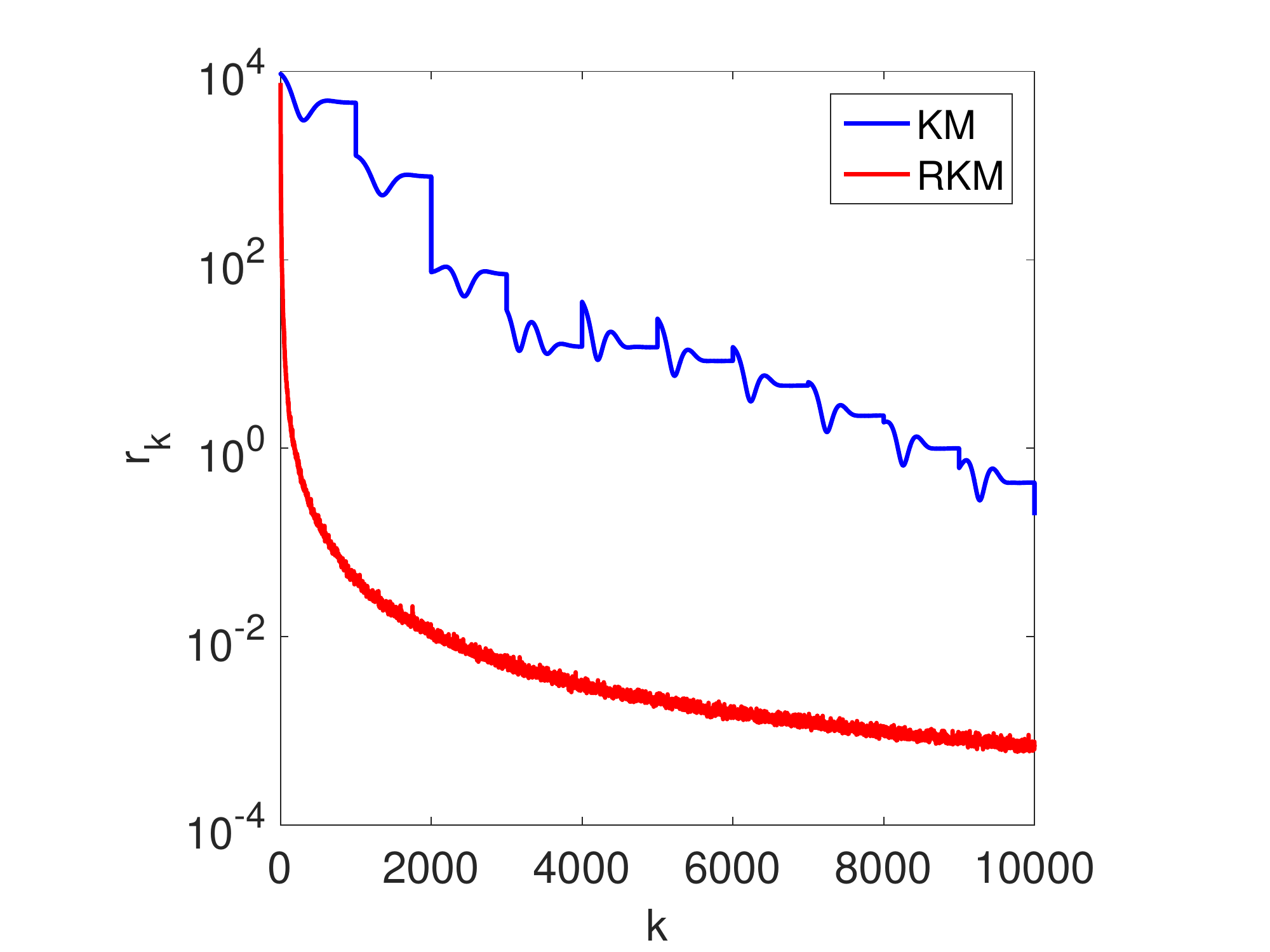}& \includegraphics[trim={2.0cm 0 2.3cm 0.2cm},clip,width=.25\textwidth]{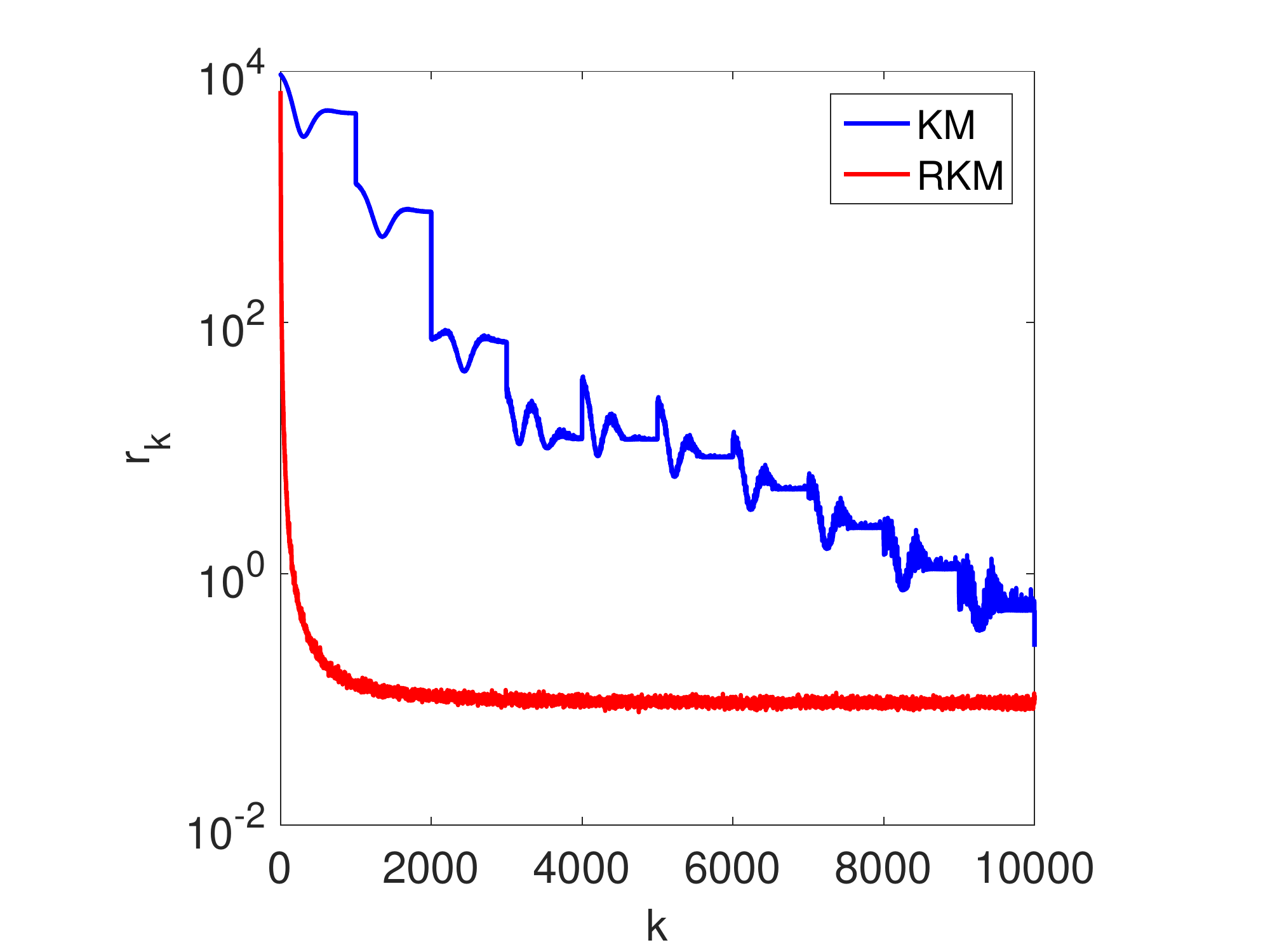}& \includegraphics[trim={2.0cm 0 2.3cm 0.2cm},clip,width=.25\textwidth]{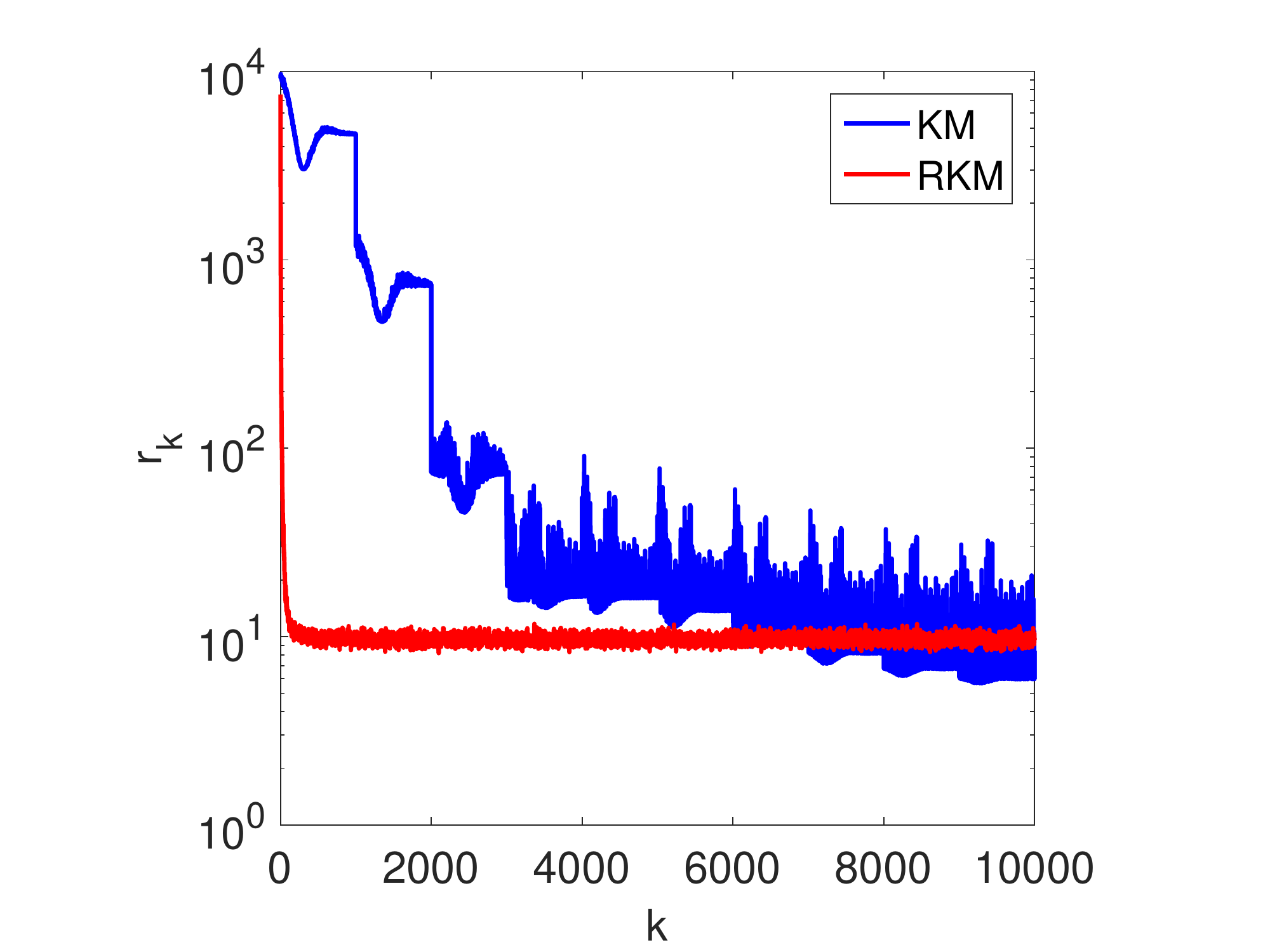}& \includegraphics[trim={2.0cm 0 2.3cm 0.2cm},clip,width=.25\textwidth]{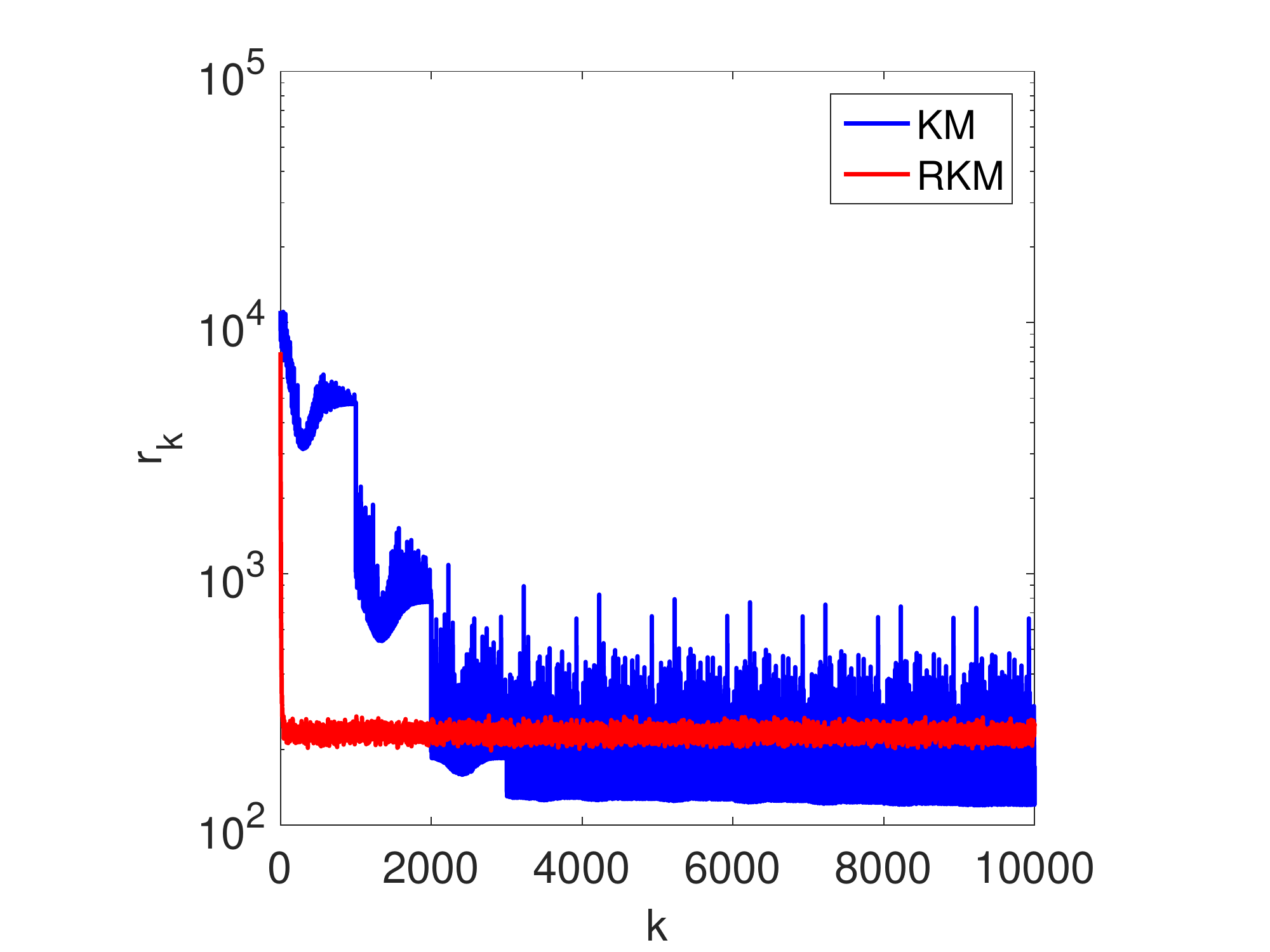}\\
   (a) $\delta=0$ & (b) $\delta=10^{-3}$ & (c) $\delta = 10^{-2}$ & (d) $\delta=5\times10^{-2}$
  \end{tabular}
  \caption{Numerical results ($e_k$ and $r_k$) for \texttt{gravity} by  KM and RKM.\label{fig:grav-kmrkm}}
\end{figure}

\begin{figure}[hbt!]
  \centering
  \setlength{\tabcolsep}{0pt}
  \begin{tabular}{cccc}
   \includegraphics[trim={2.0cm 0 2.3cm 0.2cm},clip,width=.25\textwidth]{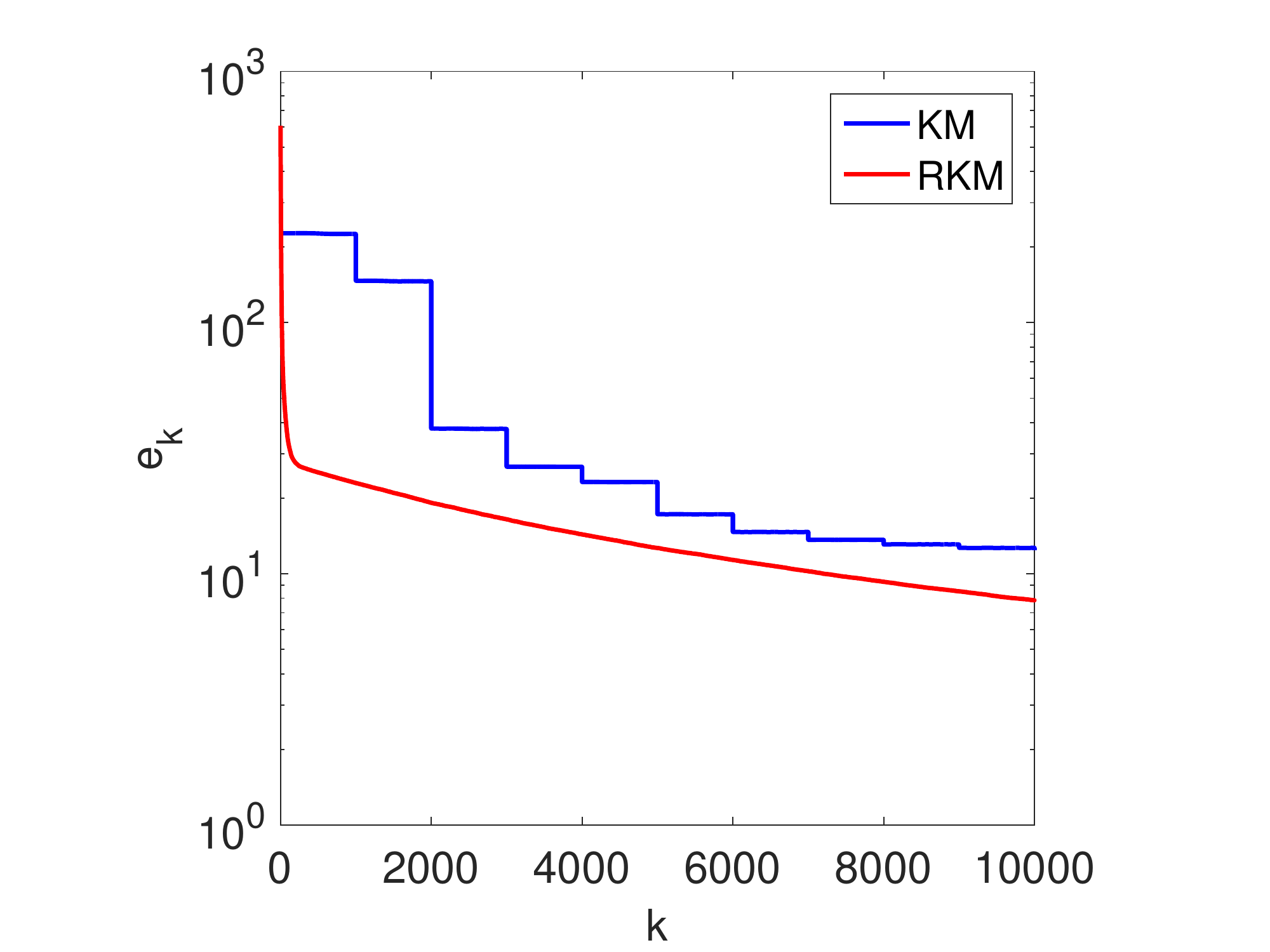}& \includegraphics[trim={2.0cm 0 2.3cm 0.2cm},clip,width=.25\textwidth]{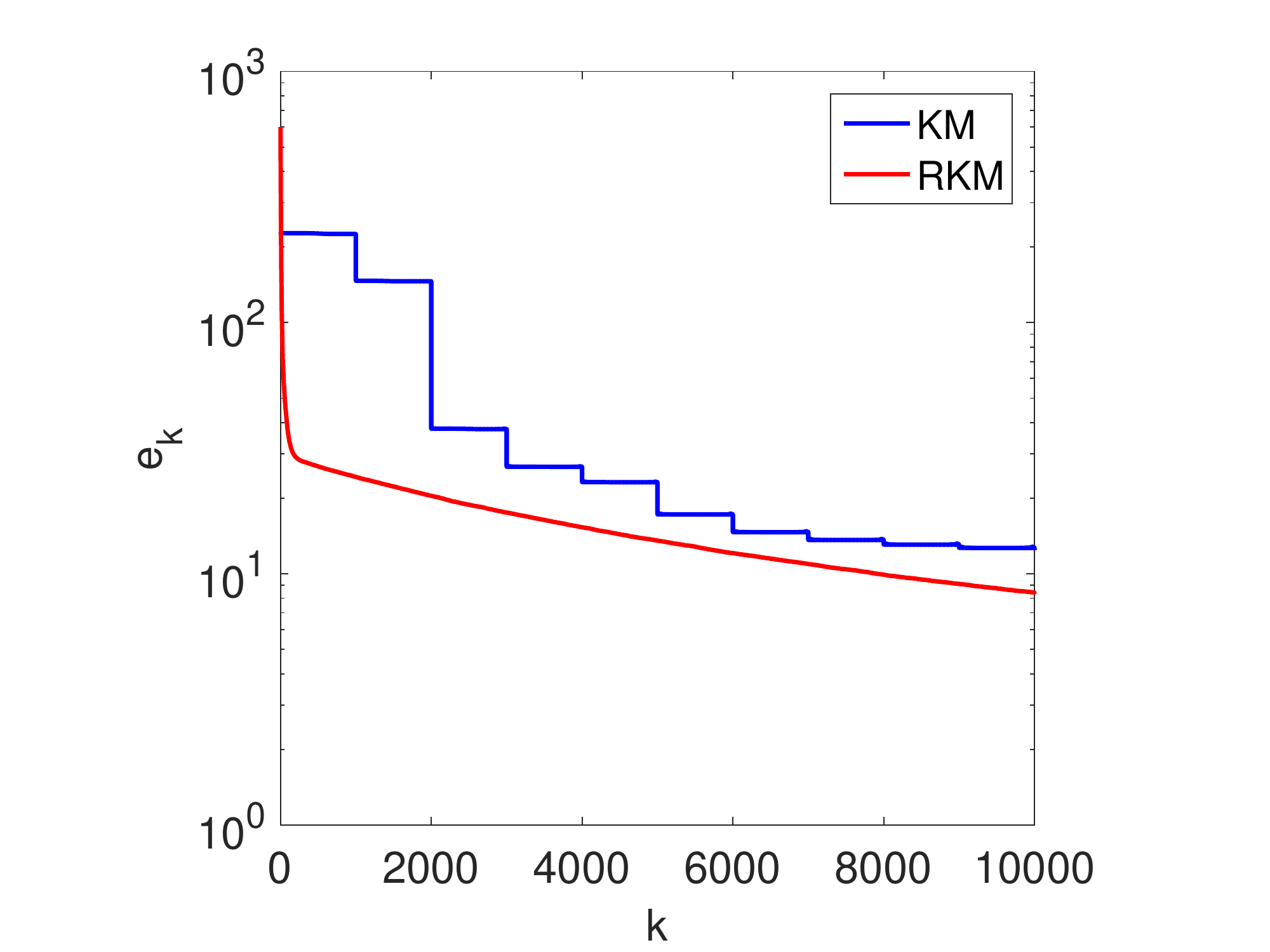}& \includegraphics[trim={2.0cm 0 2.3cm 0.2cm},clip,width=.25\textwidth]{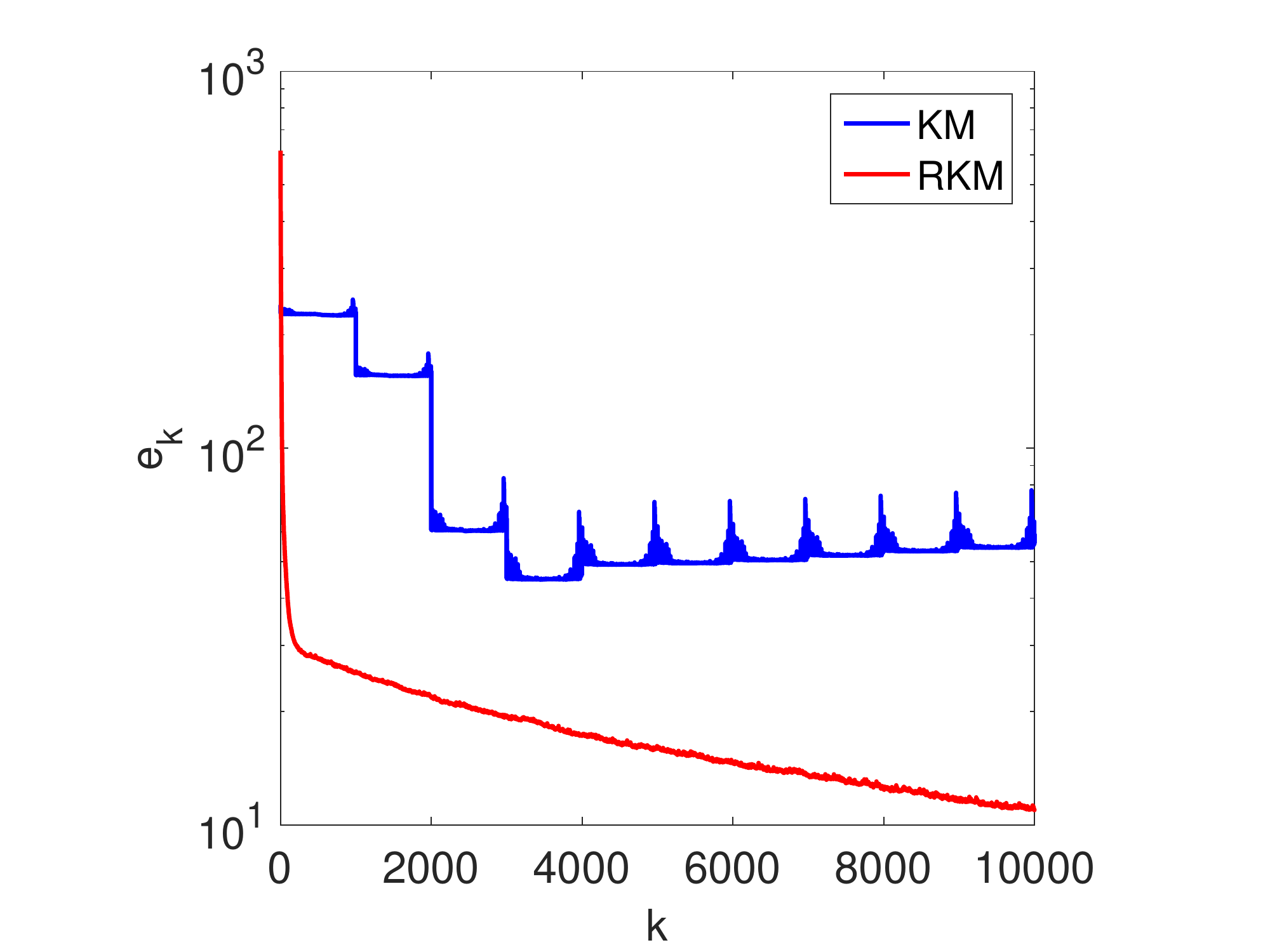}& \includegraphics[trim={2.0cm 0 2.3cm 0.2cm},clip,width=.25\textwidth]{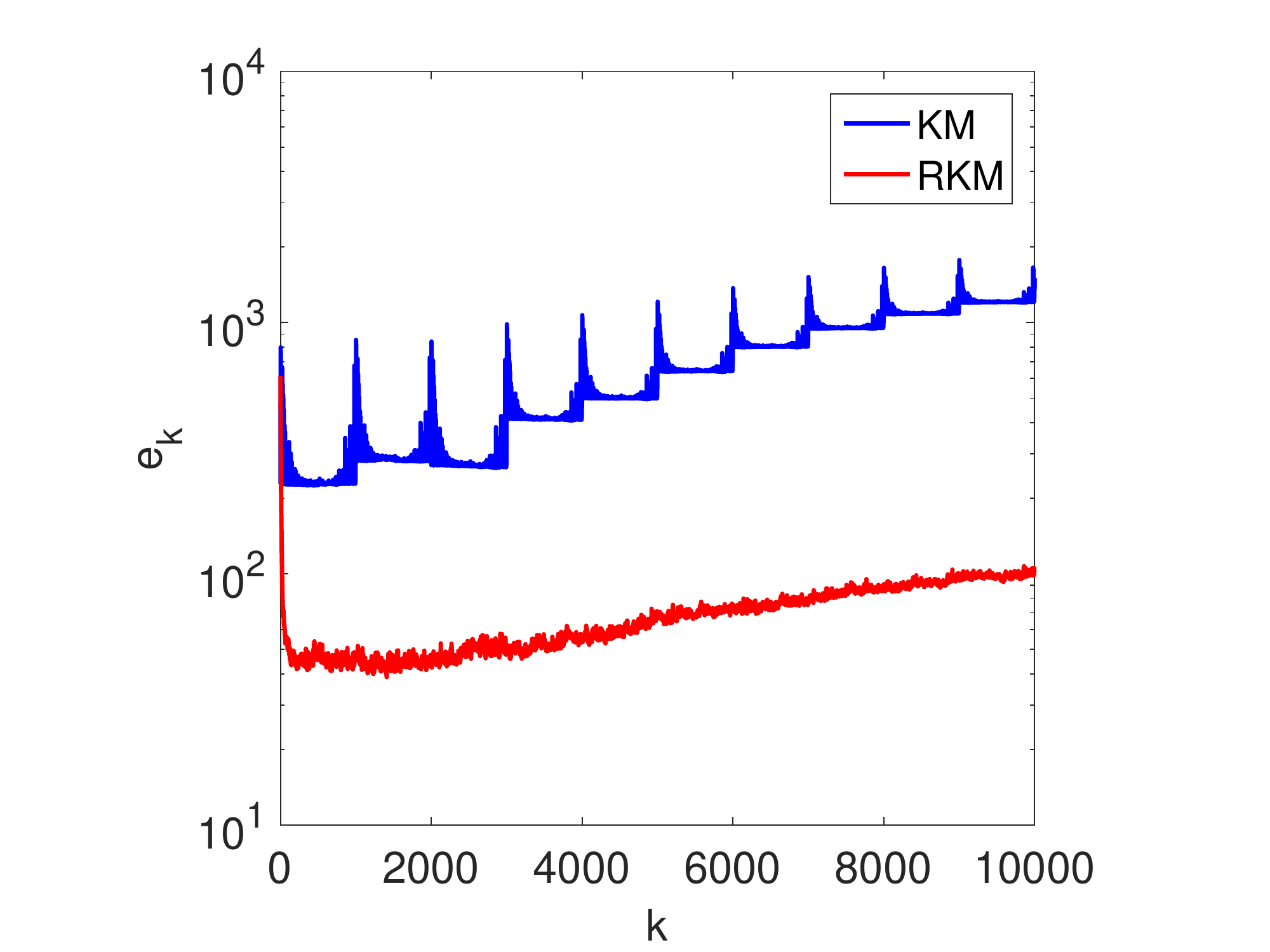}\\
   \includegraphics[trim={2.0cm 0 2.3cm 0.2cm},clip,width=.25\textwidth]{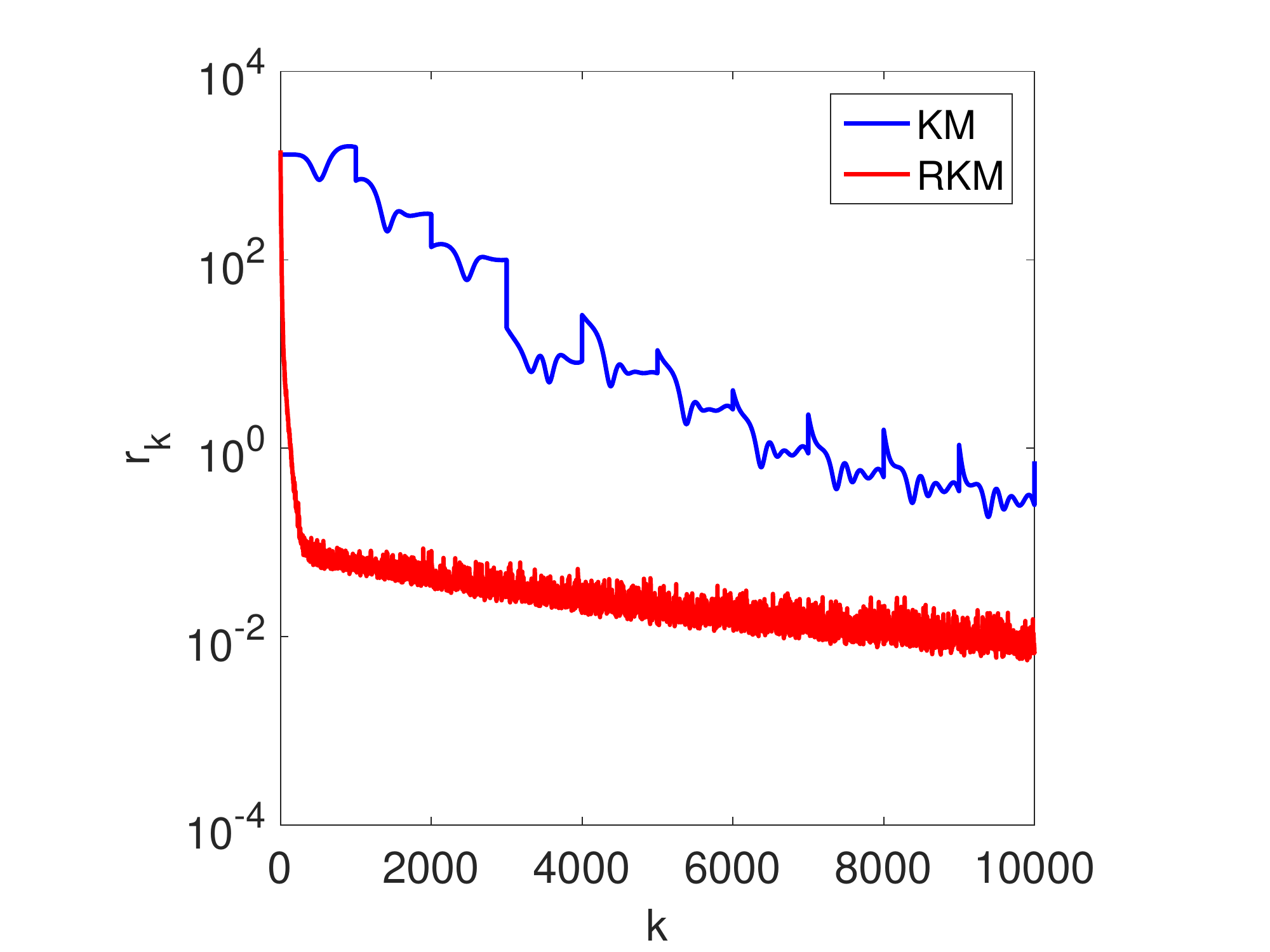}& \includegraphics[trim={2.0cm 0 2.3cm 0.2cm},clip,width=.25\textwidth]{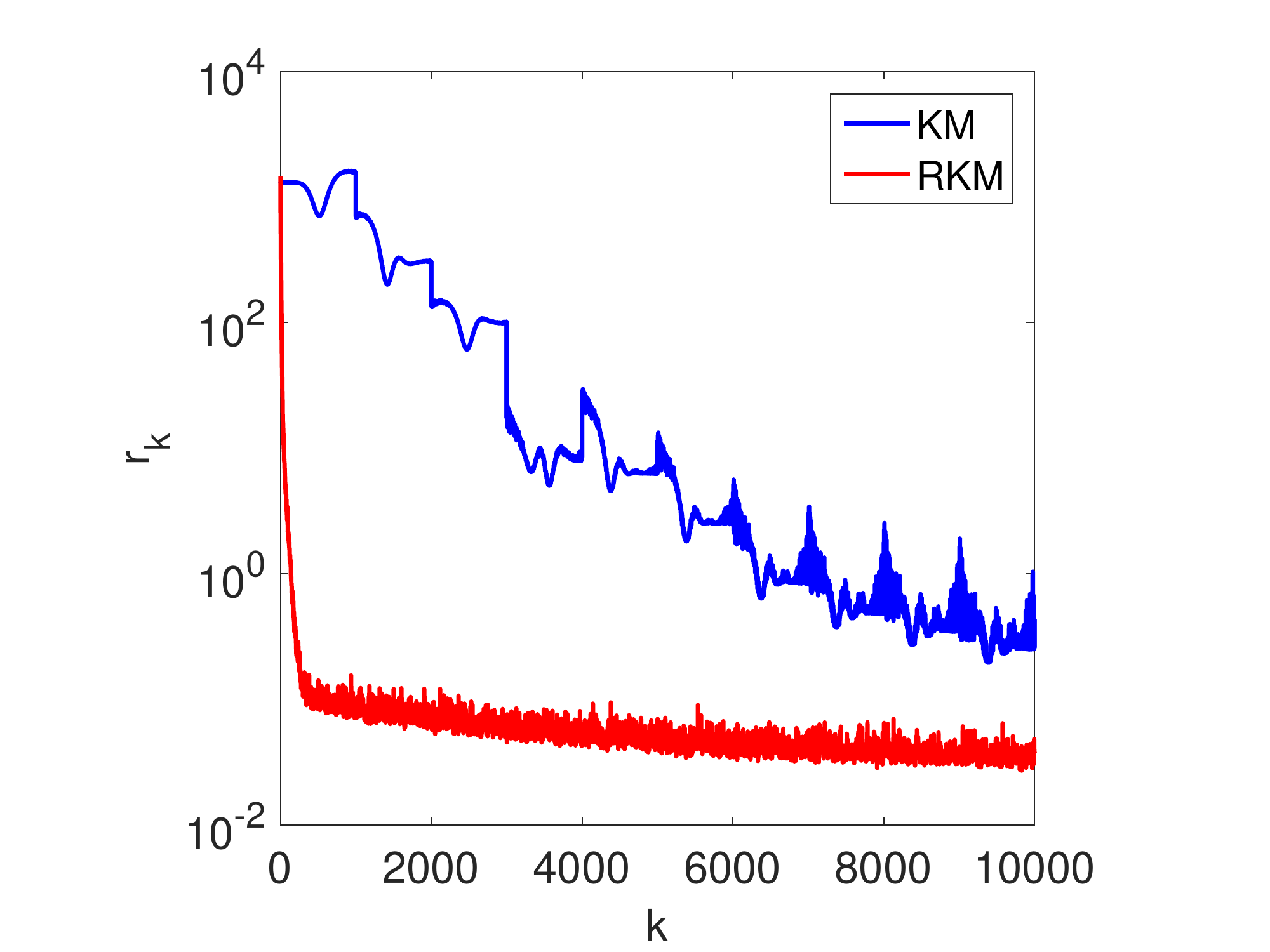}& \includegraphics[trim={2.0cm 0 2.3cm 0.2cm},clip,width=.25\textwidth]{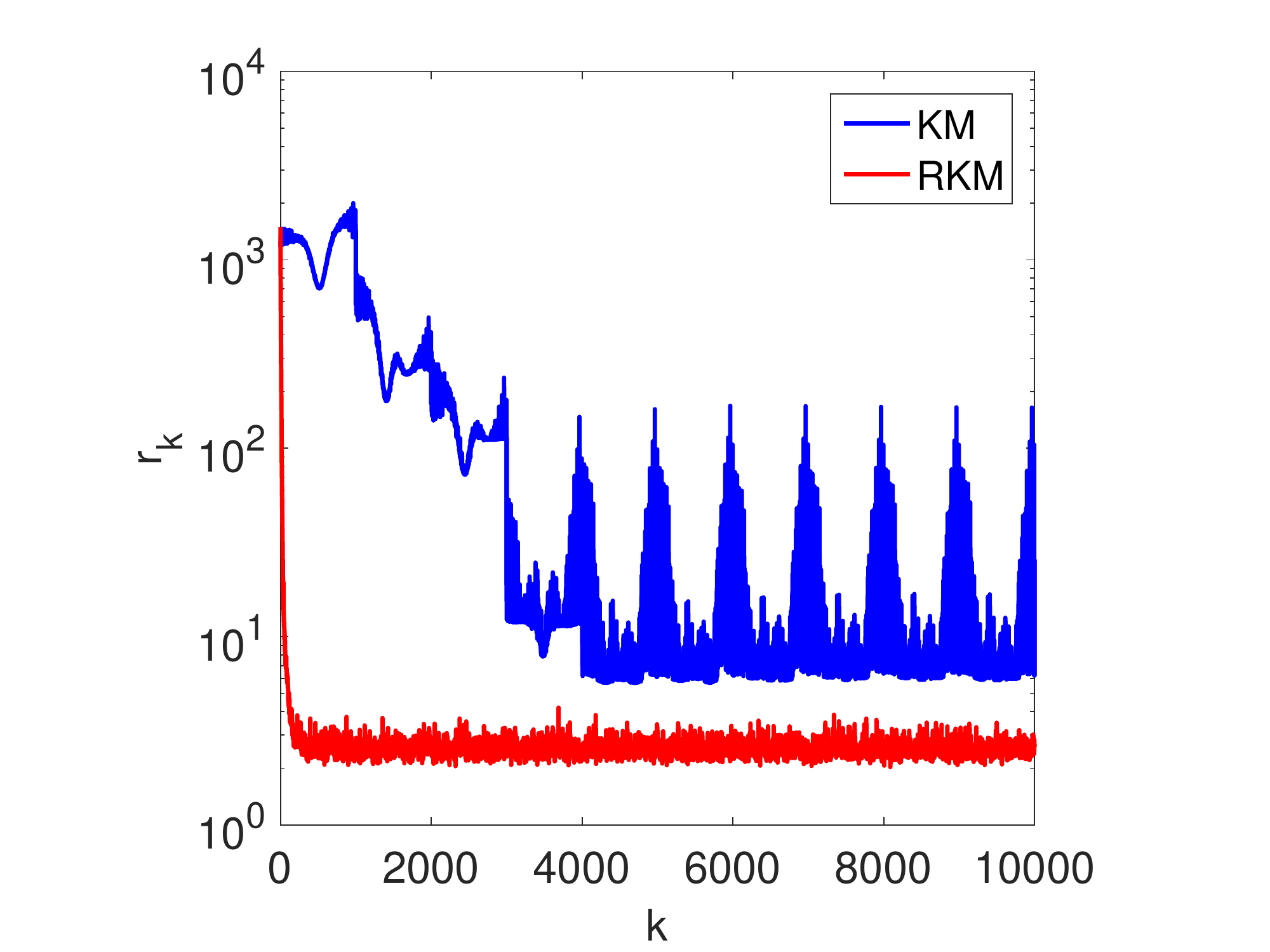}& \includegraphics[trim={2.0cm 0 2.3cm 0.2cm},clip,width=.25\textwidth]{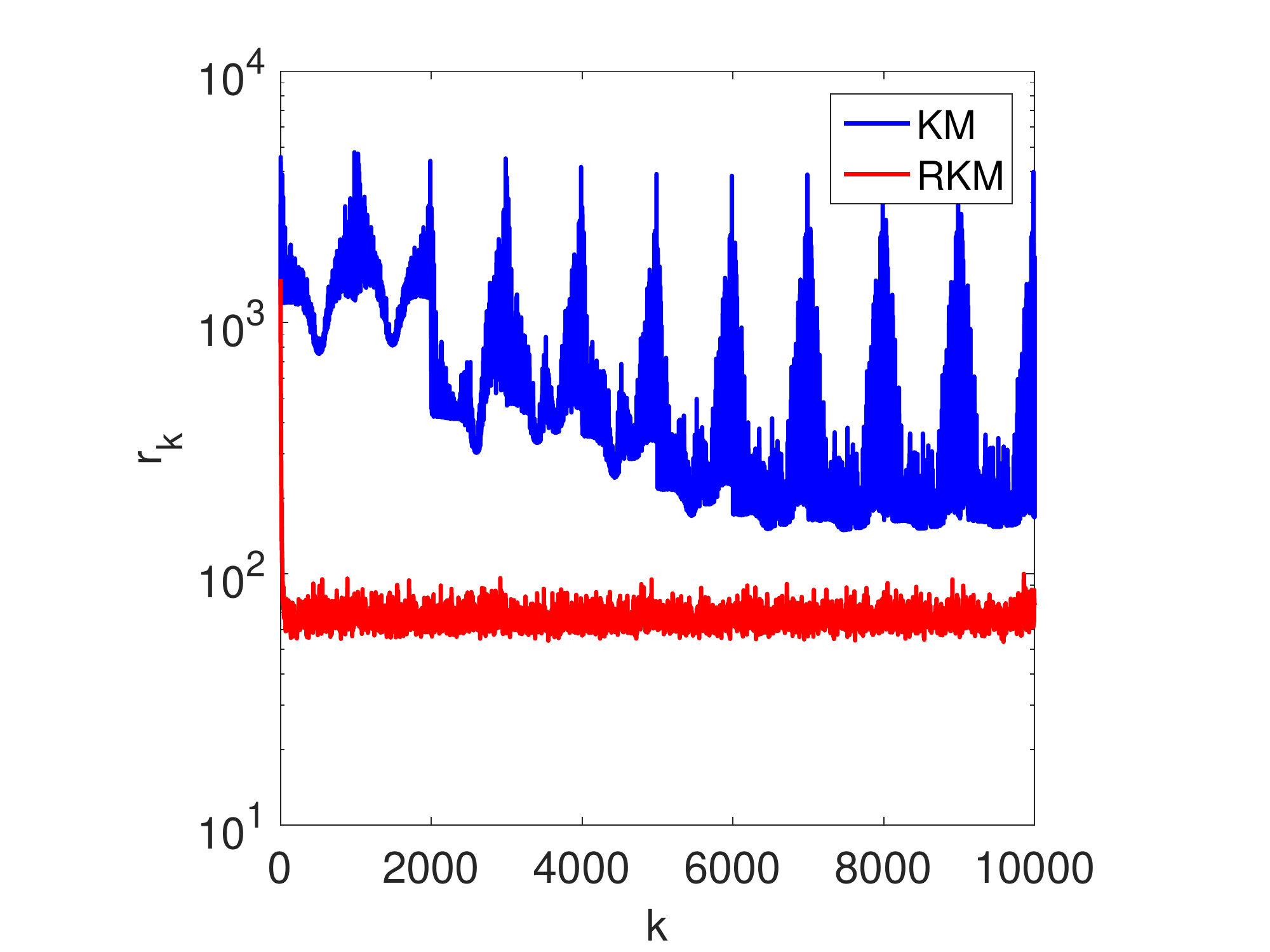}\\
   (a) $\delta=0$ & (b) $\delta=10^{-3}$ & (c) $\delta = 10^{-2}$ & (d) $\delta=5\times10^{-2}$
  \end{tabular}
  \caption{Numerical results ($e_k$ and $r_k$) for \texttt{shaw} by  KM and RKM.\label{fig:shaw-kmrkm}}
\end{figure}

\subsection{Preasymptotic convergence}

Now we examine the convergence of RKM. Theorems \ref{thm:err-exact} and \ref{thm:err-noise} predict that
during first iterations, the low-frequency error $e_L=\E[\|P_Le_k\|^2]$ decreases rapidly, but the
high-frequency error $e_H=\E[\|P_He_k\|^2]$ can at best decay mildly. For all examples, the first five
singular vectors can capture the majority of the energy of the initial error $x^*-x_0$. Thus, we choose
a truncation level $L=5$, and plot the evolution of low-frequency and high-frequency errors $e_L$ and
$e_H$, and the total error $e=\E[\|e_k\|^2]$, in Fig. \ref{fig:decom}.

Numerically, the low-frequency error $e_L$ decays much more rapidly during the initial iterations, and since
the low-frequency modes are dominant, the total error $e$ also enjoys a very fast initial decay. Intuitively,
this behavior may be explained as follows. The rows of the matrix ${A}$ mainly contain low-frequency modes,
and thus each RKM iteration tends to mostly decrease the low-frequency error $e_L$ of the initial error
$x^*-x_0$. The high-frequency error $e_H$ experiences a similar but slower decay during the iteration, and
then levels off. These observations fully confirm the preasymptotic analysis in Section \ref{sec:conv}. For
noisy data, the error $e_k$ can be highly oscillating, so is the residual $r_k$. The larger is the noise
level $\delta$, the larger is the oscillation magnitude. However, the degree of ill-posedness of the problem
seems not to affect the convergence of RKM, so long as $x^*$ is mainly composed of low-frequency modes.

\begin{figure}[hbt!]
  \centering
  \setlength{\tabcolsep}{0pt}
  \begin{tabular}{ccc}
    \includegraphics[trim={2cm 0 2cm 0},clip,width=.33\textwidth]{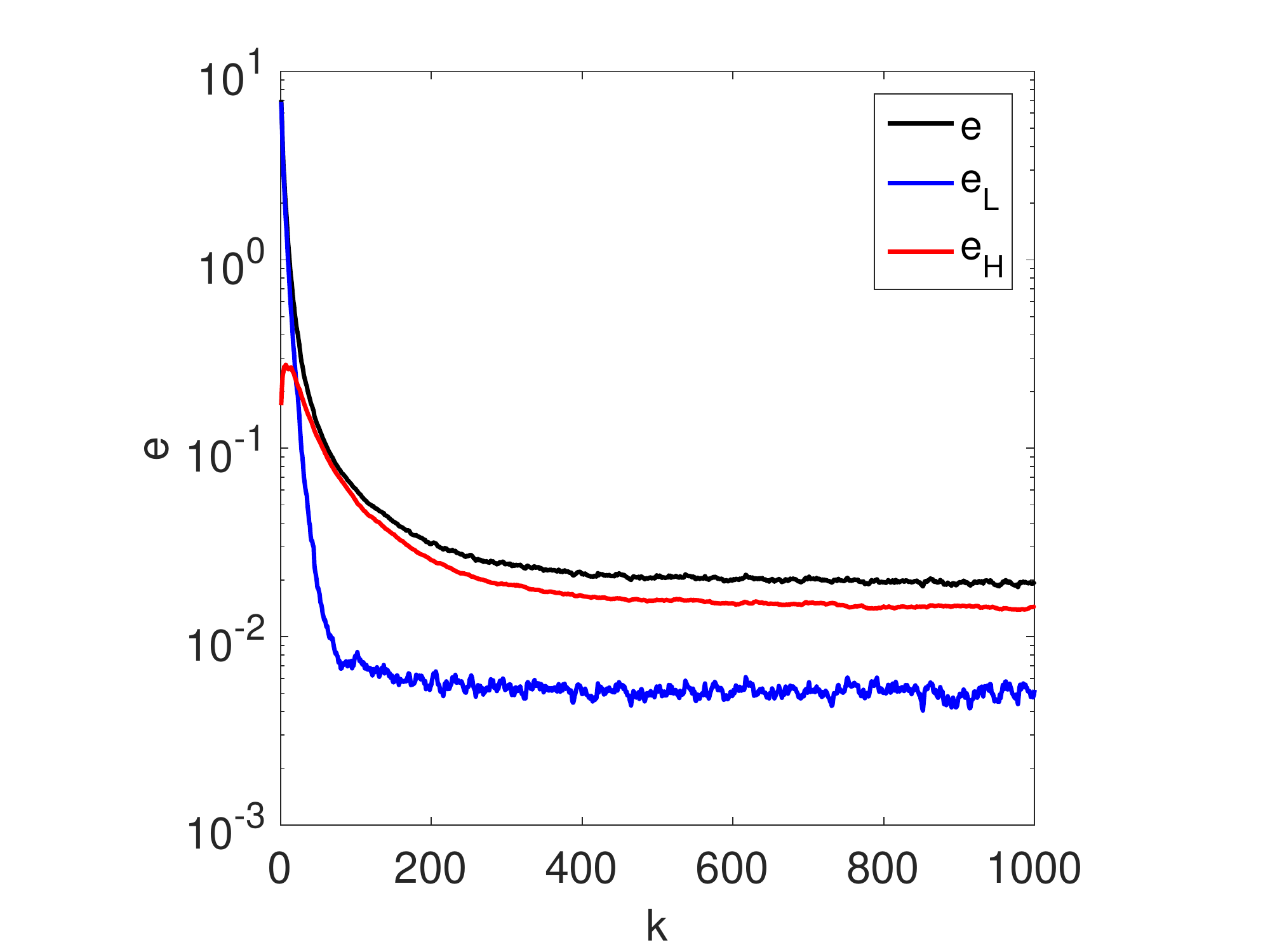} & \includegraphics[trim={2cm 0 2cm 0},clip,width=.33\textwidth]{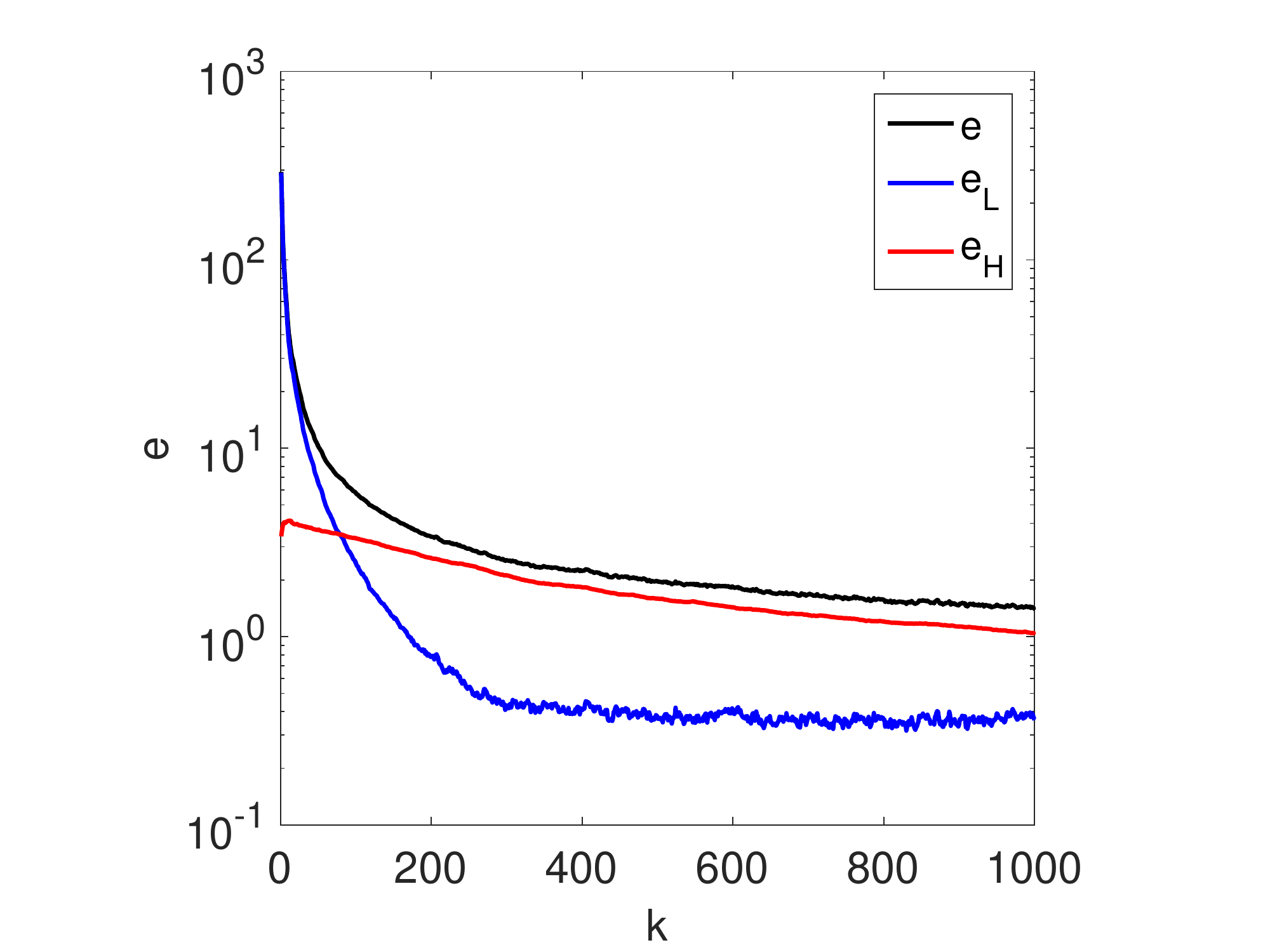}
   &\includegraphics[trim={2cm 0 2cm 0},clip,width=.33\textwidth]{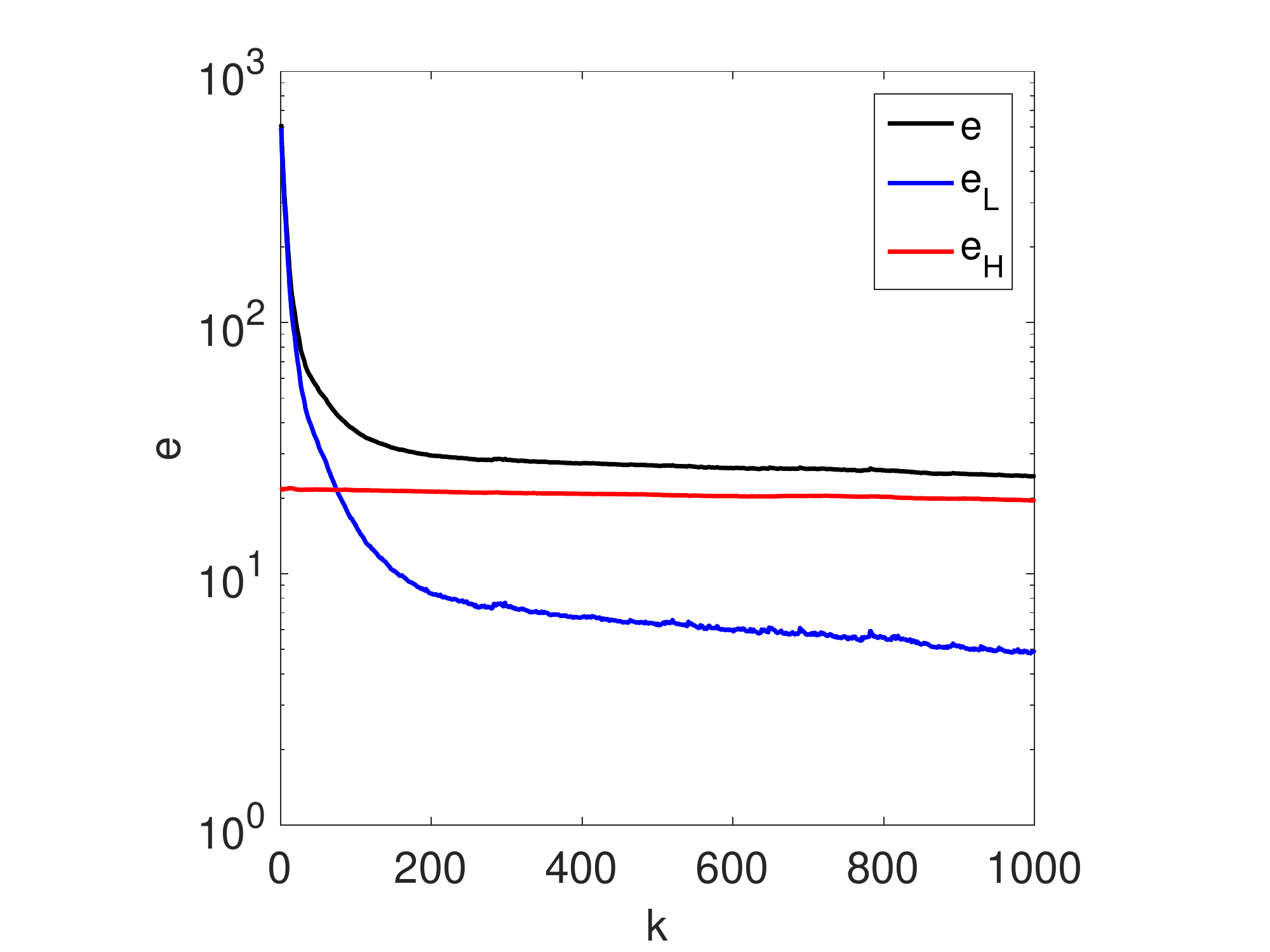}\\
   \includegraphics[trim={2cm 0 2cm 0},clip,width=.33\textwidth]{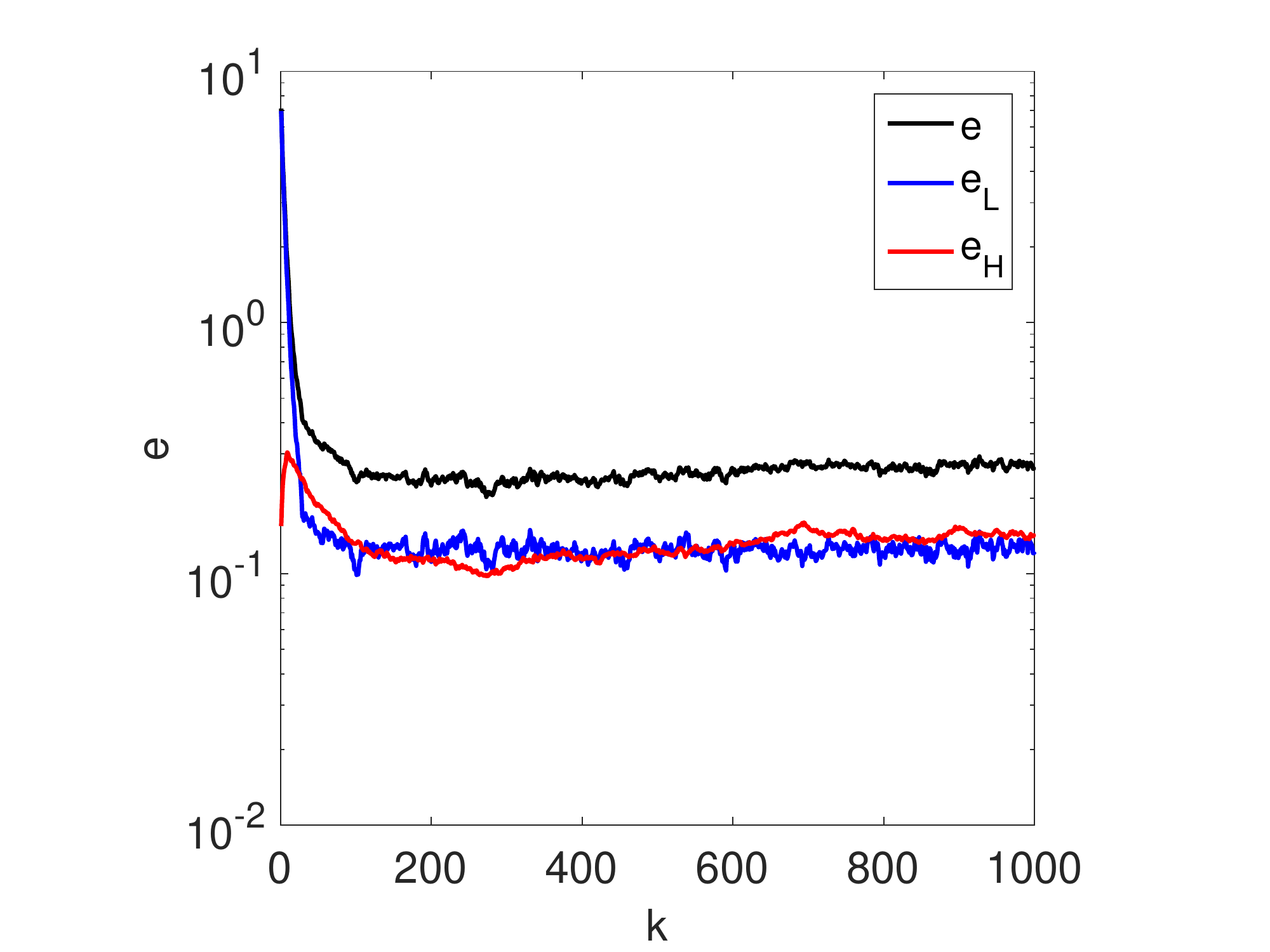} & \includegraphics[trim={2cm 0 2cm 0},clip,width=.33\textwidth]{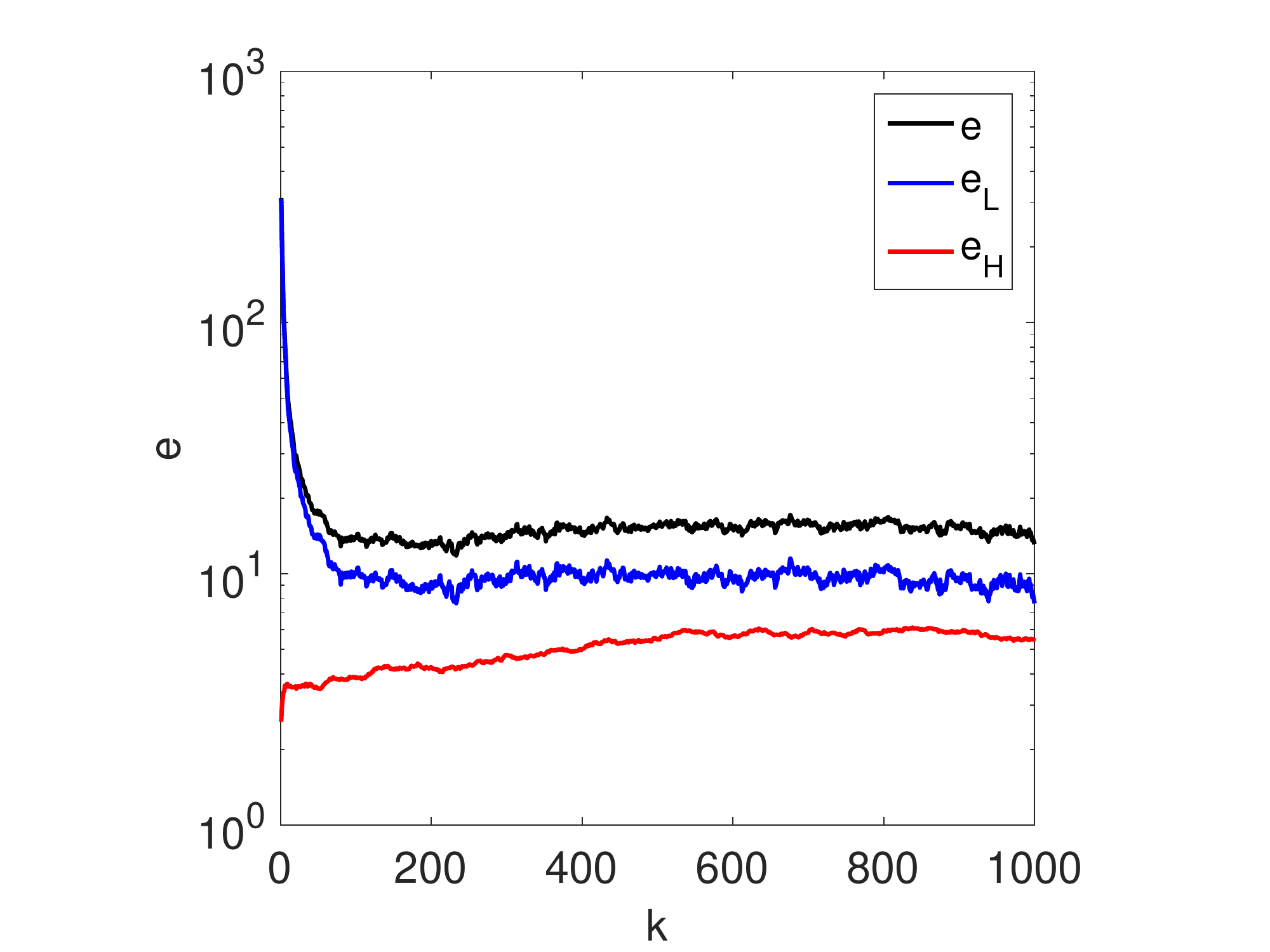}
   &\includegraphics[trim={2cm 0 2cm 0},clip,width=.33\textwidth]{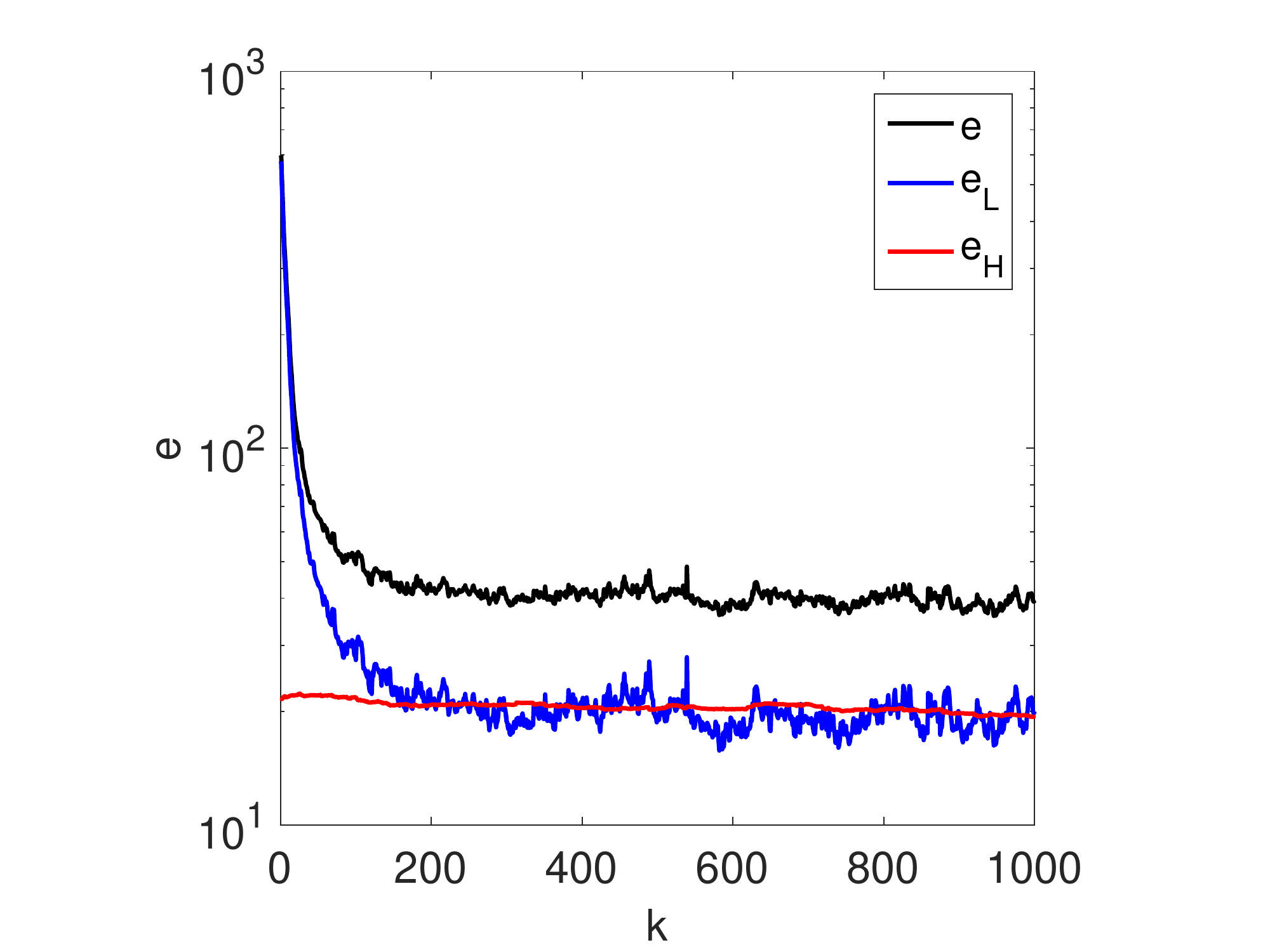}\\
      (a) \texttt{phillips} & (b) \texttt{gravity} & (c) \texttt{shaw}
  \end{tabular}
  \caption{The error decay for the examples with two noise levels: $\delta=10^{-2}$ (top) and $\delta = 5\times 10^{-2}$ (bottom), with a truncation level $L=5$.
  \label{fig:decom}}
\end{figure}

To shed further insights, we present in Fig. \ref{fig:decom-rnd} the decay behavior of the low- and high-frequency
errors for the example \texttt{phillips} with a random solution whose entries follow the i.i.d. standard normal
distribution. Then the source type condition is not verified for the initial error. Now with a truncation level
$L=5$, the low-frequency error $e_L$ only composes a small fractional of the initial error $e_0$. The low-frequency
error $e_L$ decays rapidly, exhibiting a fast preasymptotic convergence as predicted by Theorem \ref{thm:err-noise},
but the high-frequency error $e_H$ stagnates during the iteration. Thus, in the absence of the smoothness
condition on $e_0$, RKM is ineffective, thereby supporting Theorems \ref{thm:err-exact} and \ref{thm:err-noise}.

\begin{figure}[hbt!]
  \centering
  \setlength{\tabcolsep}{0pt}
  \begin{tabular}{cc}
    \includegraphics[trim={2cm 0 2cm 0},clip,width=.33\textwidth]{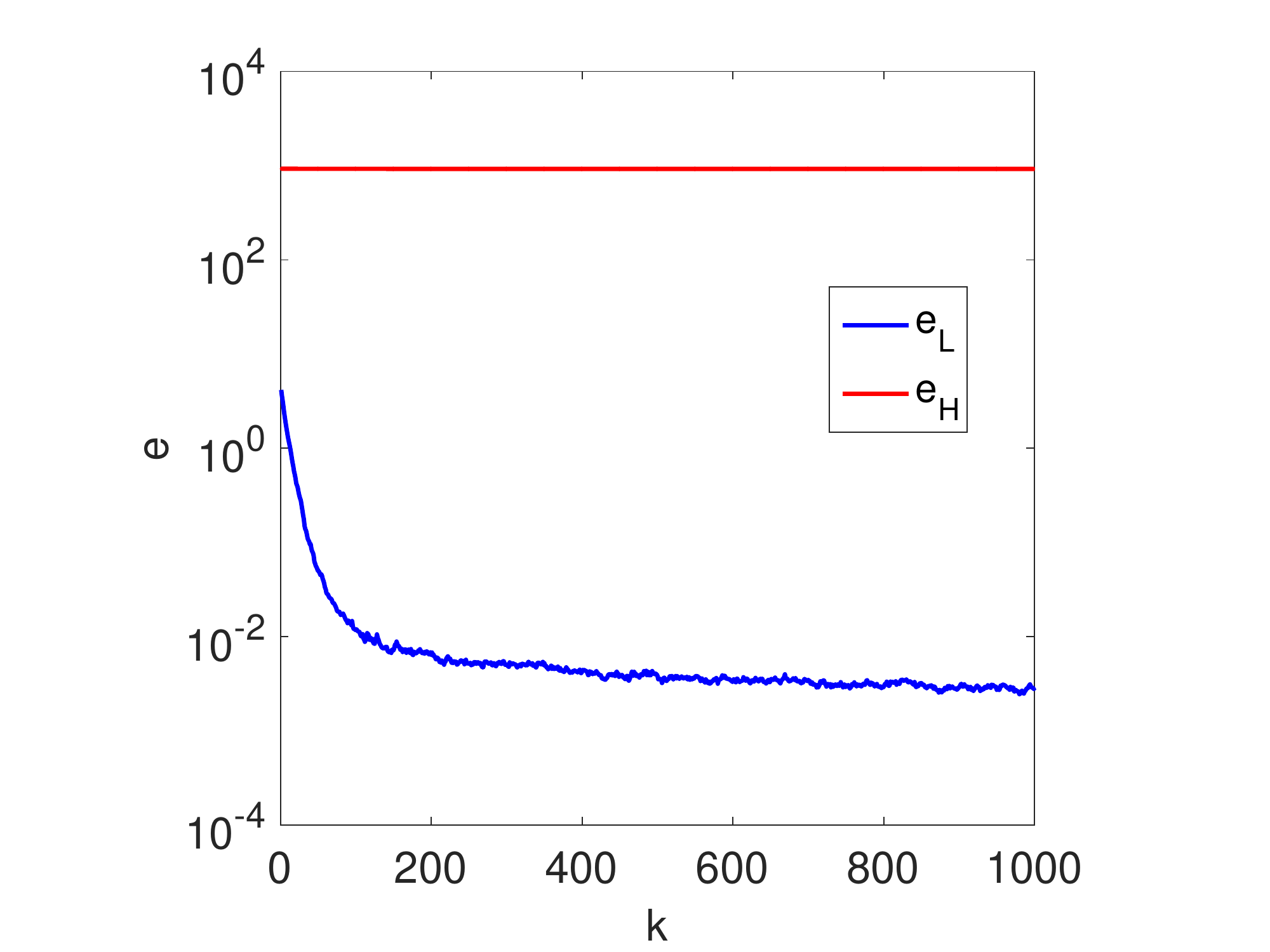} & \includegraphics[trim={2cm 0 2cm 0},clip,width=.33\textwidth]{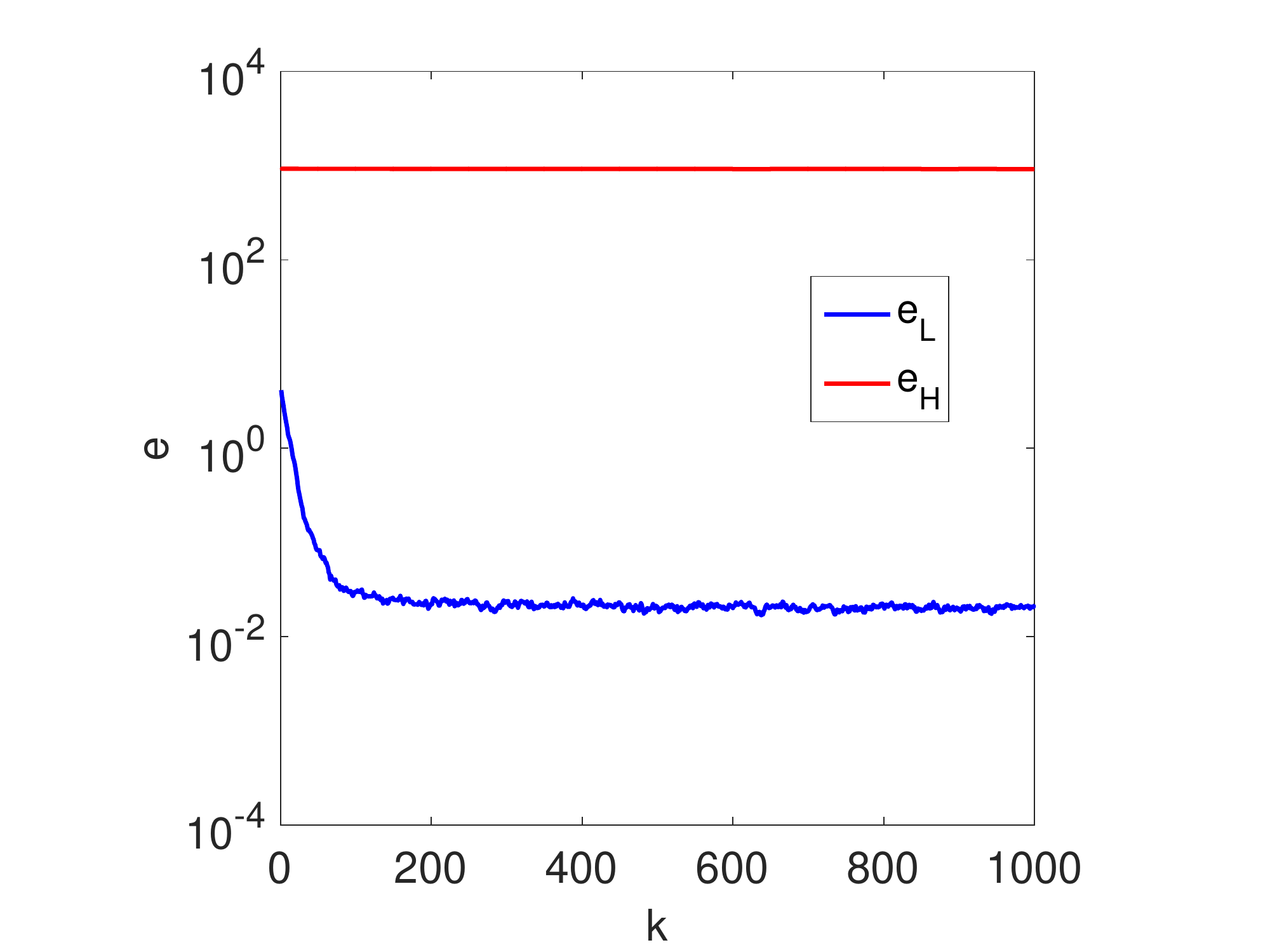}\\
     (a) $\delta=10^{-2}$ & (b) $\delta=5\times10^{-2}$
\end{tabular}
  \caption{The error decay for \texttt{phillips} with a random solution, with a truncation level $L=5$.
  \label{fig:decom-rnd}}
\end{figure}

Naturally, one may divide the total error $e$ into more than two frequency bands. The empirical behavior is
similar to the case of two frequency bands; see Fig. \ref{fig:phil-decom_multi} for an illustration on
the example \texttt{phillips}, with four frequency bands. The lowest-frequency error $e_1$ decreases fastest,
and then the next band $e_2$ slightly slower, etc. These observations clearly indicate that even though
RKM does not employ the full gradient, the iterates are still mainly concerned with the low-frequency modes
during the first iterations, like the Landweber method in the sense that the low-frequency modes  are much
easier to recover than the high-frequency ones. However, the cost of each RKM iteration is only one $n$th
of that for the Landweber method, and thus it is computationally much more efficient.

\begin{figure}[hbt!]
  \centering
  \setlength{\tabcolsep}{0pt}
  \begin{tabular}{cc}
   \includegraphics[trim={2cm 0 2cm 0},clip,width=.33\textwidth]{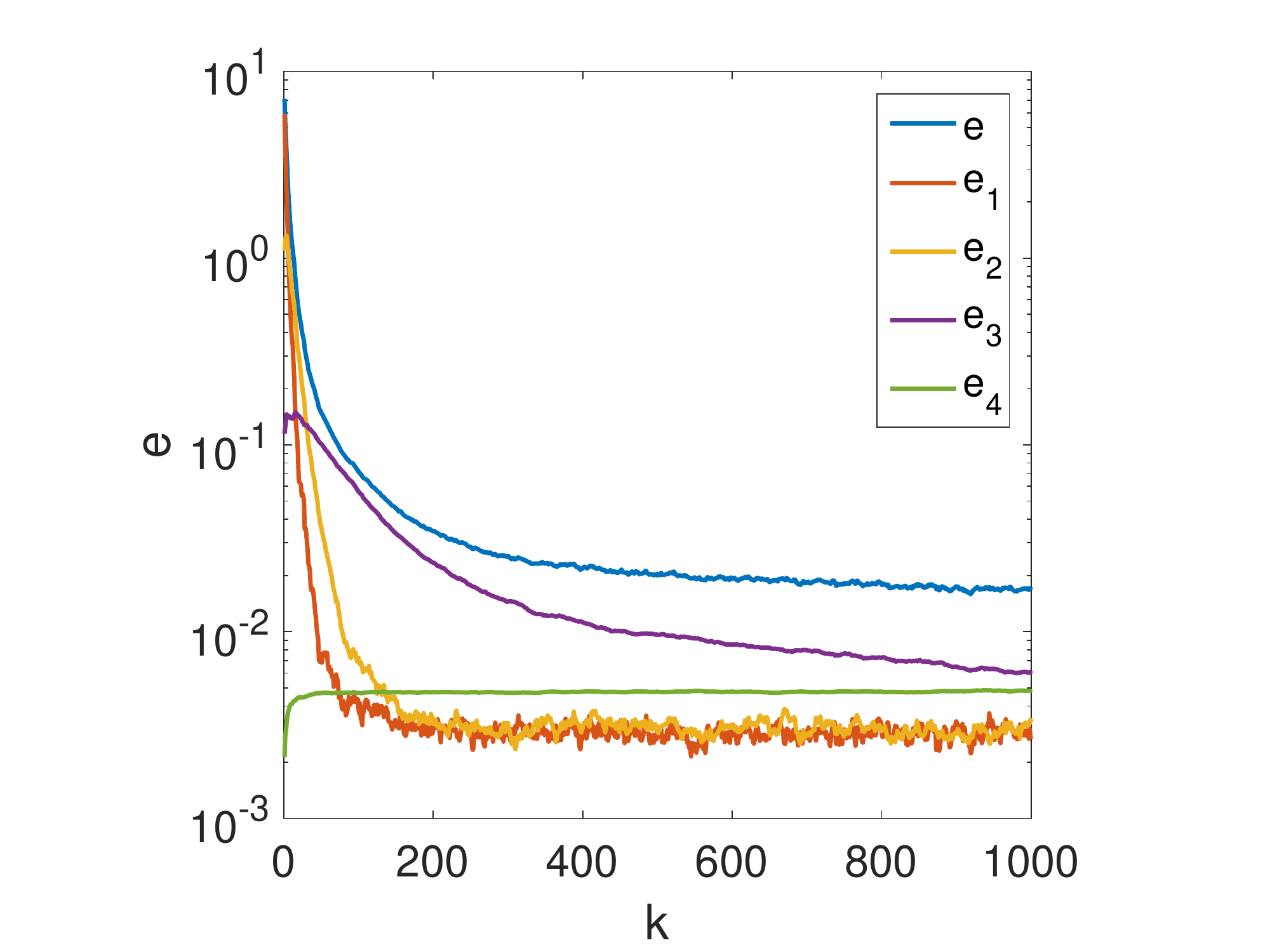} & \includegraphics[trim={2cm 0 2cm 0},clip,width=.33\textwidth]{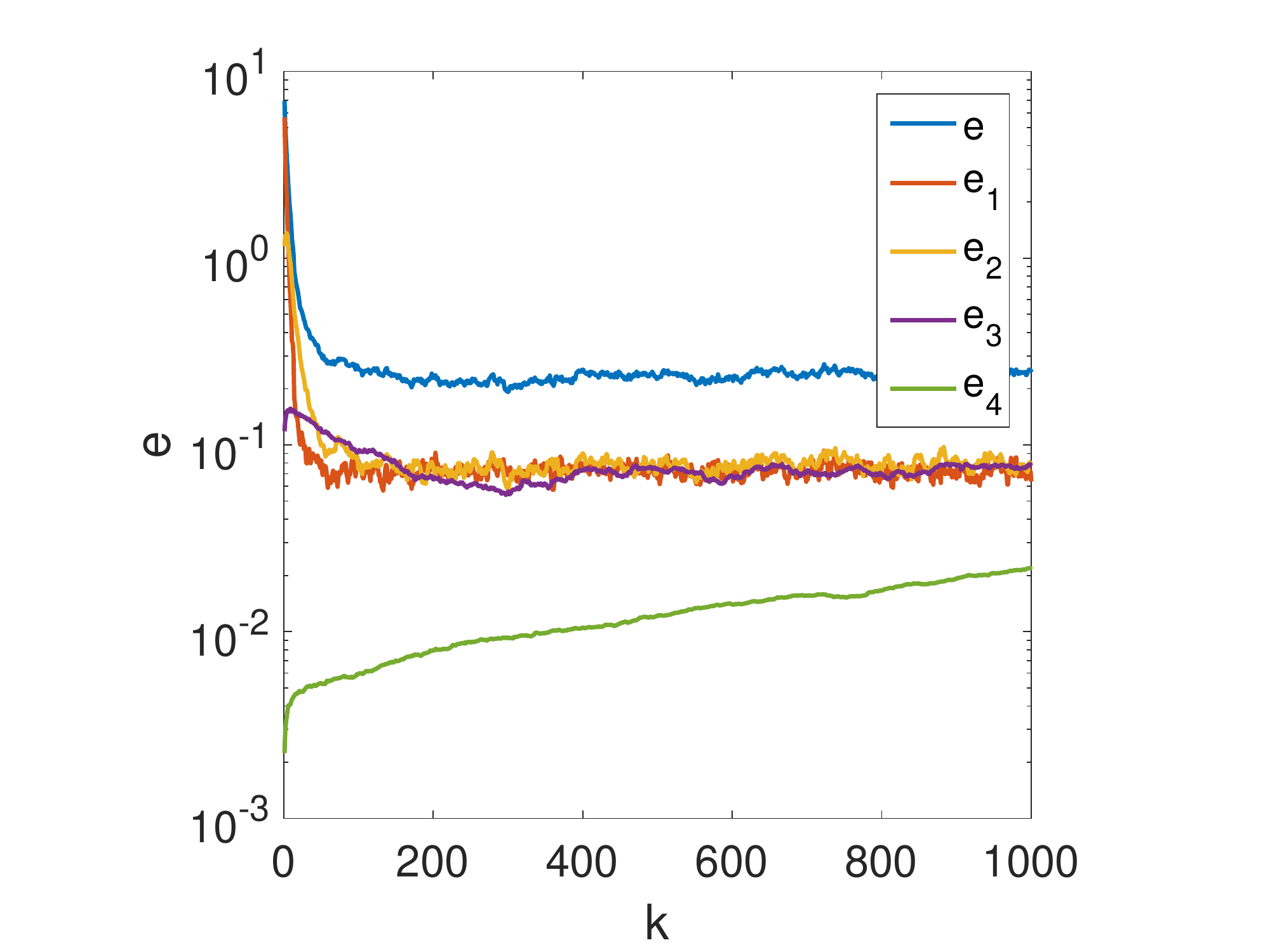}\\
   (a) $\delta=10^{-2}$ & (b) $\delta=5\times10^{-2}$
  \end{tabular}
  \caption{The error decay for \texttt{phillips}. The total error $e$ is
  divided into four frequency bands: 1-3, 4-6, 7-9, and the remaining,
  denoted by $e_i$, $i=1,\ldots,4$.  \label{fig:phil-decom_multi}}
\end{figure}

\subsection{RKM versus RKMVR}

The nonvanishing variance of the gradient $g_i(x)$ slows down the asymptotic convergence of RKM, and
the iterates eventually tend to oscillate wildly in the presence of data noise, cf. the discussion in
Section \ref{sec:implement}. This is expected: the iterate converges to the least-squares
solution, which is known to be highly oscillatory for ill-posed inverse problems. Variance reduction
is one natural strategy to decrease the variance of the gradient estimate, thereby stabilizing the
evolution of the iterates. To illustrate this, we compare the evolution of RKM with RKMVR in Fig.
\ref{fig:rkmvr}. We also include the results by the Landweber method (LM). To compare the iteration
complexity only, we count one Landweber iteration as $n$ RKM iterates. The epoch of RKMVR is set to
$n$, the total number of data points, as suggested in \cite{JohnsonZhang:2013}. Thus $n$ RKMVR
iterates include one full gradient evaluation, and it amounts to $2n$ RKM iterates. The full gradient
evaluation is indicated by flat segments in the plots.

With the increase of the noise level $\delta$, RKM first decreases the error $e_k$, and then increases
it, which is especially pronounced at $\delta = 5\times 10^{-2}$. This is well reflected by the large
oscillations of the iterates. RKMVR tends to stabilize the iteration greatly by removing the large
oscillations, and thus its asymptotical behavior resembles closely that of LM. That is, RKMVR inherits
the good stability of LM, while retaining the fast initial convergence of RKM. Thus, the stopping
criterion, though still needed, is less critical for the RKMVR, which is very beneficial from the
practical point of view. In summary, the simple variance reduction scheme in Algorithm \ref{alg:rkm-vr}
can combine the strengths of both worlds.

Last, we numerically examine the regularizing property of RKMVR with the discrepancy principle \eqref{eqn:dp}.
In Fig. \ref{fig:rkmvrdp}, we present the number of iterations for several noise levels for RKMVR (one
realization) and LM. For both methods, the number of iterations by the discrepancy principle \eqref{eqn:dp}
appears to decrease with the noise level $\delta$, and RKMVR consistently terminates much earlier than LM,
indicating the efficiency of RKMVR. The reconstructions in Fig. \ref{fig:rkmvrdp}(d) show that the error
increases with the noise level $\delta$, indicating a regularizing property. In contrast, in the absence
of the discrepancy principle, the RKMVR iterates eventually diverge as the iteration proceeds, cf. Fig. \ref{fig:rkmvr}.

\begin{figure}[hbt!]
  \centering
  \setlength{\tabcolsep}{0pt}
  \begin{tabular}{cc}
   \includegraphics[trim={2cm 0 2cm 0},clip,width=.25\textwidth]{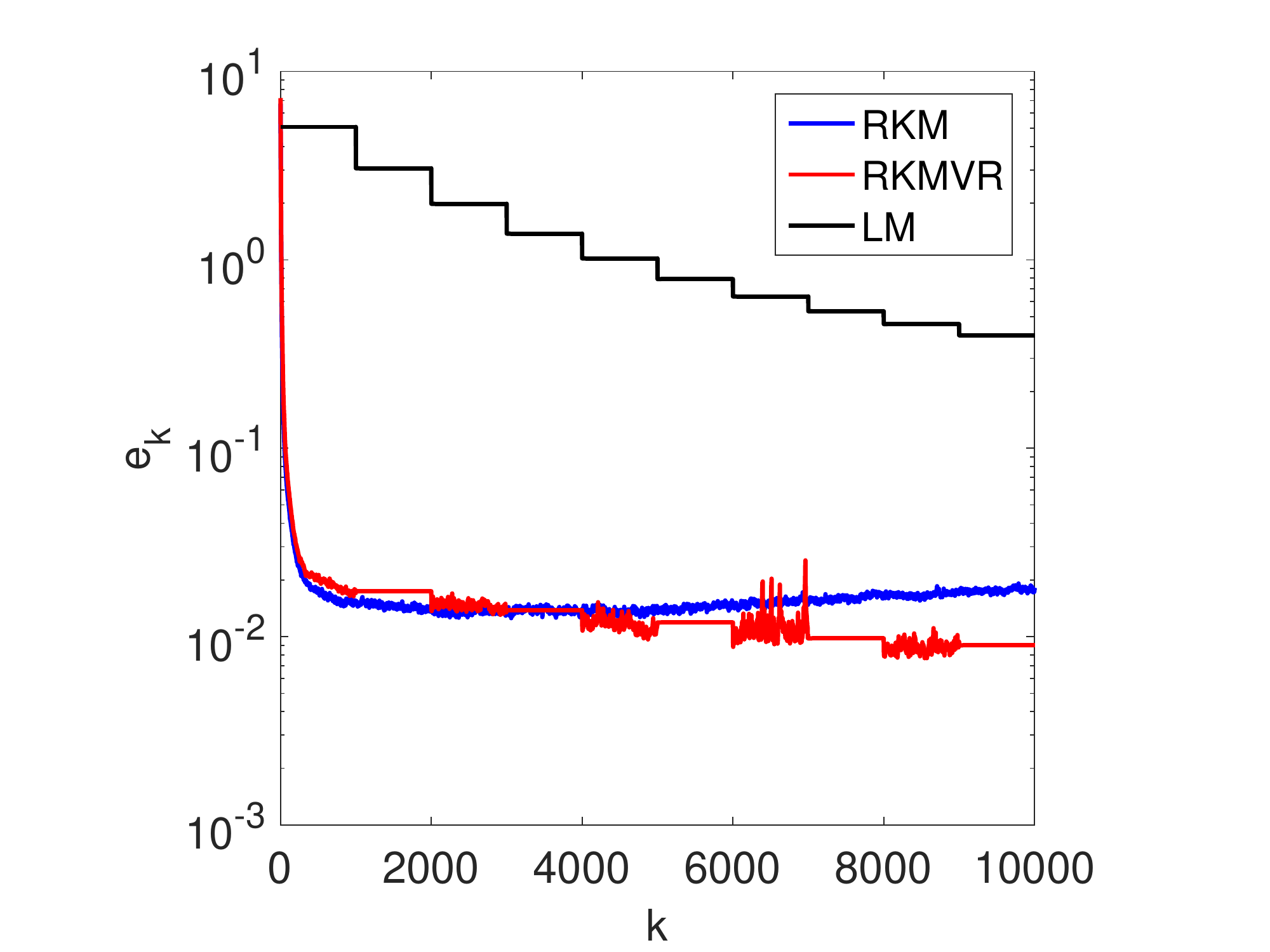} \includegraphics[trim={2cm 0 2cm 0},clip,width=.25\textwidth]{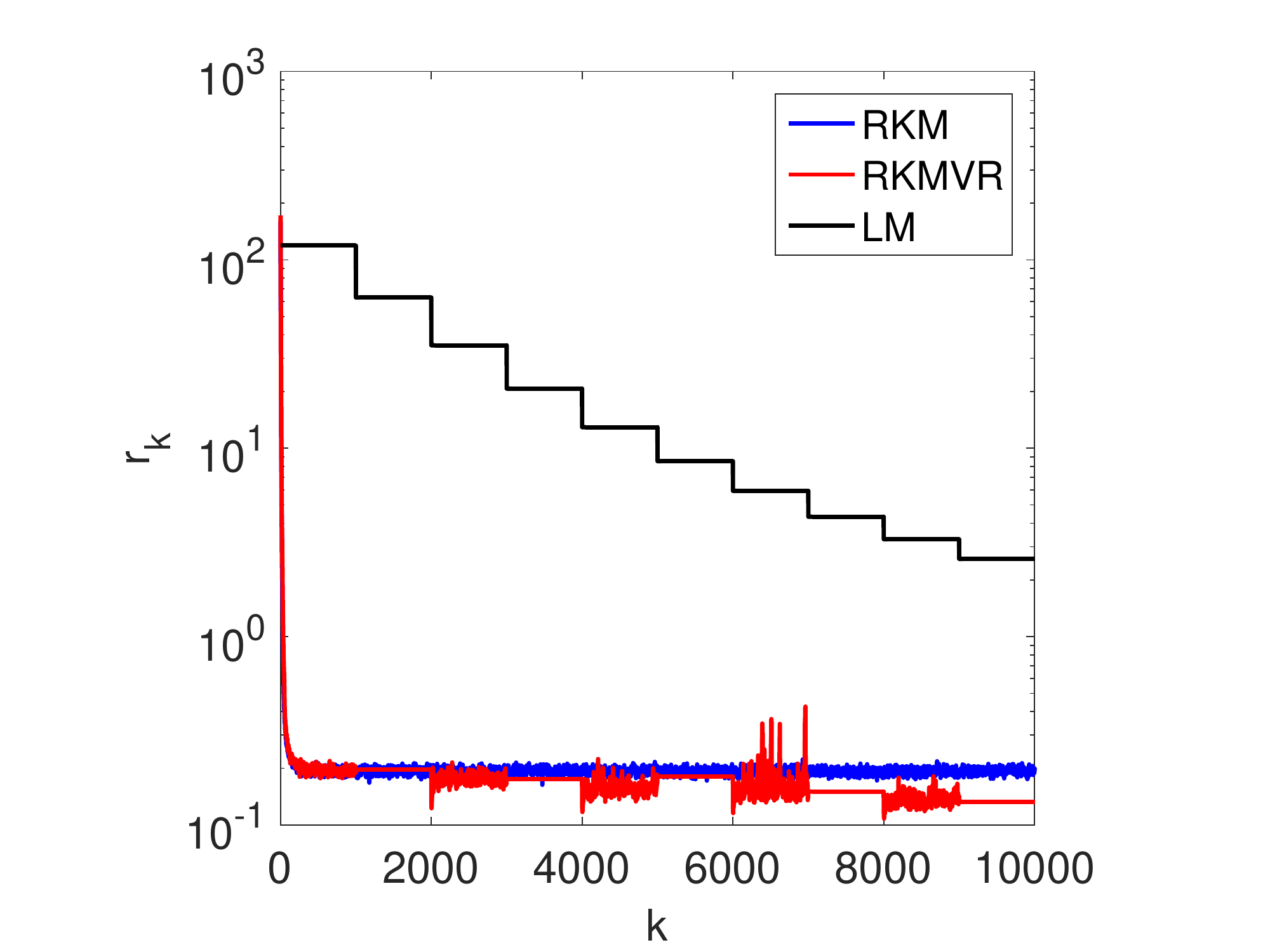}& \includegraphics[trim={2cm 0 2cm 0},clip,width=.25\textwidth]{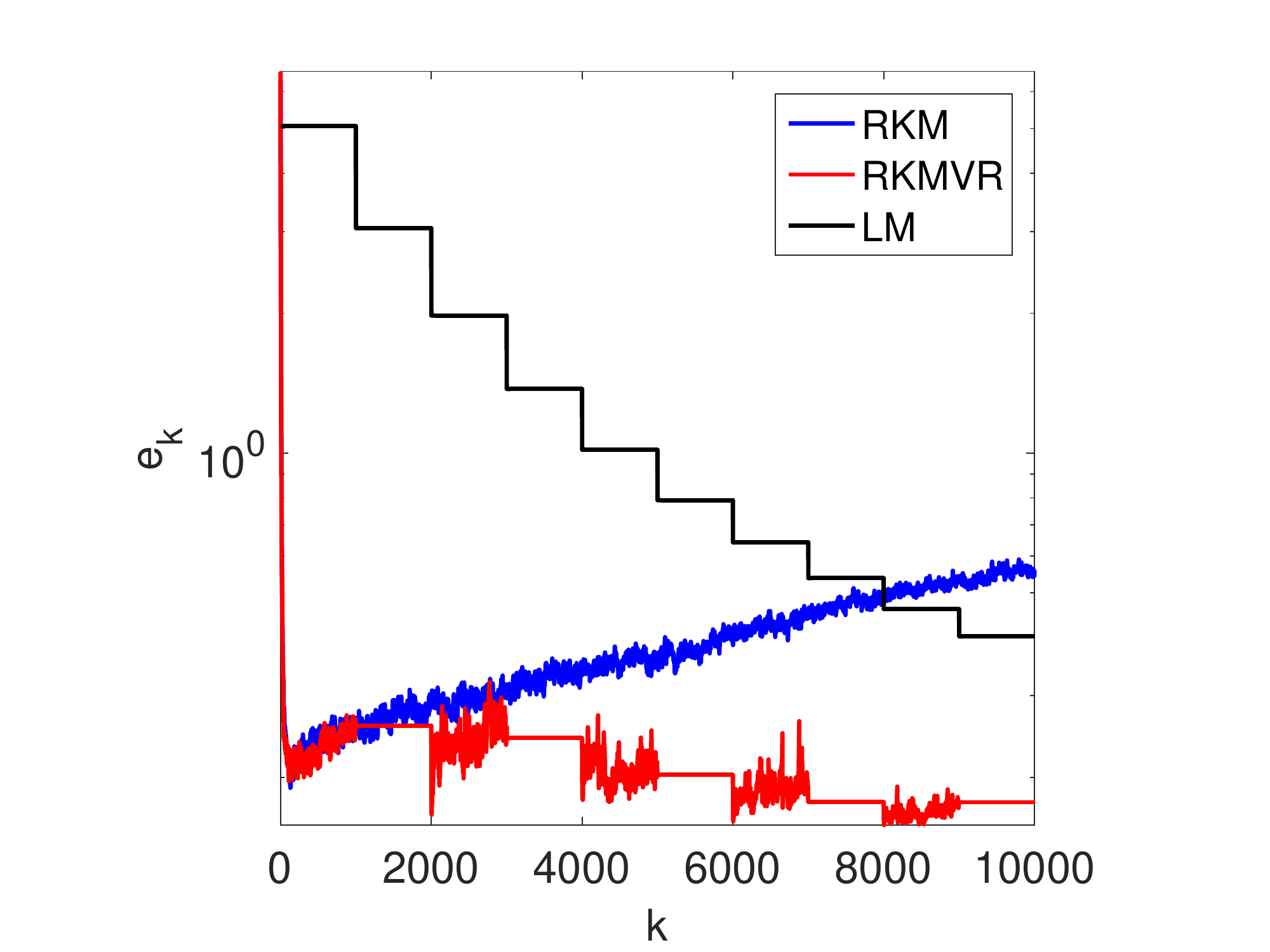} \includegraphics[trim={2cm 0 2cm 0},clip,width=.25\textwidth]{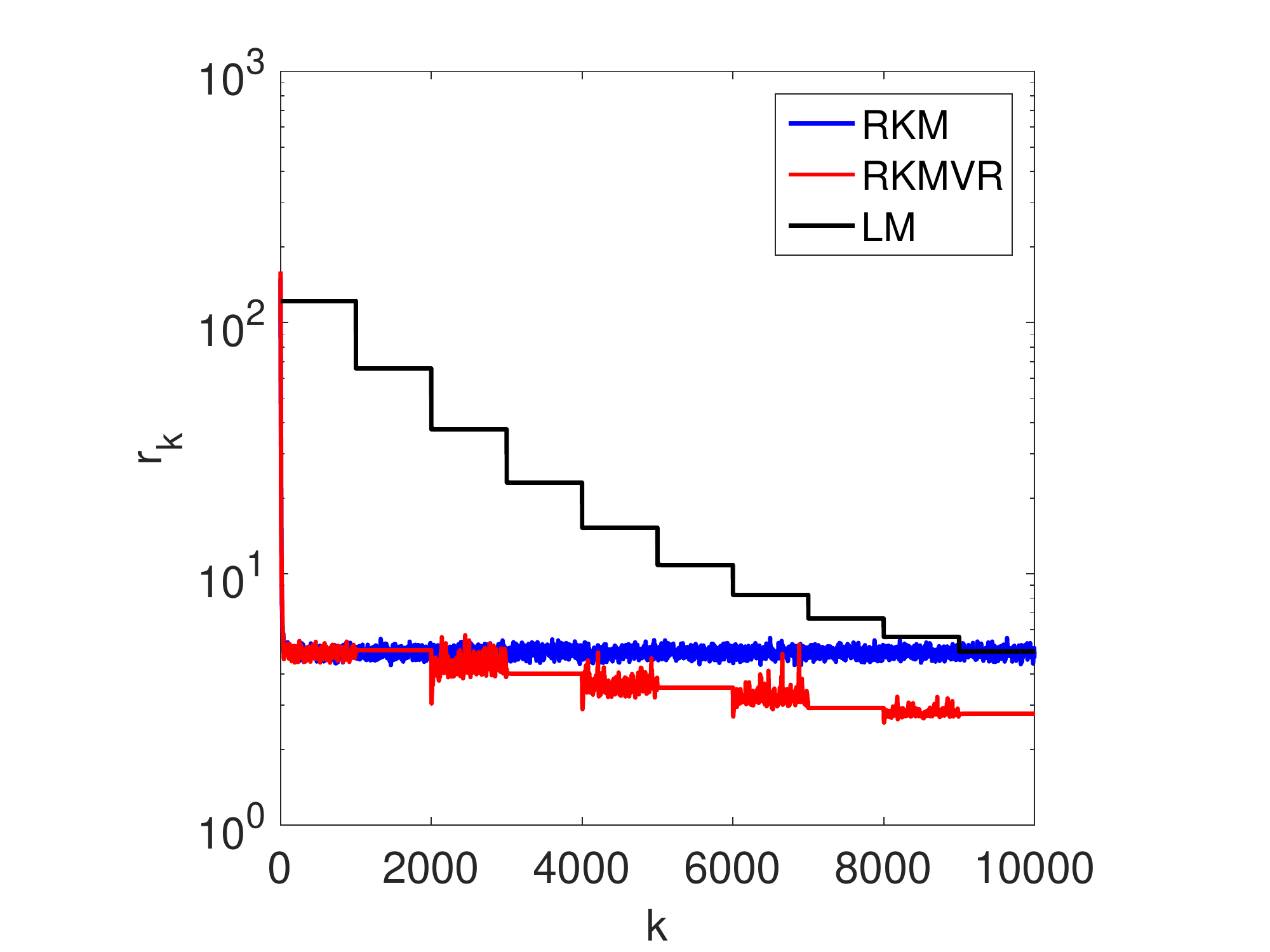}\\
   (a) \texttt{phillips}, $\delta=10^{-2}$ & (b) \texttt{phillips}, $\delta=5\times10^{-2}$ \\
   \includegraphics[trim={2cm 0 2cm 0},clip,width=.25\textwidth]{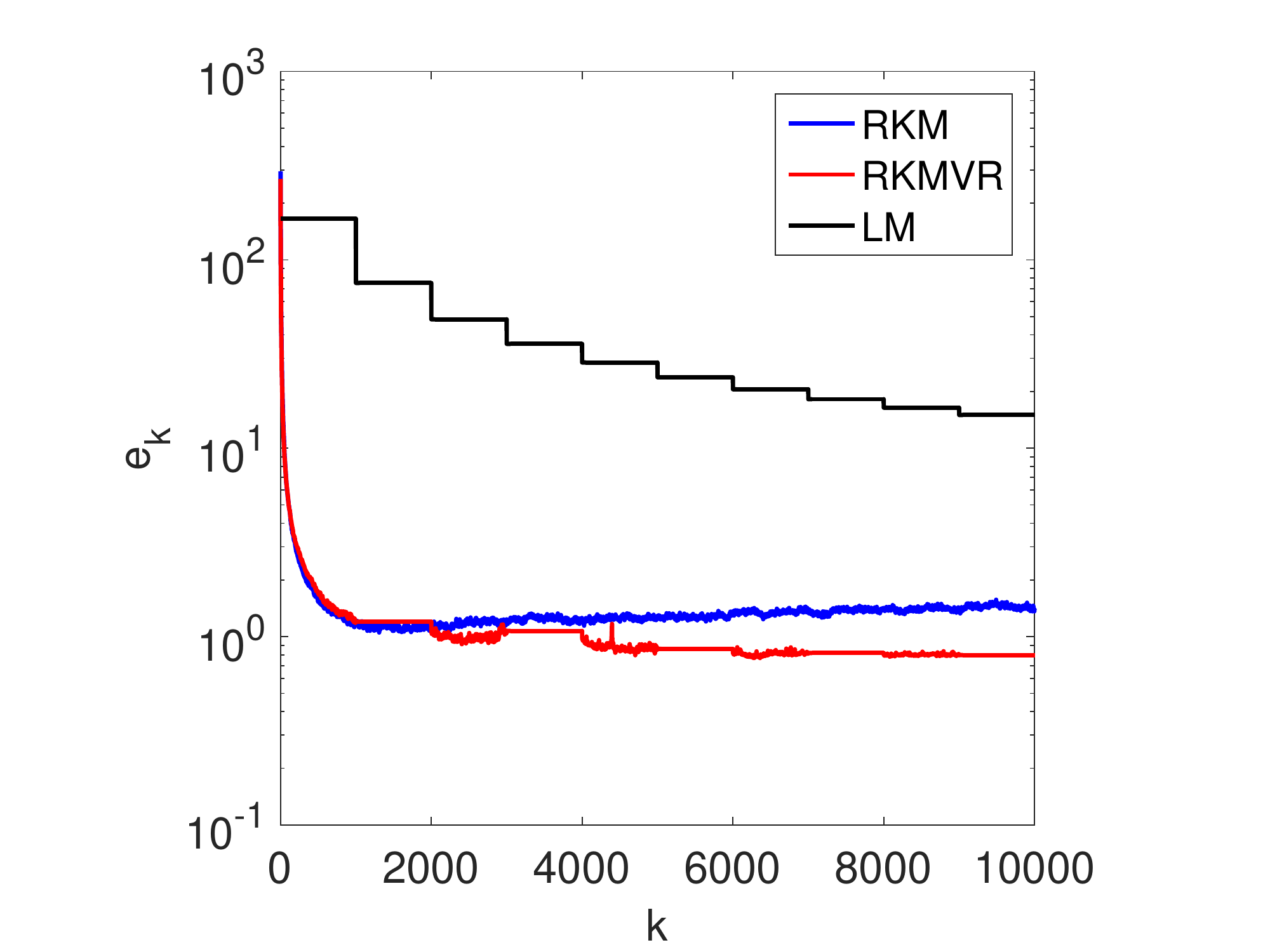} \includegraphics[trim={2cm 0 2cm 0},clip,width=.25\textwidth]{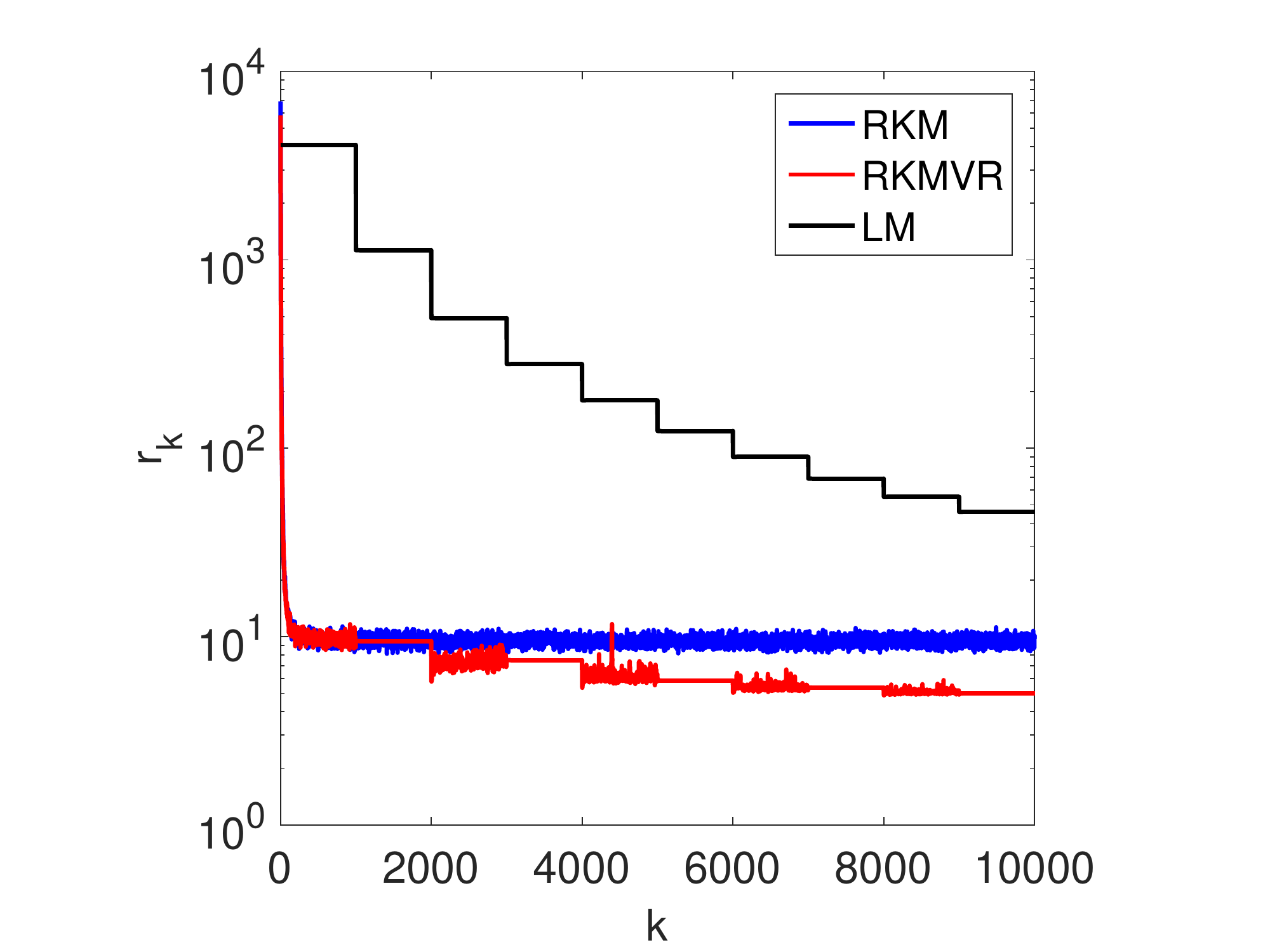}& \includegraphics[trim={2cm 0 2cm 0},clip,width=.25\textwidth]{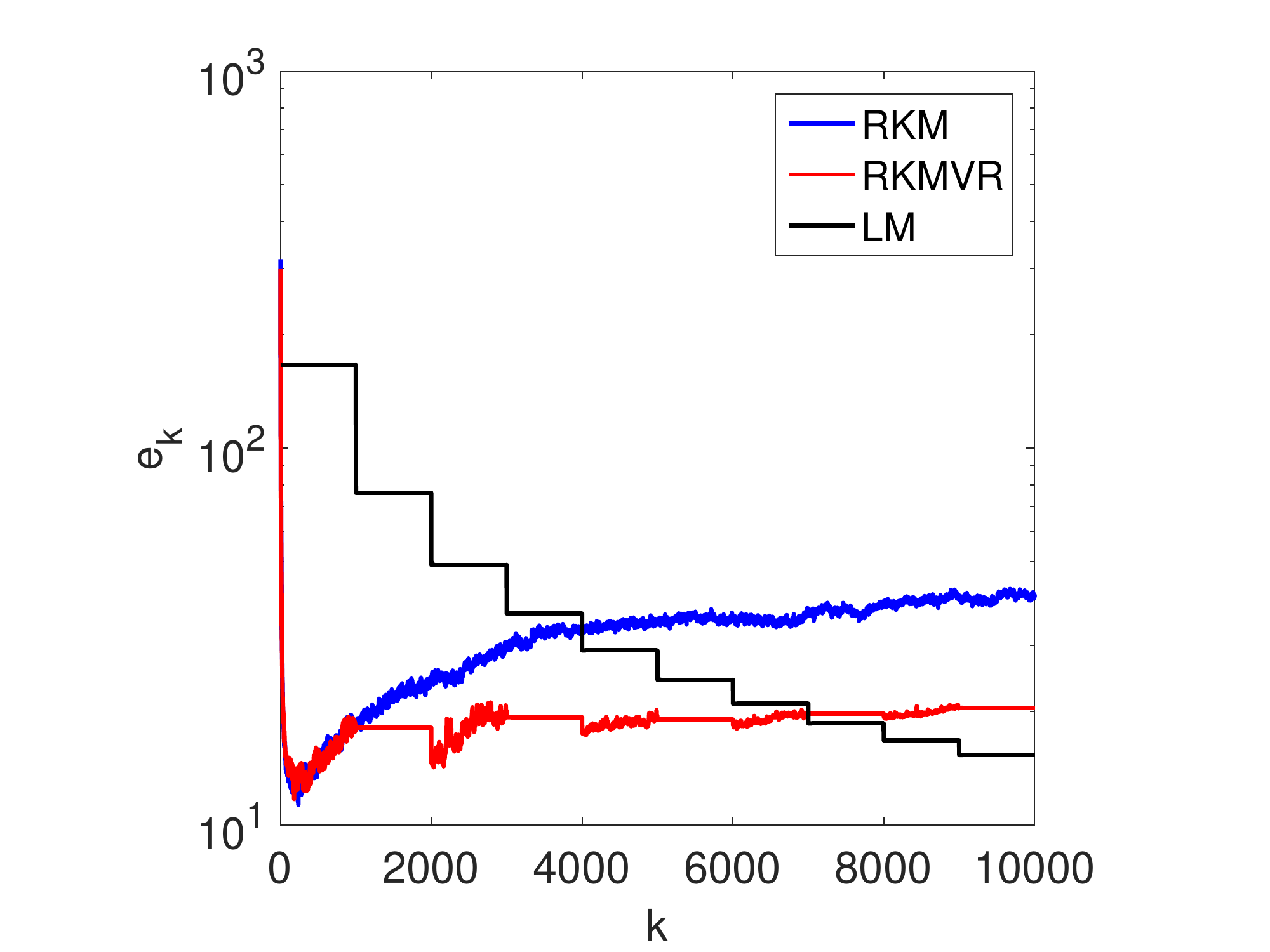} \includegraphics[trim={2cm 0 2cm 0},clip,width=.25\textwidth]{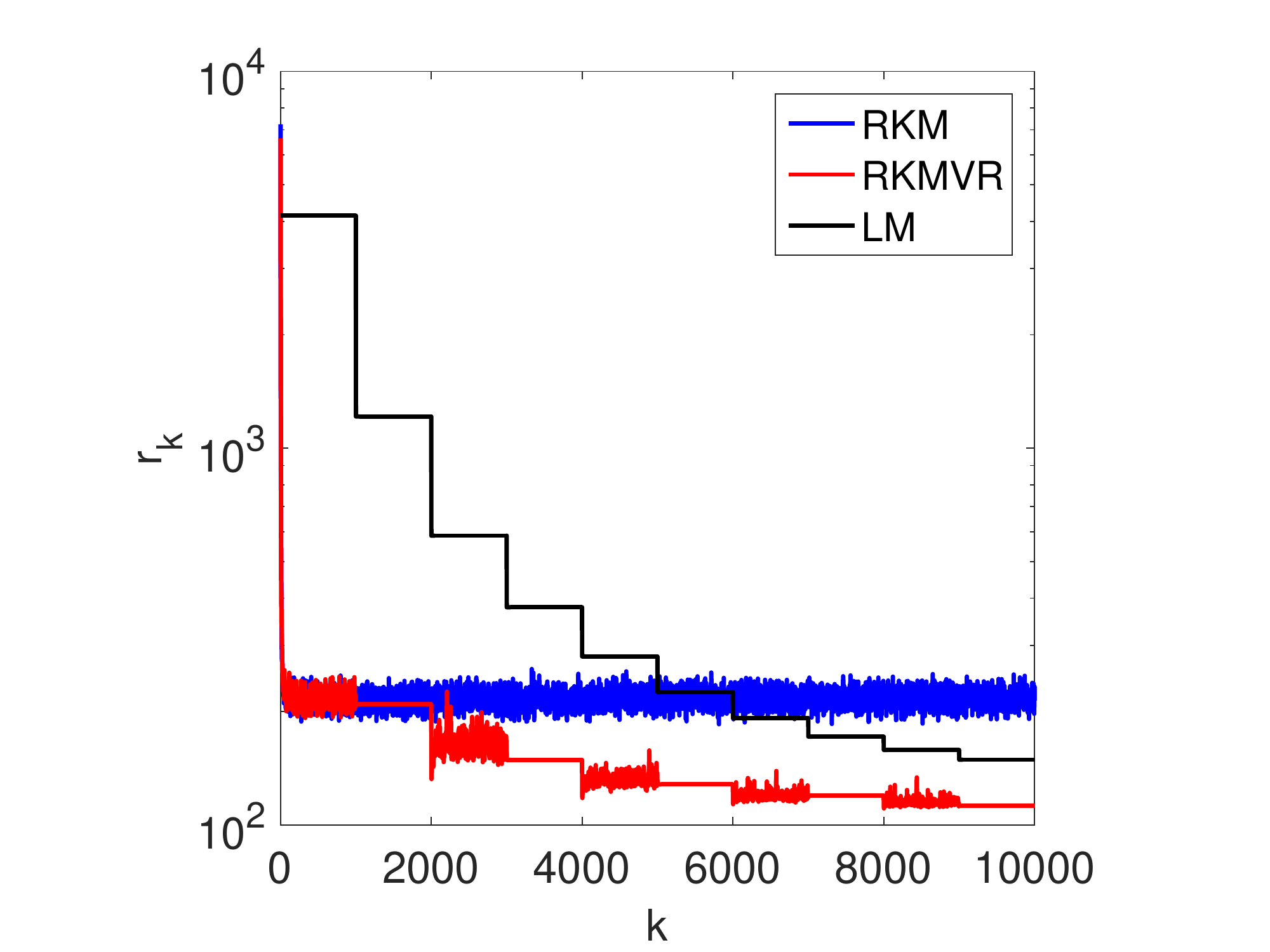}\\
   (c) \texttt{gravity}, $\delta=10^{-2}$ & (d) \texttt{gravity}, $\delta=5\times10^{-2}$  \\
   \includegraphics[trim={2cm 0 2cm 0},clip,width=.25\textwidth]{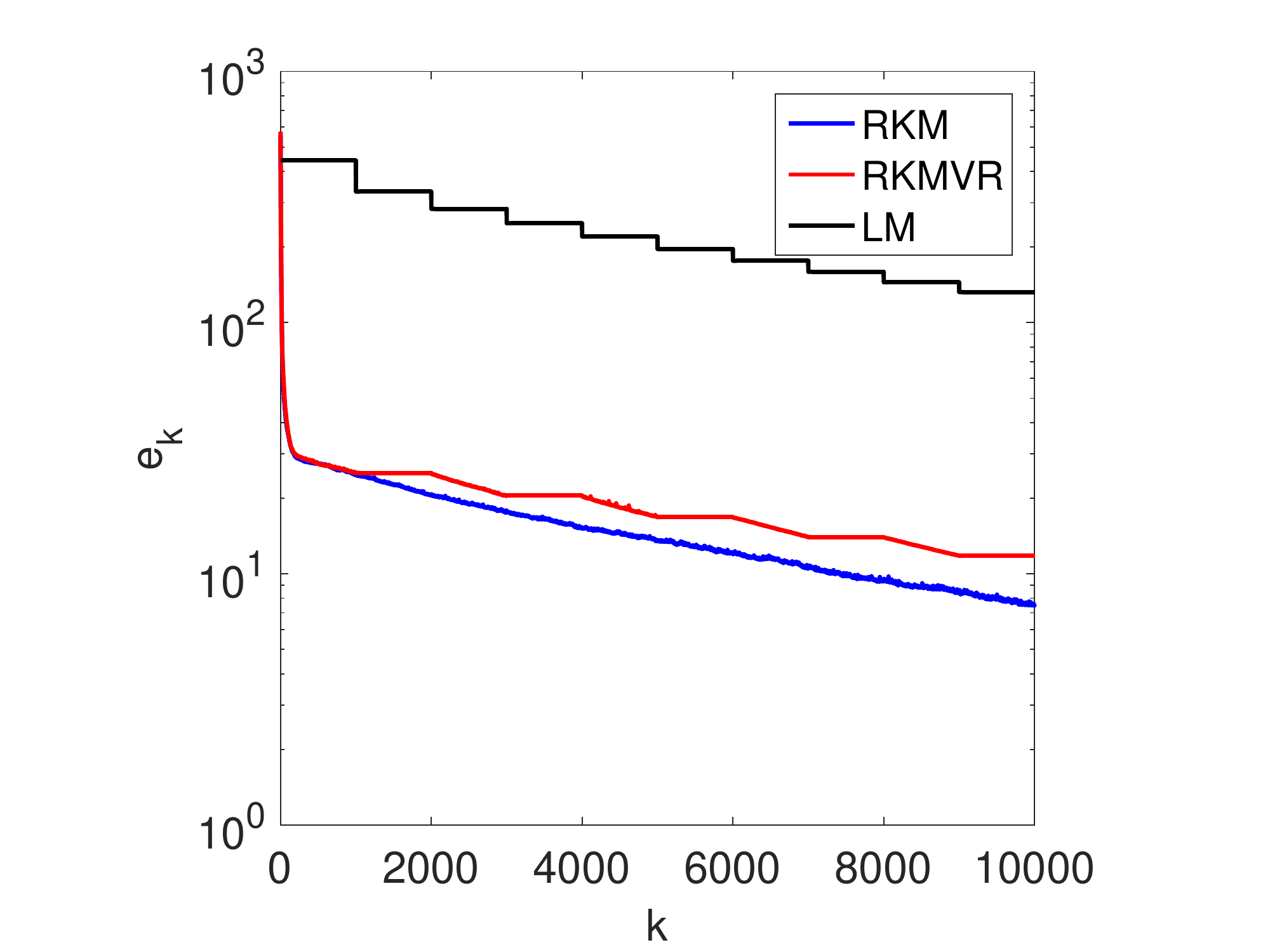} \includegraphics[trim={2cm 0 2cm 0},clip,width=.25\textwidth]{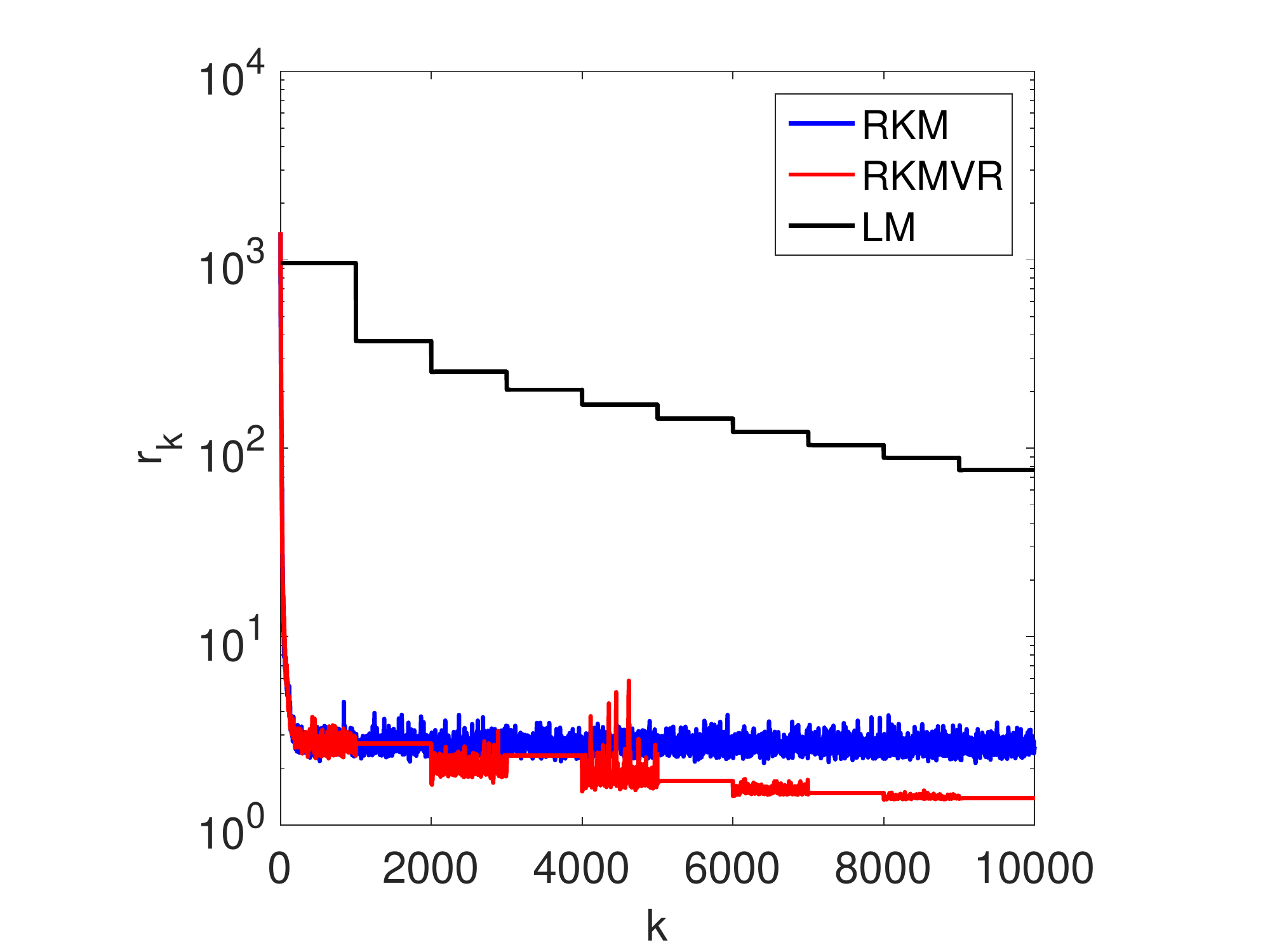} & \includegraphics[trim={2cm 0 2cm 0},clip,width=.25\textwidth]{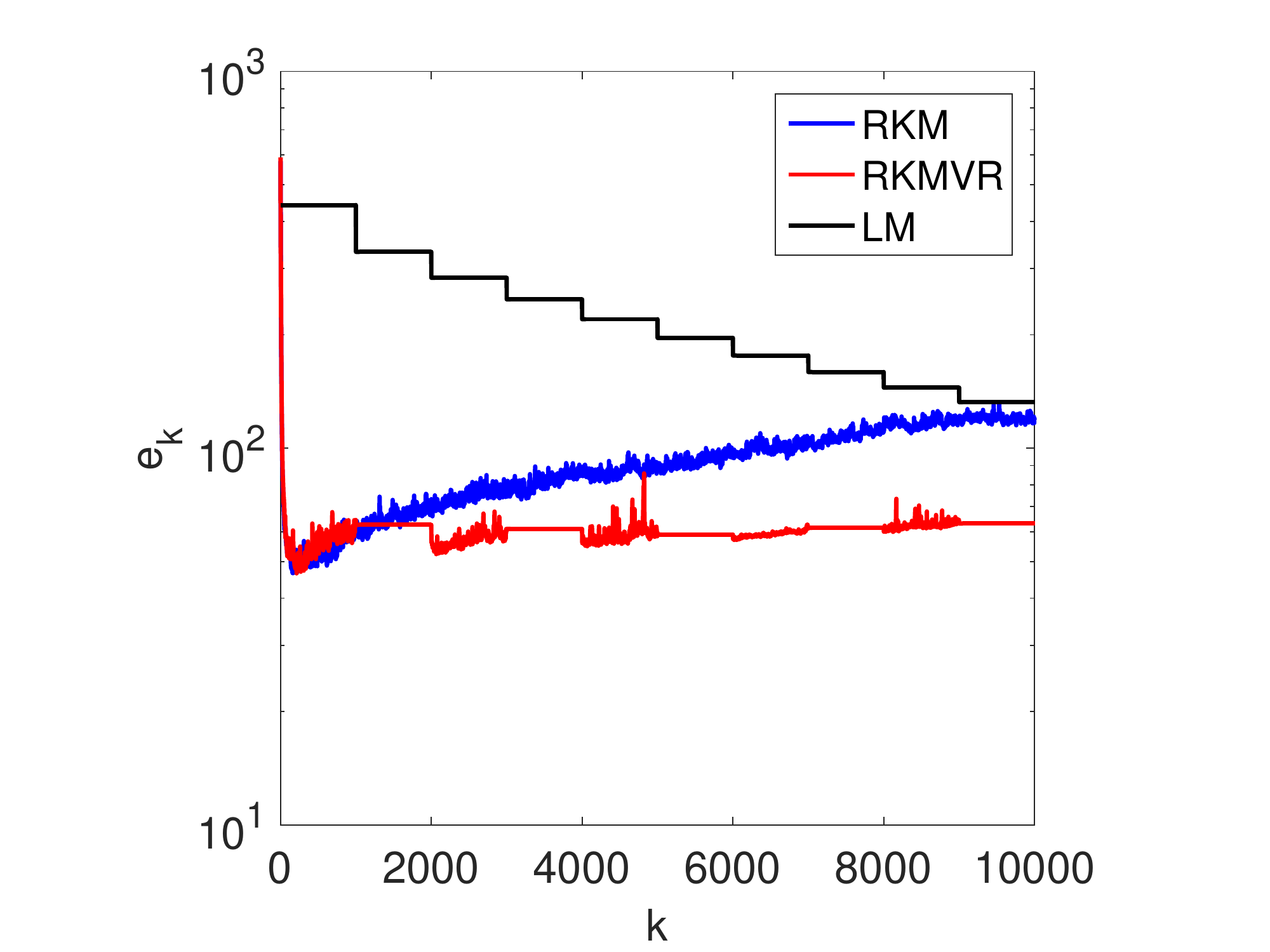} \includegraphics[trim={2cm 0 2cm 0},clip,width=.25\textwidth]{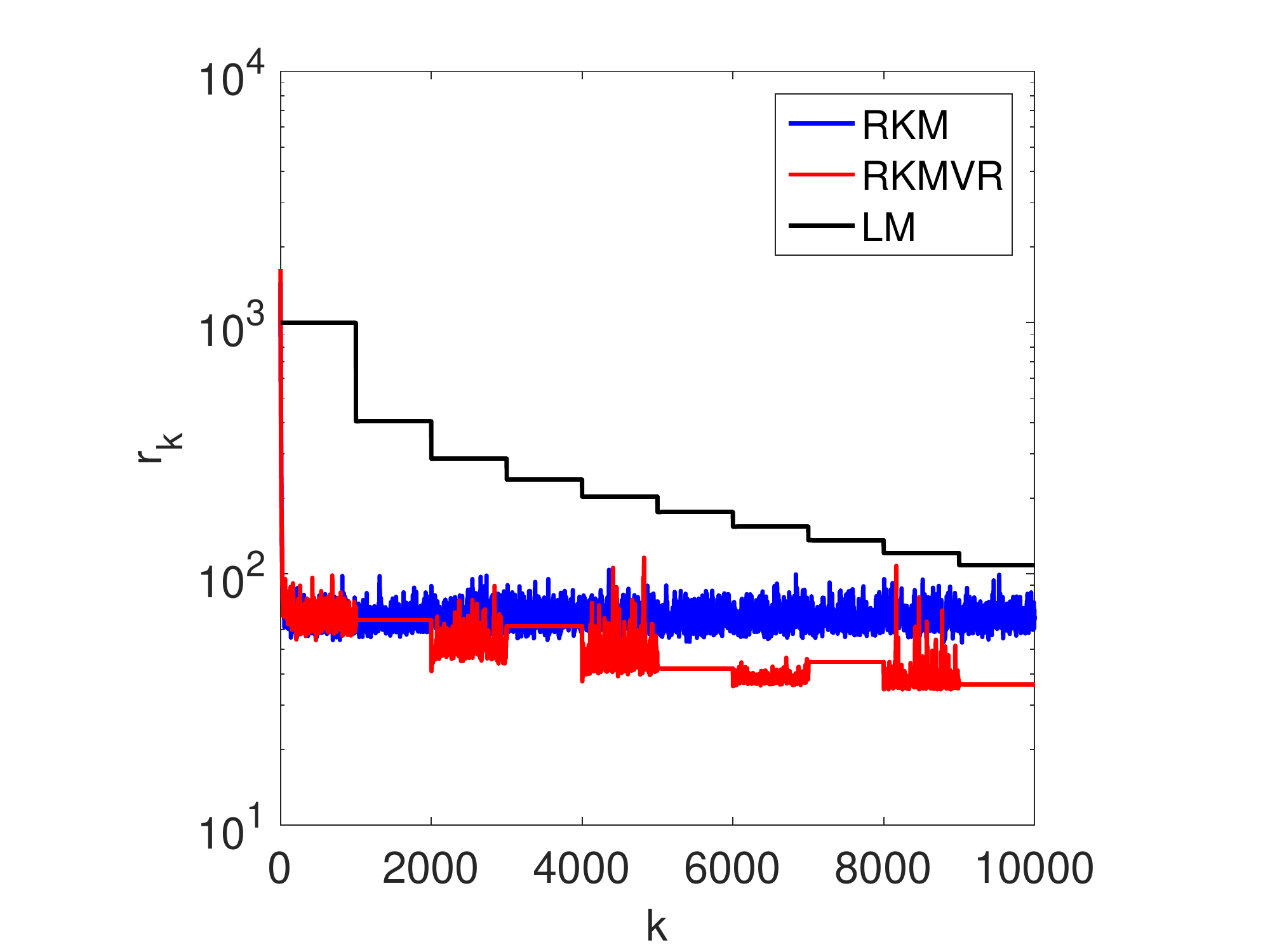}\\
   (e) \texttt{shaw}, $\delta=10^{-2}$ & (f) \texttt{shaw}, $\delta=5\times10^{-2}$
  \end{tabular}
  \caption{Numerical results for the examples by RKM, RKMVR and LM.\label{fig:rkmvr}}
\end{figure}

\begin{figure}[hbt!]
  \centering
  \setlength{\tabcolsep}{0pt}
  \begin{tabular}{cccc}
   \includegraphics[trim={2cm 0 2cm 0},clip,width=.25\textwidth]{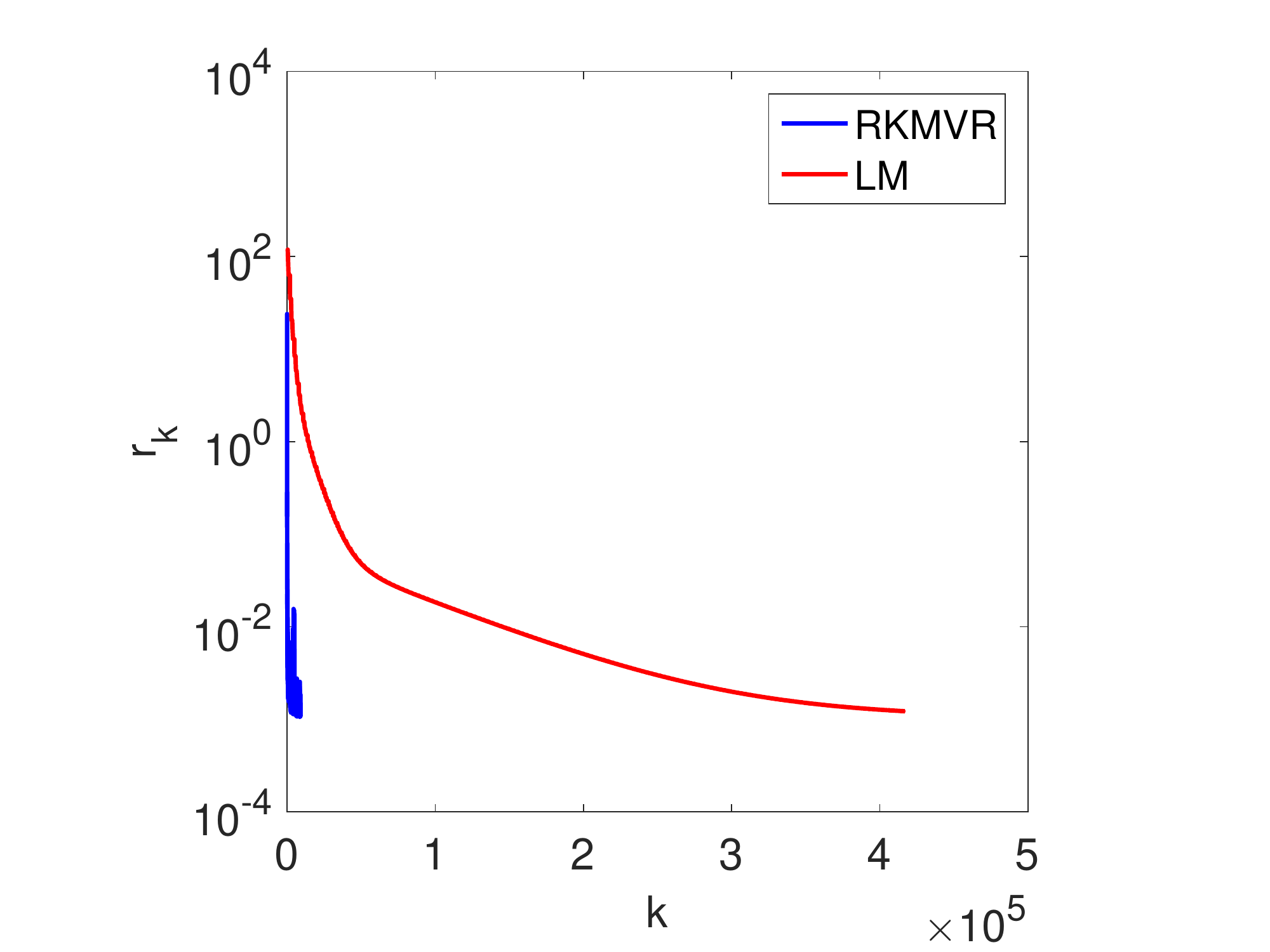} & \includegraphics[trim={2cm 0 2cm 0},clip,width=.25\textwidth]{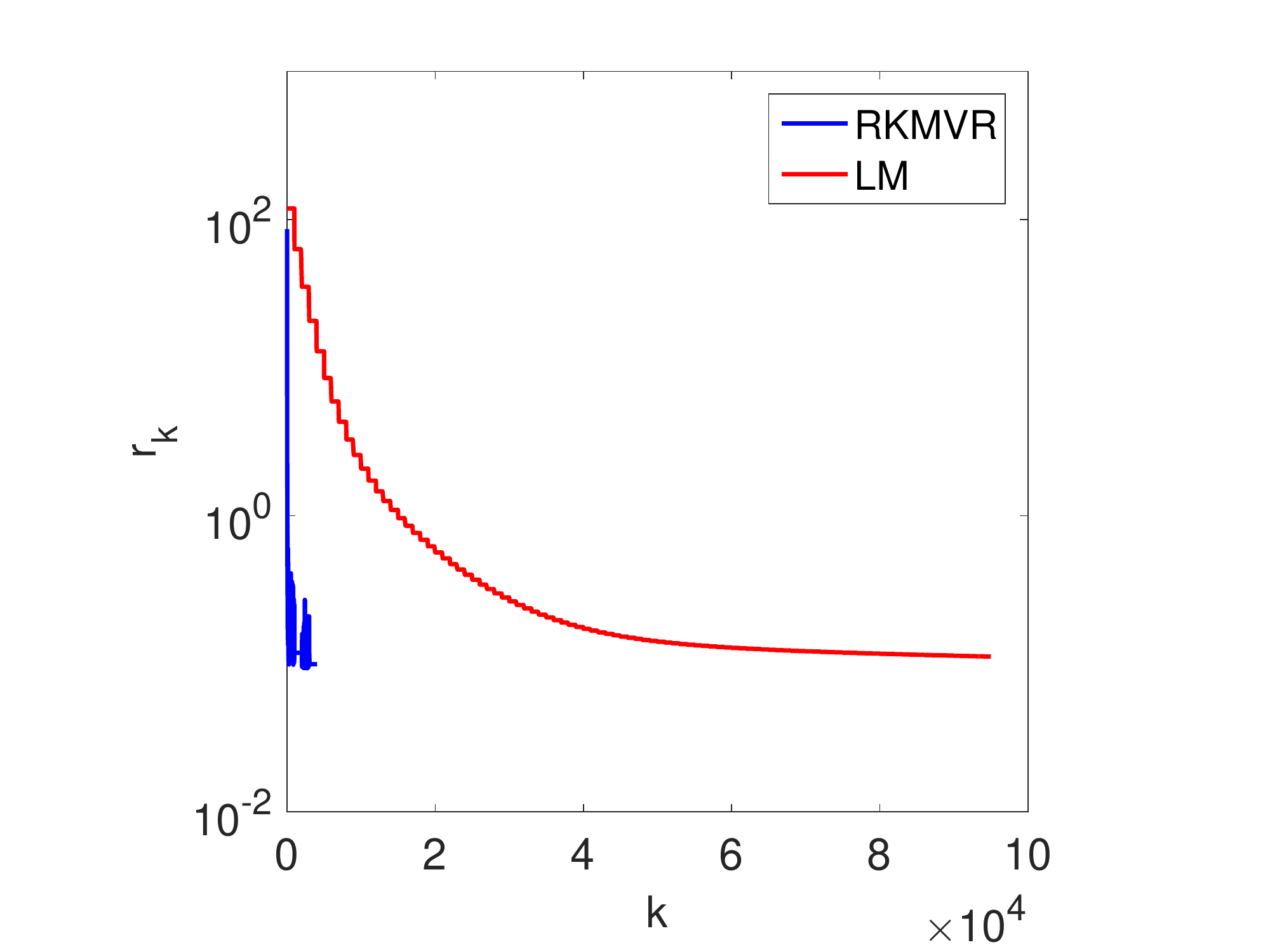}& \includegraphics[trim={2cm 0 2cm 0},clip,width=.25\textwidth]{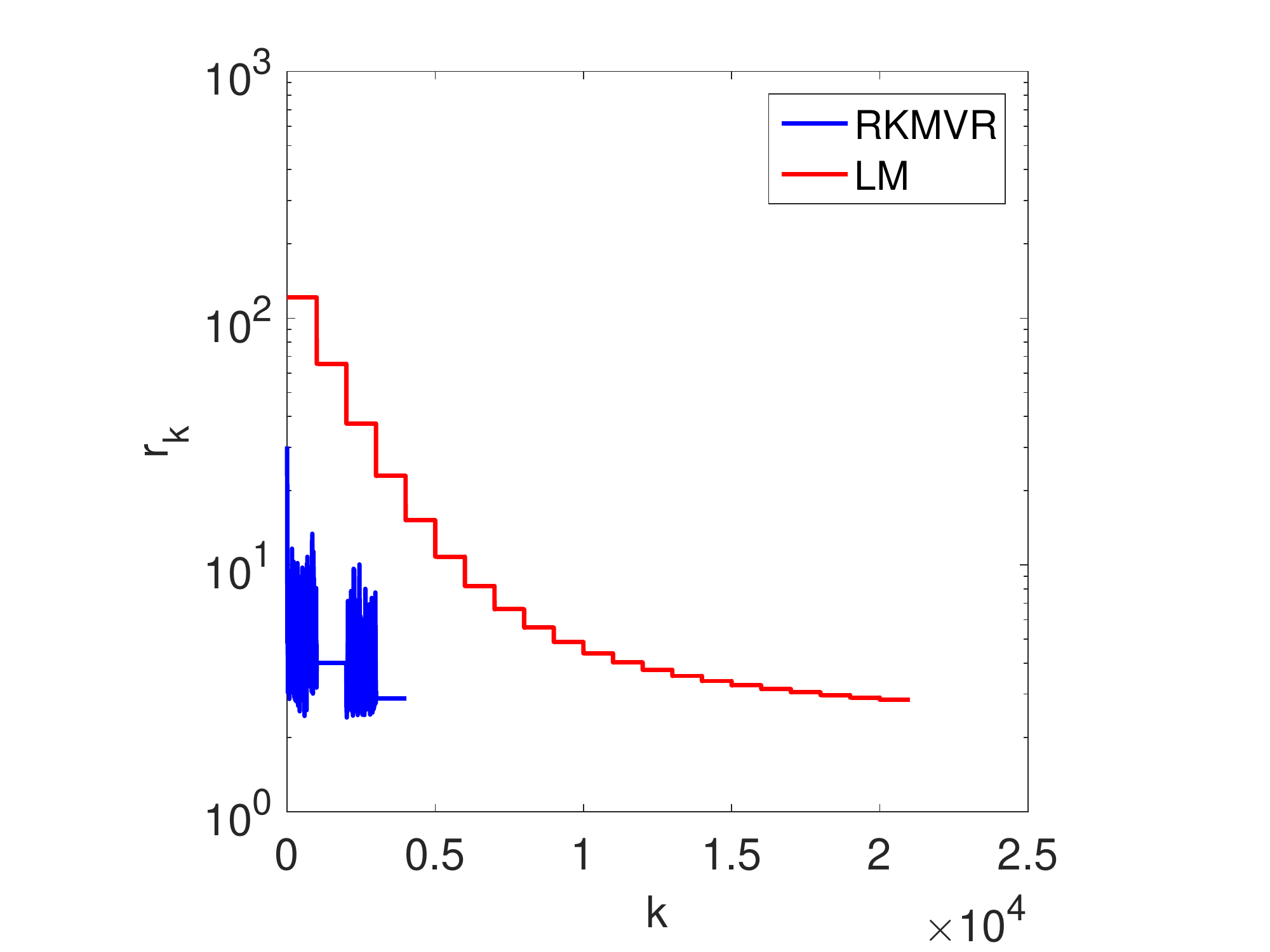} & \includegraphics[trim={2cm 0 2cm 0},clip,width=.25\textwidth]{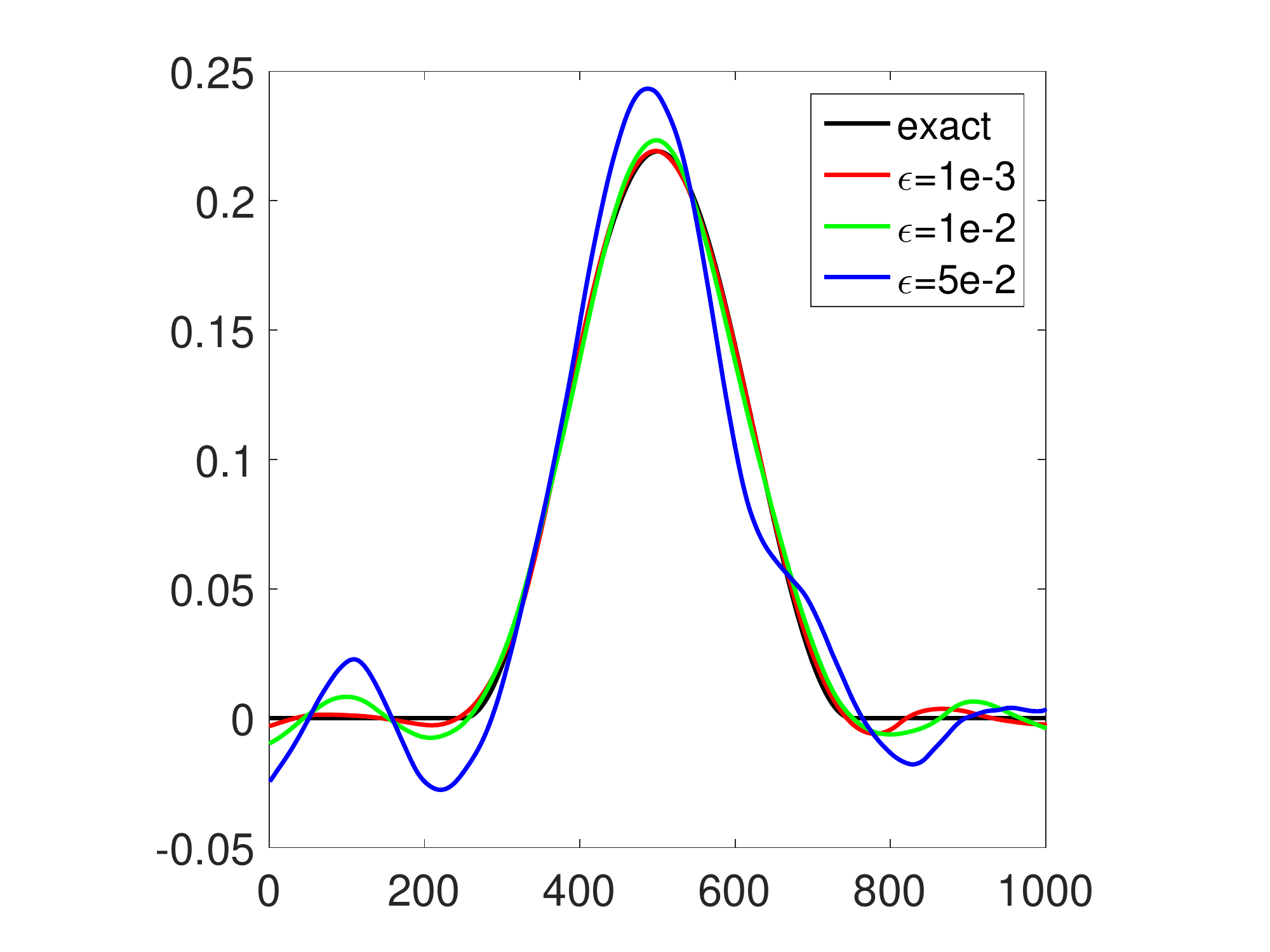}\\
   \includegraphics[trim={2cm 0 2cm 0},clip,width=.25\textwidth]{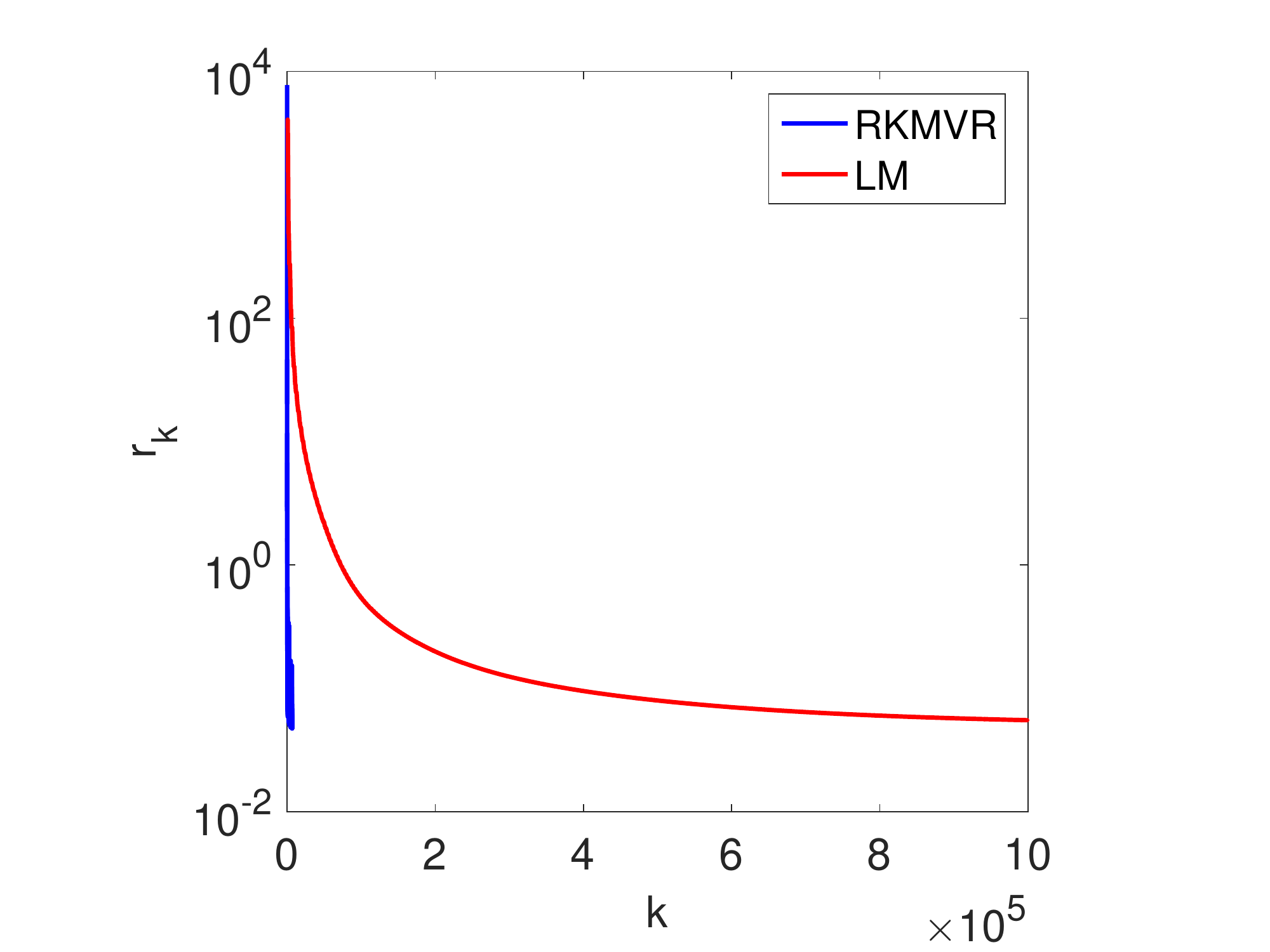} & \includegraphics[trim={2cm 0 2cm 0},clip,width=.25\textwidth]{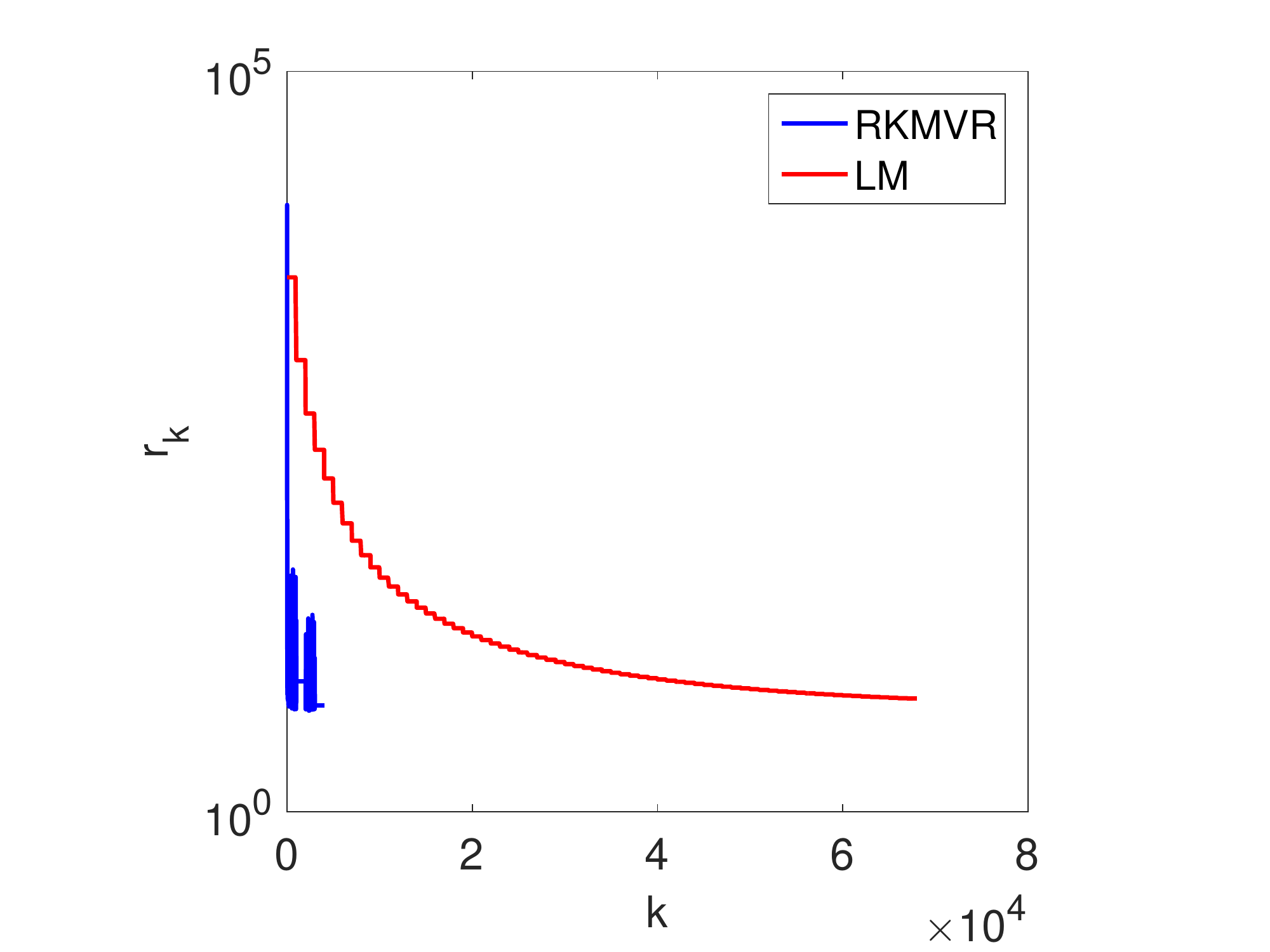}& \includegraphics[trim={2cm 0 2cm 0},clip,width=.25\textwidth]{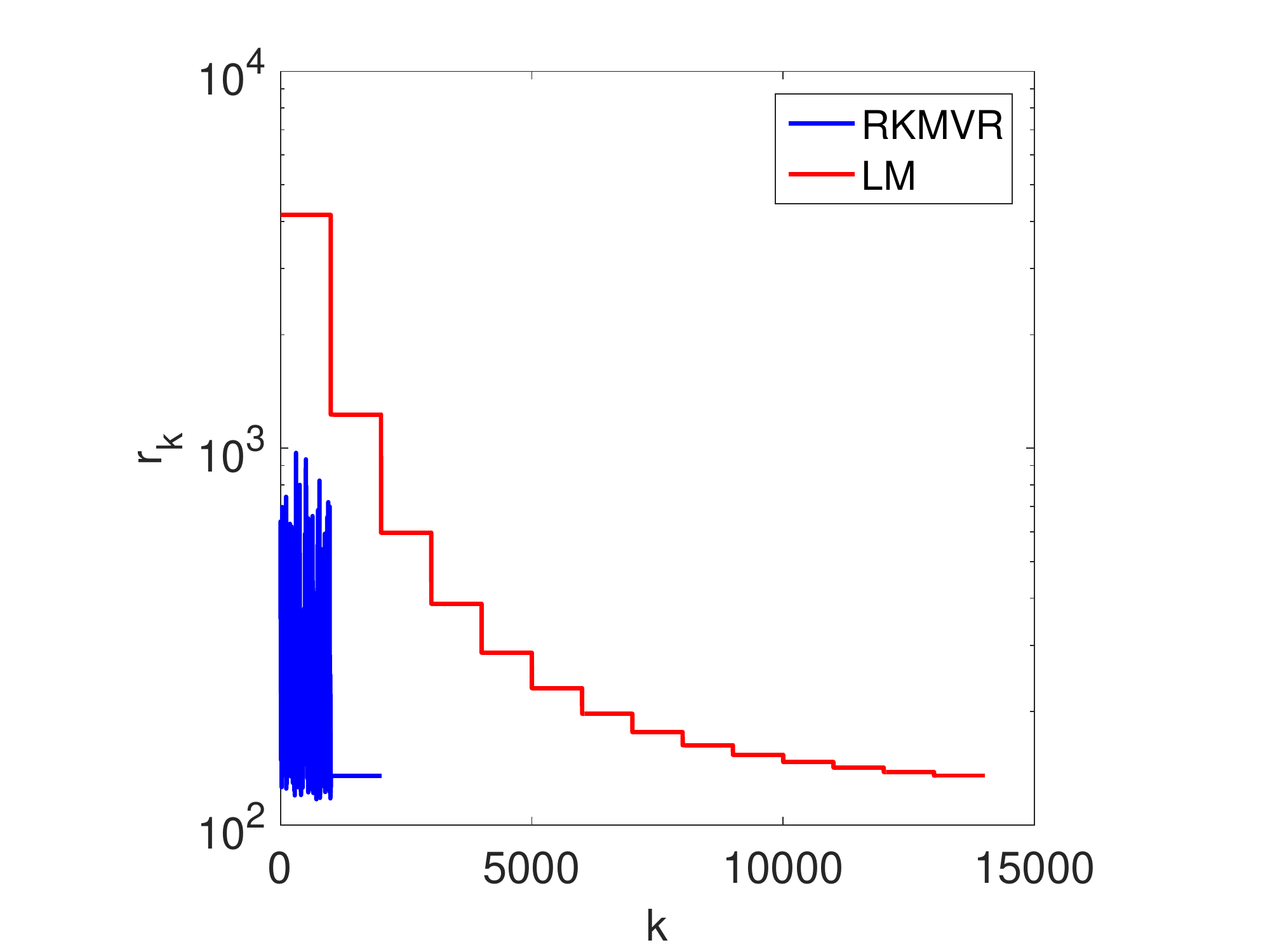} & \includegraphics[trim={2cm 0 2cm 0},clip,width=.25\textwidth]{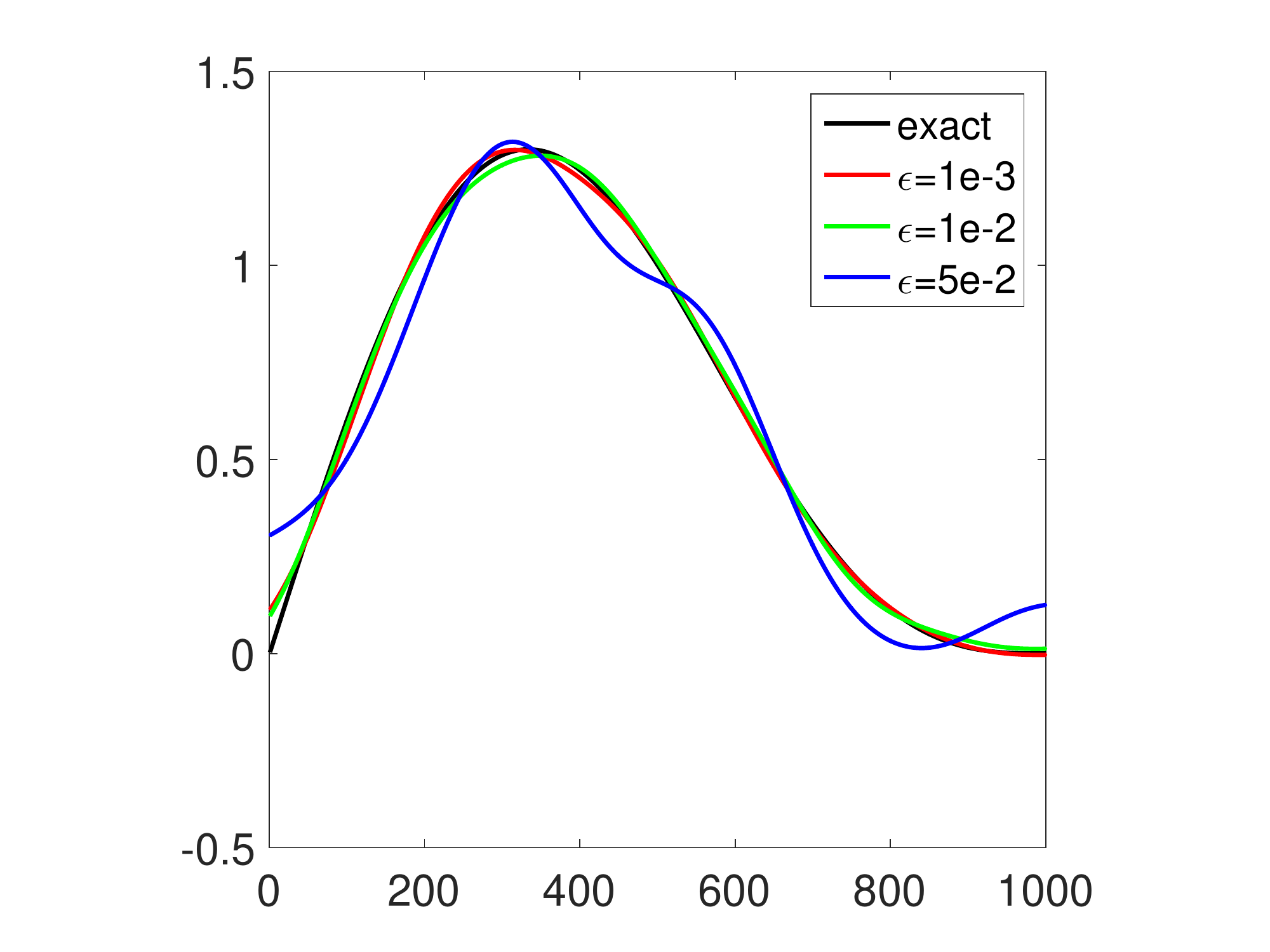}\\
  \includegraphics[trim={2cm 0 2cm 0},clip,width=.25\textwidth]{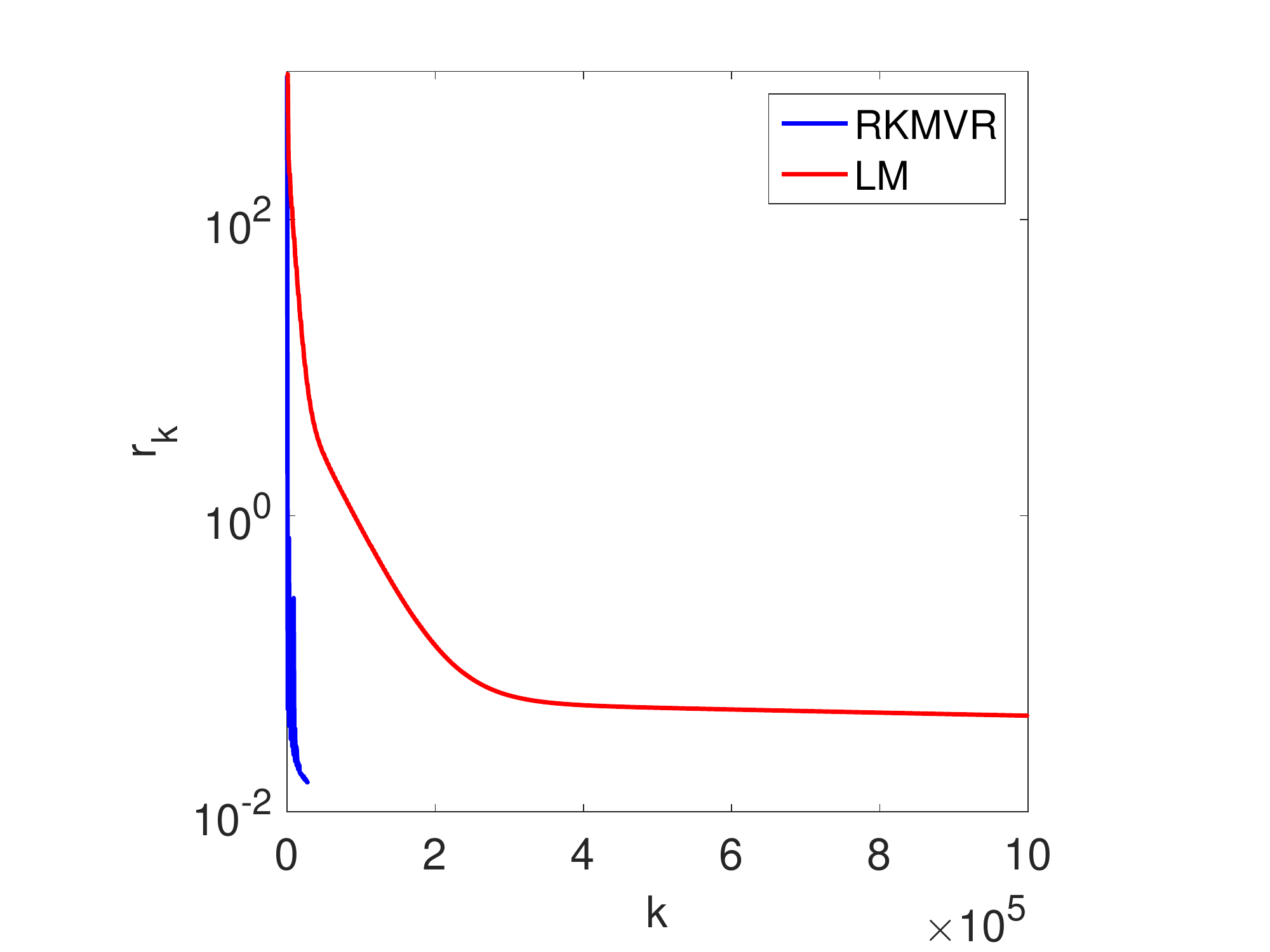} & \includegraphics[trim={2cm 0 2cm 0},clip,width=.25\textwidth]{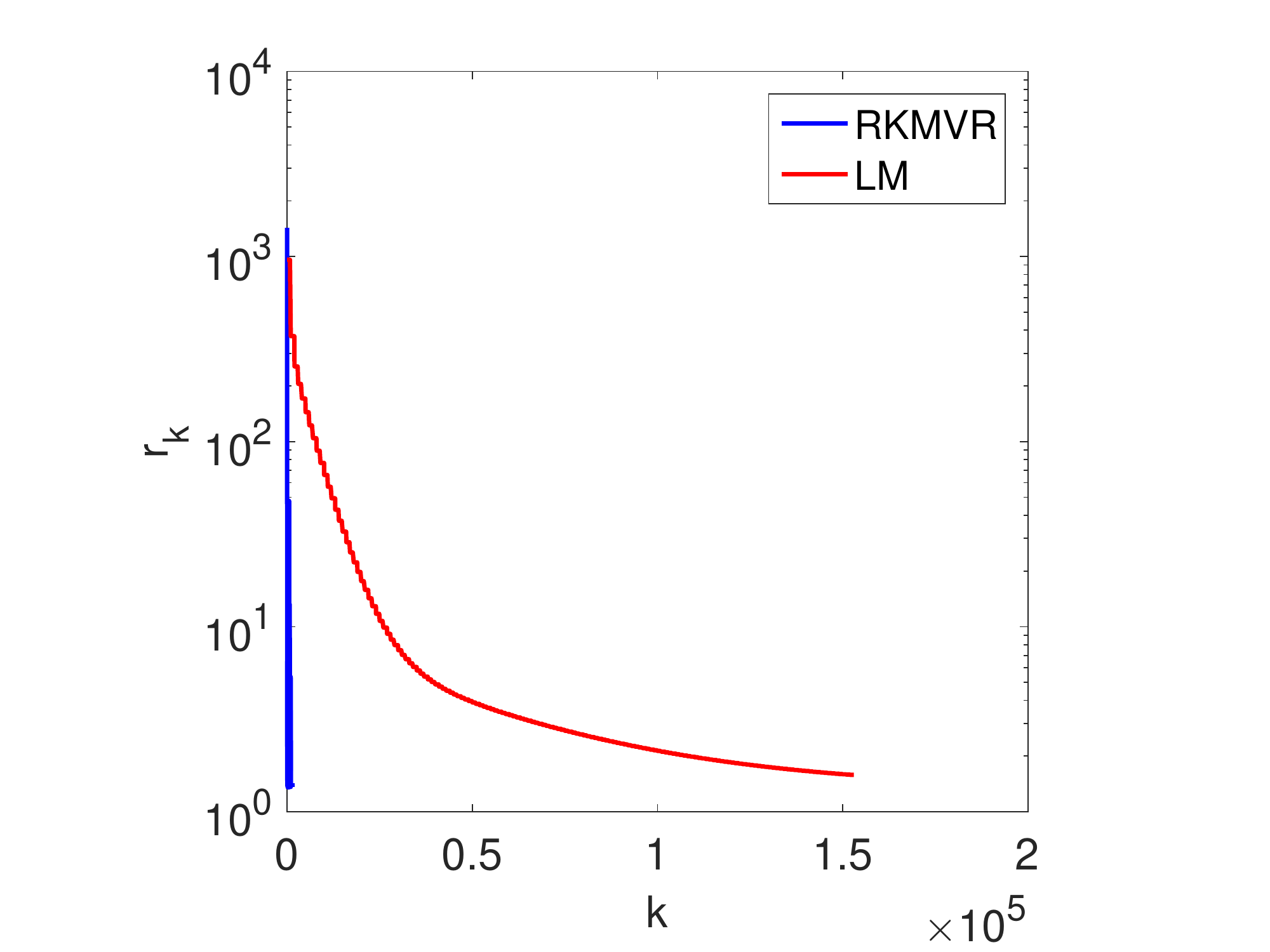}& \includegraphics[trim={2cm 0 2cm 0},clip,width=.25\textwidth]{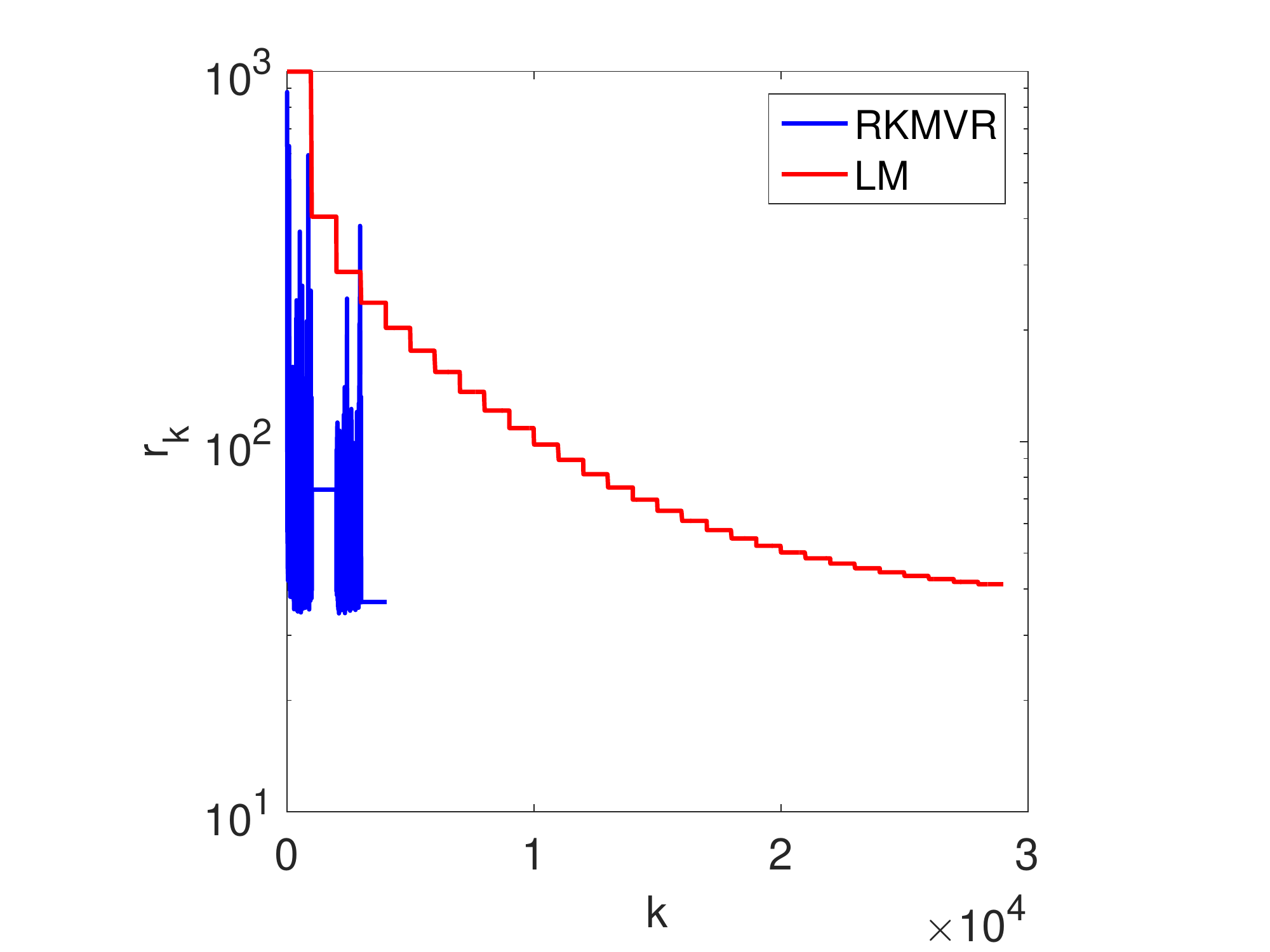} & \includegraphics[trim={2cm 0 2cm 0},clip,width=.25\textwidth]{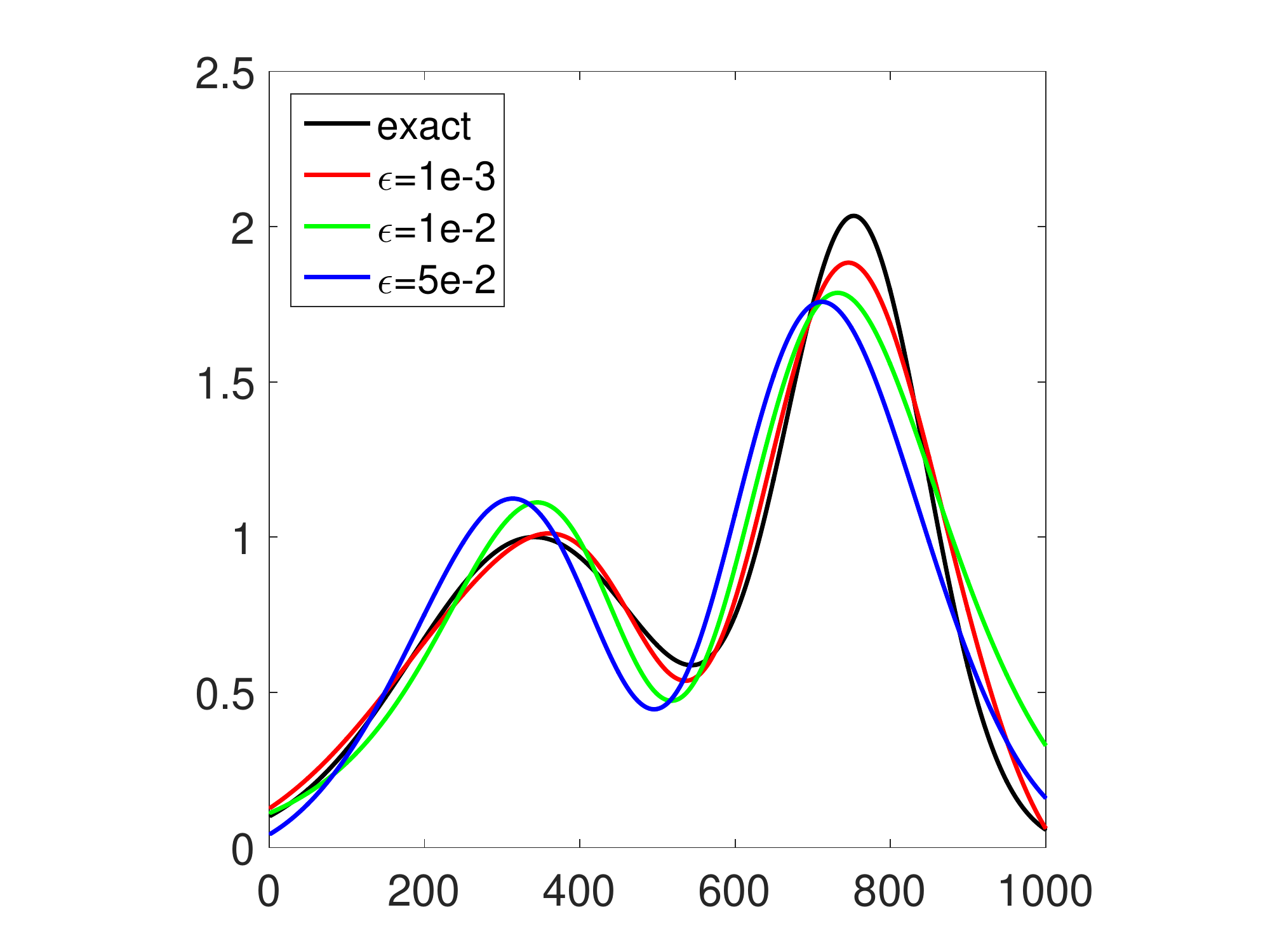}\\
  (a) $\delta=10^{-3}$ & (b) $\delta=10^{-2}$ & (c) $\delta=5\times10^{-2}$ & (d) solutions\\
  \end{tabular}
  \caption{The residual $r_k$ and the recoveries for \texttt{phillips} (top), \texttt{gravity} (middle), \texttt{shaw} (bottom) by RKMVR and
LM with the discrepancy principle \eqref{eqn:dp} with $\tau=1.1$.\label{fig:rkmvrdp}}
\end{figure}

\section{Conclusions}
We have presented an analysis of the preasymptotic convergence behavior of the randomized Kaczmarz method.
Our analysis indicates that the low-frequency error decays much faster than the high-frequency
one during the initial randomized Kaczmarz iterations. Thus, when the low-frequency modes are dominating
in the initial error, as typically occurs for inverse problems, the method enjoys very fast initial error
reduction. Thus this result sheds insights into the excellent practical performance of the method, which
is also numerically confirmed. Next, by interpreting it as a stochastic gradient method, we proposed a
randomized Kaczmarz method with variance reduction by hybridizing it with the Landweber method. Our numerical
experiments indicate that the strategy is very effective in that it can combine the strengths of both randomized
Kaczmarz method and Landweber method.

Our work represents only a first step towards a complete theoretical understanding of the randomized Kaczmarz
method and related stochastic gradient methods (e.g., variable step size, and mini-batch version) for efficiently solving
inverse problems. There are many important theoretical and practical questions awaiting further research.
Theoretically, one outstanding issue is the regularizing property (e.g., consistency, stopping criterion and
convergence rates) of the randomized Kaczmarz method from the perspective of classical regularization theory.

\section*{Acknowledgements}
The authors are grateful to the constructive comments of the anonymous referees, which have helped improve
the quality of the paper. In particular, the remark by one of the referees has led to much improved results as well as more concise proofs.
The research of Y. Jiao is partially supported by National Science Foundation of  China (NSFC) No. 11501579 and
National Science Foundation of  Hubei Province No. 2016CFB486, B. Jin by EPSRC grant EP/M025160/1 and UCL
Global Engagement grant (2016--2017), and X. Lu by NSFC Nos. 11471253 and 91630313.

\bibliographystyle{abbrv}
\bibliography{rkm}

\end{document}